\documentclass[10pt]{amsart}

\pdfoutput=1 

\usepackage[utf8]{inputenc}
\usepackage[T1]{fontenc}
\usepackage[backend=biber,
isbn=false,
doi=false,
maxbibnames=5,
giveninits=true,
style=alphabetic,
citestyle=alphabetic]{biblatex}
\usepackage{url}
\renewbibmacro{in:}{}
\bibliography{clifford.bib}

\numberwithin{equation}{section}

\usepackage{xcolor}
\definecolor{e-mail}{rgb}{0,.40,.80}
\definecolor{reference}{rgb}{0,.40,.80}
\definecolor{citation}{rgb}{0,.40,.80}

\PassOptionsToPackage{hyphens}{url}
\usepackage[colorlinks=true,
linkcolor=reference,
citecolor=citation,
urlcolor=e-mail]{hyperref}
\usepackage[all, 2cell]{xy}
\UseTwocells
\usepackage{mathtools}
\usepackage{tikz-cd}
\usepackage{amssymb}
\usepackage{eucal}
\usepackage[capitalize]{cleveref}
\usepackage{fullpage}

\usepackage{relsize}

\newtheorem{thm}{Theorem}[section]
\crefname{thm}{Theorem}{theorems}

\newtheorem{prop}[thm]{Proposition}
\crefname{prop}{Proposition}{propositions}

\newtheorem{cor}[thm]{Corollary}
\newtheorem{lm}[thm]{Lemma}
\newtheorem{conjecture}[thm]{Conjecture}
\theoremstyle{definition}
\newtheorem{defn}[thm]{Definition}

\newtheorem{notation}[thm]{Notation}
\theoremstyle{remark}
\newtheorem{remark}[thm]{Remark}
\newtheorem{example}[thm]{Example}

\newcommand{\cA}{\mathcal{A}}
\newcommand{\cB}{\mathcal{B}}

\newcommand{\cM}{\mathcal{M}}

\newcommand{\cO}{\mathcal{O}}

\newcommand{\cR}{\mathcal{R}}

\newcommand{\Det}{\mathrm{det}}
\newcommand{\bKZ}{\mathrm{bKZ}}
\newcommand{\KZ}{\mathrm{KZ}}
\newcommand{\bCas}{\mathrm{bCas}}

\newcommand{\B}{\mathrm{B}}

\newcommand{\C}{\mathbb{C}}

\newcommand{\U}{\mathrm{U}}
\newcommand{\Z}{\mathbb{Z}}

\newcommand{\rY}{\mathrm{Y}}
\newcommand{\GL}{\mathrm{GL}}

\newcommand{\tw}{\mathrm{tw}}

\newcommand{\fS}{\mathfrak{S}}
\newcommand{\fP}{\mathfrak{P}}
\newcommand{\fR}{\mathfrak{R}}
\renewcommand{\sp}{\mathfrak{sp}}
\renewcommand{\b}{\mathfrak{b}}
\newcommand{\g}{\mathfrak{g}}
\newcommand{\h}{\mathfrak{h}}
\newcommand{\n}{\mathfrak{n}}
\renewcommand{\sl}{\mathfrak{sl}}

\renewcommand{\k}{\mathfrak{k}}
\newcommand{\so}{\mathfrak{so}}
\newcommand{\fo}{\mathfrak{o}}

\newcommand{\rO}{\mathrm{O}}
\newcommand{\rX}{\mathrm{X}}
\newcommand{\Cas}{\mathrm{Cas}}
\newcommand{\reg}{\mathrm{reg}}
\newcommand{\Span}{\mathrm{span}}

\newcommand{\Beta}{\mathrm{B}}
\newcommand{\act}{\mathrm{act}}

\newcommand{\End}{\mathrm{End}}
\newcommand{\ev}{\mathrm{ev}}

\newcommand{\gen}{\mathrm{gen}}

\newcommand{\Hom}{\mathrm{Hom}}

\newcommand{\id}{\mathrm{id}}

\newcommand{\SL}{\mathrm{SL}}

\newcommand{\wt}{\mathrm{wt}}

\newcommand{\gl}{\mathfrak{gl}}

\newcommand{\Cl}{\mathrm{Cl}}
\newcommand{\SO}{\mathrm{SO}}

\newcommand{\bl}{\mathbf{l}}

\newcommand{\Usig}{U(\mathfrak{g})[\sigma]}

\newcommand{\defterm}[1]{\textbf{\emph{#1}}}
\newcommand{\adj}[2]{
	\xymatrix{
		#1 \ar@<.5ex>[r] & #2 \ar@<.5ex>[l]
	}
}

\usepackage[foot]{amsaddr}

\makeatletter
\renewcommand{\email}[2][]{%
  \ifx\emails\@empty\relax\else{\g@addto@macro\emails{,\space}}\fi%
  \@ifnotempty{#1}{\g@addto@macro\emails{\textrm{(#1)}\space}}%
  \g@addto@macro\emails{#2}%
}
\makeatother

\begin{document}
	
	\title{Orthogonal Howe duality and dynamical (split) symmetric pairs}

    \author{Elijah Bodish}
    \address{Department of Mathematics, Indiana University Bloomington, Rawles Hall, Bloomington, IN, 47405, USA, ORCiD 0000-0003-1499-5136}
    \email{ebodish@iu.edu}

    \author{Artem Kalmykov}
    \address{Department of Mathematics, McGill University, Burnside Hall, Montreal, Quebec, H3A 0B9, Canada, ORCiD 0000-0002-0139-6831}
    \email{artem.kalmykov@mcgill.ca}
 
	\begin{abstract}
	   Inspired by Etingof--Varchenko's dynamical Weyl group for Lie algebras and Reshetikhin-Stokman's boundary fusion operators for split symmetric pairs, we introduce a dynamical Weyl group for split symmetric pairs. We then turn to the study of $(\so_{2n},\rO_m)$-duality and prove that the standard Knizhnik-Zamolodchikov and dynamical operators (both differential and difference) on the $\so_{2n}$-side are exchanged with the symmetric pair analogs, for $\rO_m\subset \GL_m$, on the $\rO_m$-side. 
	\end{abstract}
\maketitle
	
	\section{Introduction}
	
	Schur--Weyl duality states that the commuting actions of the general linear group $\GL_n$ and the symmetric group $S_m$ on the tensor power $(\C^n)^{\otimes m}$ generate each other's commutant. Replacing $(\C^n)^{\otimes m}$ with $(\Lambda^{\bullet}(\C^n))^{\otimes m}$, there are commuting actions of $\GL_n$ and $\GL_m$ on $\fP_{nm} = \Lambda^{\bullet}(\C^n \otimes \C^m)$  such that each group generates the other's commutant. This is known as skew Howe $(\gl_n,\gl_m)$-duality \cite{Howe-perspectives}. 
	
	Skew Howe duality exhibits \emph{bispectral duality}. Operators depending on spectral parameters on the $\gl_n$ side are exchanged with operators depending on the dynamical parameters (lying on the Cartan subalgebra) on the $\gl_m$ side\footnote{Strictly speaking, bispectrality usually means that the operators depend on both types of variables, and the duality exchanges them; in what follows, we will use the term in this weaker sense.}. For example, the quantum affine $R$-matrix for $\gl_n$ is exchanged with the dynamical quantum Weyl group for $\gl_m$ \cite{DalipiFelder}.
 
    There is also skew Howe $(\so_{2n}, \rO_m)$-duality \cite{Howe-perspectives}. The even-dimensional orthogonal Lie algebra $\so_{2n}$ has a spin representation, which can be identified with $\Lambda^{\bullet}(\C^n)$ as a vector space. On the other hand, the orthogonal group $\rO_m$ acts naturally on the vector space $\C^m$, hence on the exterior algebra $\Lambda^{\bullet}(\C^m)$. Together, the actions of $\so_{2n}$ and $\rO_m$ on $\Lambda^{\bullet}(\C^n \otimes \C^m)$ commute and generate each other's commutant.
	
	It turns out that, unlike in the case of $(\gl_n,\gl_m)$-duality, the two sides of orthogonal Howe duality are not symmetric. Rather, the $\rO_m$-side should be interpreted as the symmetric pair $\GL_m^{\theta}\subset \GL_m$, where $\theta$ is the involution on $\GL_m$ such that $\rO_m=\GL_m^{\theta}$. This is a known phenomenon, especially manifest in the $q$-version of the orthogonal Howe duality, where replacing $\so_{2n}$ with the Drinfeld--Jimbo quantum group $\U_q(\so_{2n})$ forces the group $\rO_m$ to be replaced by an $\iota$quantum group \cite{wenzl2020dualities}.

    To generalize bispectral duality from $(\gl_n,\gl_m)$ to $(\so_{2n}, \rO_m)$, we first construct a dynamical Weyl group for a split\footnote{The associated Satake diagram has all white nodes and no folding.} symmetric pair $\k\subset \g$. Working over $\C$, this means $\g$ is simple and $\k$ is the fixed points of the Chevalley involution swapping generators $e_{\alpha_i}\leftrightarrow -f_{\alpha_i}$ and acting by $-1$ on $\h$. 
	
	\begin{remark}
		We \textbf{do not} use $\iota$quantum groups in this work. Although we will occasionally discuss $\iota$quantum groups to motivate certain constructions.
	\end{remark}

    \begin{remark}
        In the interest of clarity we choose to not refer to the constructions in this paper as ``coideal", ``quantum symmetric pair", or "$\iota$quantum". By analogy with integrable systems, we will instead indicate the corresponding constructions with the adjective \textbf{``boundary''} \cite{CherednikQKZ}. For consistency, we will also refer to the ``restricted" or ``relative" Weyl group of a symmetric pair \cite{Bump-Liegps-book} as the boundary Weyl group.
    \end{remark}
	
	\subsection{Boundary Weyl group} 
	
	We begin with the boundary (non-dynamical) Weyl group. Consider the real group $G= SL_2(\mathbb{R})$, the subgroup $T\subset G$ of diagonal matrices $\alpha^{\vee}(t):=\begin{bmatrix} t & 0\\ 
		0 & t^{-1}
	\end{bmatrix}$, for $t\in (0,\infty)$, and the subgroup $K=SO(2)\subset G$ consisting of rotation matrices $\rho(\theta):=\begin{bmatrix} \cos(\theta) & -\sin(\theta)\\ 
		\sin(\theta) & \cos(\theta)
	\end{bmatrix}$, for $\theta\in [0, 2\pi)$. The Weyl group is defined as $W=N_G(T)/T$, and the element $\tilde{s} = \begin{bmatrix} 0 & -1\\ 
		1 & 0
	\end{bmatrix}\in G$ maps to a generator $s\in W\cong \mathbb{Z}/2$. Note that $\tilde{s} = \rho(\pi/2)\in K$, so $\tilde{s}$ acts in representations of $K$. This is the boundary Weyl group action.
 
    To understand the boundary Weyl group action on representations of $K$ algebraically, the building block is the symmetric pair $\so_2\subset \sl_2$. Here $\so_2$ is the fixed points of the involution of $\sl_2$ which swaps $e$ with $-f$ and sends $h$ to $-h$. Note that $\so_2$ is spanned by the element $b:=f-e$. Writing $b|_{\C^2}$ to denote $b$ acting in $\C^2$, one finds that
	\[
	\exp \left( \theta\cdot b|_{\C^2} \right) = \begin{bmatrix}
		\cos(\theta) & -\sin(\theta)  \\ \sin(\theta) & \cos(\theta)
	\end{bmatrix} \qquad \text{and} \qquad \exp\left(\frac{\pi}{2}\cdot b|_{\C^2}\right) = \begin{bmatrix} 0 & -1 \\ 1 & 0 \end{bmatrix} = \tilde{s}.
	\]
    We define the \emph{boundary simple reflection} operator to be $\exp\left( \frac{\pi b}{2} \right)$.
	
    The Iwasawa decomposition says in particular that $\g$ contains a Borel subalgebra $\b$ such that $\g = \k\oplus \b$. A general split symmetric pair $\k\subset \g$ contains a pair isomorphic to $\so_2\subset \sl_2$ for every simple root with respect to $\b$. We can then define simple reflection generators of the boundary Weyl group accordingly. In Proposition \ref{prop:boundary_nondyn_braid_rel}, we prove the boundary Weyl group operators satisfy braid relations by using compatibility of the boundary Weyl group operators for $\k\subset \g$ with the Weyl group operators for $\g$.

    \subsection{Boundary dynamical Weyl group}

    The main purpose of the first part of the paper is the construction of a boundary analog of the usual Dynamical Weyl group defined by Tarasov-Varchenko \cite{EtingofVarchenkoDynamicalWeyl}. The construction is based on the following observation.
	
	The Lie algebra $\g=\sl_2$ contains the subalgebra $\k = \so_2$, generated by $b=f-e$, and the subalgebra $\mathfrak{b}$, generated by $\{h,e\}$. The two bases $\{f, h, e\}$ and $\{b, h, e\}$ correspond to the usual triangular decomposition of $\g$ and the Iwasawa decomposition of $\g$ respectively. These two decompositions give two bases for the Verma module $M_{\lambda}:=\U(\g)\otimes_{\U(\b)}\C_{\lambda}$, namely
    \[
    \{f^i\otimes 1_{\lambda}\}_{i\in \mathbb{Z}_{ \ge 0}} \qquad \text{and} \qquad \{b^i\otimes 1_{\lambda}\}_{i\in \mathbb{Z}_{\ge 0}}.
    \]
    The first basis is adapted to studying nontrivial $\sl_2$-submodules of $M_{\lambda}$, which are used by Etingof--Varchenko to define the usual Dynamical Weyl group \cite{EtingofVarchenkoDynamicalWeyl}. The second basis makes it clear that restricting $M_{\lambda}$ to $\U(\k)$ yields a free module of rank one, so for any $\k$-module $X$ we have an isomorphism $\Hom_{\k}(M_{\lambda}|_{\k}, X)\cong X$.
    
    The interplay of the two bases allows us to define for every simple root $\alpha_i$  the \emph{boundary dynamical simple reflection} operator $A^{\k,s_i}_X(\lambda)$ acting on any integrable $\k$-representation $X$. These operators are boundary analogs of the standard dynamical operators $A^{\g,s_i}_{V}(\lambda)$ from \emph{loc. cit.} (denoted by $A_{s_i,V}(\lambda)$ there) acting on any integrable $\g$-module $V$. In particular, we show the following.
    \begin{thm}[Theorem \ref{thm:boundary_braid_relations}]\label{thm:mainthm1}
        The operators $A^{\k,s_i}_X(\lambda)$ satisfy the dynamical braid relation
    \begin{equation}
    \label{eq::intro_boundary_dyn_braid_relation}
		\ldots A^{\k,s_{i}}_X(s_{j}s_{i}\cdot \lambda)A^{\k,s_{j}}_X(s_{i}\cdot \lambda) A^{\k,s_{i}}_X(\lambda) = \ldots A^{\k,s_{j}}_X(s_{i}s_{j}\cdot \lambda)A^{\k,s_{i}}_X(s_{j}\cdot \lambda) A^{\k,s_{j}}_X(\lambda),
	\end{equation}
    where the number of factors on each side of the equation is the order of $(s_i s_j)$ in $W$, the Weyl group of $\g$.
    \end{thm}

	

    A key feature of the usual dynamical Weyl group is the tensor compatibility with the dynamical fusion operators introduced in \cite{EtingofVarchenkoExchange}, see Proposition \ref{P:weyl-fusion-compatible}. Our proof of \eqref{eq::intro_boundary_dyn_braid_relation} relies on a parallel tensor compatibility between the boundary dynamical Weyl group and the boundary fusion operator introduced in \cite{ReshetikhinStokman}. The precise statement appears in Proposition \ref{prop:boundary_tensor}. Along the way, we compute the boundary fusion operator in particular examples and derive a boundary version of the \emph{ABRR equation} \cite{ABRR}.
    

    \begin{remark}
        The boundary fusion operator gives rise to a \emph{dynamical $K$-matrix}, parallel to how the usual fusion operator defines a dynamical $R$-matrix. For more details see \cite[Section 6.4.3]{DeClercq-thesis}.
    \end{remark}
	
    \begin{remark}
		The boundary fusion operator and boundary dynamical Weyl group is defined in terms of structure constants for intertwining operators between Verma modules (and their tensor products with finite dimensional modules) for $\g$ and finite dimensional modules for $\k$. Unsurprisingly, the formulas turn out to be closely related to real groups. In Section \ref{subsect:k-B-operator}, we introduce a boundary analog of the \emph{$B$-operators}, used to define the dynamical difference equation, which turn out to coincide with the intertwining operators between principal series representations of real groups \cite{Wallach}. Despite having purely algebraic origin, the definition of boundary dynamical Weyl group operators requires analytic functions. 
	\end{remark}

    \subsection{Simplified series expressions}
    
        The usual Weyl group operators for simple reflections are expressed in terms of the triple exponential formula
        \[
        \tilde{s_i}^{\g} = \exp(-e_{\alpha_i})\exp(f_{\alpha_i})\exp(-e_{\alpha_i}).
        \]
        This is the $q=1$ version of formulas for the quantum Weyl group discovered by Lusztig, see \cite[Chapter 8]{JantzenLectures} for an elementary exposition. There is a simplified series for the usual generating reflection operators \cite[Lemma 4.2]{Chuang-Rouquier}. 
        
        The boundary Weyl group operators for generating reflections are expressed in terms of the exponential formula 
        \[
        \tilde{s_i}^{\k} = \exp\left(\frac{\pi b_{\alpha_{i}}}{2}\right).
        \]
        In Proposition \ref{prop:bweyl_series}, we give a new simplified series for the generating boundary Weyl group operators. This is the $q=1$ version of the $\iota$quantum Weyl group operators proposed in \cite{Zhang-Weinan-thesis}.
        
        The simplified series expression for the boundary Weyl group operators have dynamical analogs. In Definition \ref{D:simplified-dynamical-series} and Theorem \ref{T:simplified-dynamical-series} we describe simplified series for (normalized) boundary dynamical simple reflections. 
        
        \begin{remark}
        The coefficients in the simplified series for boundary dynamical simple reflections are rational functions of $\lambda$. However, we are required to restrict to integral eigenvalues of $b_{\alpha_i}$, and use two series: one for even eigenvalues and one for odd eigenvalues.
        \end{remark}
	
	\subsection{Quantum bispectral duality}
	
	The Yangian $\rY(\gl_n)$ \cite{MolevYangians} acts on $\Lambda^{\bullet}(\C^n)(z)$ via the evaluation homomorphism to $\U(\gl_n)$, which acts on $\Lambda^{\bullet}(\C^n)$, composed with a twist by $z\in \C$. The Yangian is a Hopf algebra, so $\rY(\gl_n)$ also acts on $\Lambda^{\bullet}(\C^n)(z_1)\otimes \Lambda^{\bullet}(\C^n)(z_2)$, which as a vector space is identified with $\Lambda^{\bullet}(\C^n\otimes \C^2)$. In particular, the rational $R$-matrix $\cR^{\gl_n}(z_1-z_2)$ acts. 
	
	On the other hand, the vector space $\Lambda^{\bullet}(\C^n\otimes \C^2)$ is identified with $\Lambda^{\bullet}(\C^2)^{\otimes n}$ as an $\sl_2$-module. It follows that the dynamical Weyl group, and the dynamical Weyl group $B$-operator $\mathcal{B}_{\alpha}^{\sl_2}(\langle \lambda, \alpha^{\vee}\rangle)$ from equation \eqref{eq::b_operator}, also act on $\Lambda^{\bullet}(\C^n\otimes \C^2)$. 
 
	A consequence of \cite[Theorem 4.4]{TarasovUvarov}, which we restate in Theorem \ref{thm::gl_quantum_bispectrality}, see \cite{DalipiFelder} for the trigonometric analog, is that the two operators
	\[
	\cR^{\gl_n}(z_1-z_2)\quad \mathlarger{\mathlarger{\circlearrowright}}\quad  \Lambda^{\bullet}(\C^n\otimes \C^2) \quad \mathlarger{\mathlarger{\circlearrowleft}} \quad \mathcal{B}_{\alpha}^{\sl_2}(-z_1+z_2)
	\]
	agree, up to a scalar factor depending on $z_1-z_2$. In fact more is true. There is a difference analog of the KZ connection giving rise to \emph{quantized Knizhnik-Zamolodchikov} (simply \emph{qKZ}) difference operators (see \cite{EtingofFrenkelKirillov} for a review). They are constructed from the rational $R$-matrices. The $qKZ$ difference equation corresponds under $(\gl_n, \gl_m)$-duality to \emph{dynamical difference operators} defined using the dynamical Weyl group $B$-operators.
	
	The (extended) Yangian $\rX(\so_{2n})$ acts\footnote{It is well known that there is no evaluation homomorphism for the Yangian outside of type $A$. However, see Section \ref{ss:sotn-Yangian}.} on the spinor representation $\fP_n(z)$, for $z\in \C$. Thus, there is a spinorial $R$-matrix\footnote{This is the $R$-matrix operator $\check{\cR}^{\so_{2n}}(z_1-z_2)$ from Definition \ref{def:spinorial_r_operator}, post-composed with the permutation operator.} $\cR^{\so_{2n}}(z_1-z_2)$ acting on $\fP_n(z_1)\otimes \fP_n(z_2)$, which as a vector space is $\fP_n^{\otimes 2}\cong \fP_{n\cdot 2}$. On the other hand, $\fP_{n\cdot 2}$ is identified with $\Lambda^{\bullet}(\C^n\otimes \C^2)\cong \Lambda^{\bullet}(\C^2)^{\otimes n}$. Therefore, $\fP_{n\cdot 2}$ carries an action of $\so_2\subset\sl_2$, the associated boundary dynamical Weyl group operator, and the boundary dynamical Weyl group $B$-operator $\mathcal{B}_{\alpha}^{\so_2}(t)$, see Definition \ref{def:boundary_b_operator}.
	
	The algebra $\so_{2\cdot 2}$ acts on $\fP_{2\cdot m}$, so the usual dynamical Weyl group for $\so_4$ acts as well. In particular, we have the dynamical Weyl group $B$-operators corresponding to the simple root $\epsilon_1+\epsilon_2$. The $\so_{2\cdot 2}$ action commutes with the action of $\rO_m= \GL_m^{\theta}$, and the spectral analog of this symmetric pair is given by \emph{twisted Yangian} \cite{MolevYangians}. In Section \ref{ss:twisted-Yang}, we introduce a (variant of the) twisted $R$-matrix $\mathcal{R}^{\gl_m, t_2}(z_1-z_2)$ which acts on $\fP_{2\cdot m}\cong \Lambda^{\bullet}(\C^m)(z_1)\otimes \Lambda^{\bullet}(\C^m)(z_2)$.
	
	Our main results on quantum bispectral duality for $(\so_{2n},\rO_m)$-duality is summarized as follows. 
	
	\begin{thm}[Propositions \ref{prop:so_r_operator_equal_bweyl} and \ref{prop:twisted_rmatrix_equal_dyn_operator} and Theorems \ref{thm:so_dynamical_bkz_bispectrality} and \ref{thm:orthogonal_duality_qkz_bdyn}]
		The operators
		\[
		\cR^{\so_{2n}}(z_1-z_2)\quad \mathlarger{\mathlarger{\circlearrowright}} \quad \fP_{n\cdot 2} \quad \mathlarger{\mathlarger{\circlearrowleft}} \quad \mathcal{B}_{\alpha}^{\so_2}(-z_1+z_2)
		\]
		and
		\[
		\mathcal{B}^{\so_4}_{\epsilon_1+\epsilon_2}(-z_1+z_2-1)\quad \mathlarger{\mathlarger{\circlearrowright}} \quad \fP_{2\cdot m} \quad \mathlarger{\mathlarger{\circlearrowleft}} \quad \cR^{\gl_m,t_2}(z_1-z_2)
		\]
		respectively, agree, up to scalar factors depending on $z_1-z_2$. Moreover, the qKZ equations for $\so_{2n}$ coincide with the boundary analog of the dynamical difference equations for $\rO_m$, and the dynamical difference equations for $\so_{2n}$ coincide with the boundary qKZ equations\footnote{The difference analog of the boundary KZ connection gives rise to \emph{boundary qKZ equations} \cite{CherednikQKZ} which are constructed from the $\theta$-twisted $R$-matrix.} for $\rO_m$.
	\end{thm}

    For the second part of the duality, see also the related work \cite{NazarovKhoroshkin}.
	
	\subsection{Classical bispectral duality}
	
	A common philosophy is to degenerate from quantum to classical by replacing difference equations with differential equations. Under this dictionary, the qKZ difference equation is replaced with the classical KZ equation and the related \emph{Knizhnik-Zamolodchikov} (or \emph{KZ}) \emph{connection} from conformal field theory \cite{KnizhnikZamolodchikov}. For a simple Lie algebra $\g$ and finite dimensional $\g$-modules $V_1, \dots, V_m$, one obtains the KZ connection on the trivial bundle $V_1\otimes \dots \otimes V_m$ over $\C^m$, with poles on small diagonals.
 
    Analogously, one can replace the dynamical difference equations with the \emph{Casimir connection} (for instance, see \cite{FelderMarkovTarasovVarchenko}). In this case, for any simple Lie algebra $\g$ and any finite dimensional module $V$, one obtains the Casimir connection on the trivial bundle $V$ on $\mathfrak{h}$, with poles on root hyperplanes. 
	
	The $\gl_n$ action on $\fP_{nm}\cong \Lambda^{\bullet}(\C^n)^{\otimes m}$ induces the KZ connection for $\fP_{nm}$ on $\C^n$. There is also the $\gl_m$ action on $\fP_{nm}$ and so we have the Casimir connection for $\fP_{nm}$ over $\mathfrak{h}_{\gl_m}$. We indicate this schematically by writing
	\begin{equation}
    \label{eq::intro_gl_duality_connections}
    	\nabla_{\gl_n}^{\mathrm{KZ}}\quad \mathlarger{\mathlarger{\circlearrowright}} \quad \fP_{n\cdot m} \quad \mathlarger{\mathlarger{\circlearrowleft}} \quad \nabla_{\gl_m}^{\mathrm{Cas}}.
	\end{equation}
	Amazingly, by identifying $\C^m$ with $\h_{\gl_m}$, these two connections match under $(\gl_n,\gl_m)$-duality \cite{TarasovUvarov}, see \cite{TarasonVarchenkoDuality} for analogous result in the context of symmetric Howe $(\gl_n,\gl_m)$-duality. A precise statement is given in Theorem \ref{eq::gl_classical_bispectrality}.
	
	The boundary analog of the KZ connection is a particular case of the so-called \emph{cyclotomic KZ connection} introduced in \cite{EnriquezEtingof} for a finite-order automorphism of a Lie algebra. To match other notations, we will refer to the cyclotomic KZ-connection for an order $2$ automorphism $\theta:\g\rightarrow \g$ as the boundary KZ connection, denoted $\nabla_{\g^{\theta}\subset \g}^{\mathrm{bKZ}}$. 
 
    One of the outcomes of our studying the $(\so_{2n},\rO_m)$-duality was the discovery of the following boundary analog of the Casimir connection. 
	
	\begin{thm}[Definition \ref{D:boundary-Casimir} and Theorem \ref{thm::bcasimir_flat}]
		Let $\k\subset \g$ be a split symmetric pair and let $X$ be an integrable $\k$-module. The boundary Casimir connection $\nabla^{\mathrm{bCas}}_{\k\subset \g}$ is a flat connection for the trivial $X$ bundle on $\h$. 
	\end{thm}
	
	\begin{remark}
		In analogy with \cite{ToledanoLaredoKohnoDrinfeldWeylGroup}, we make Conjecture \ref{conjecture:bcas_monodromy}: the monodromy of $\nabla_{\k\subset \g}^{\mathrm{bCas}}$ is related to the $q$-analog of the boundary Weyl group for the quantum symmetric pair $\U_q^{\iota}(\k)\subset \U_q(\g)$.     
	\end{remark}
	
	The following theorem summarizes our main results on classical bispectral duality for $(\so_{2n}, \rO_m)$-duality.
	
	\begin{thm}[Theorem \ref{prop::so_bkz_cas_duality} and Theorem \ref{thm:orthogonal_duality_bcas_kz}]
    \label{thm:orthogonal_classical_duality}
		The connections
		\[
		\nabla_{\so_{2n}}^{\mathrm{KZ}}\quad \mathlarger{\mathlarger{\circlearrowright}} \quad \fP_{n\cdot m} \quad \mathlarger{\mathlarger{\circlearrowleft}} \quad \nabla_{\gl_m^{\theta}\subset \gl_m}^{\mathrm{bCas}}
		\]
		and
		\[
		\nabla_{\so_{2n}}^{\mathrm{Cas}}\quad \mathlarger{\mathlarger{\circlearrowright}} \quad \fP_{n\cdot m} \quad \mathlarger{\mathlarger{\circlearrowleft}} \quad \nabla_{\gl_m^{\theta}\subset \gl_m}^{\mathrm{bKZ}}
		\]
		respectively, agree, after making suitable identifications, under $(\so_{2n}, \rO_m)$-duality.
	\end{thm}

    For any $\g$, the standard KZ connection is closely related to \emph{Gaudin Hamiltonians} \cite{GaudinBook} acting on a tensor product of finite dimensional $\mathfrak{g}$-modules $V_1\otimes \ldots \otimes V_m$. The Gaudin Hamiltonians, and their higher analogs constructed with the \emph{Feigin-Frenkel center} \cite{FFR}, depend on $m$ distinct points $z_1, \dots, z_m$ in $\mathbb{C}$. 
    
    The (higher) Gaudin Hamiltonians generate  a commutative subalgebra $\mathcal{A}^{\g}(z_1,\ldots,z_m) \subset \End(V_1\otimes \ldots \otimes V_m)$. By \cite{RybnikovShiftOfArgument}, this is a particular case of a commutative subalgebra $\cA_{\lambda}^{\g} (z_1,\ldots,z_m)$ which also depends on a dynamical parameter $\lambda\in \h^*$. When $m=1$, it gives the other special case: the quantum shift of argument algebra $\cA_{\lambda}^{\mathfrak{g}}$, which is generated by purely dynamical Hamiltonians closely related to the Casimir connection for $\mathfrak{g}$. 
    
    It turns out that the aforementioned $(\gl_n,\gl_m)$-duality yields a \emph{bispectral duality} between the (higher) Gaudin Hamiltonians. Precisely this is the following equality 
    \[
        \cA^{\gl_n}_{(\lambda_1,\ldots,\lambda_n)}(z_1,\ldots,z_m) = \cA^{\gl_m}_{(z_1,\ldots,z_m)}(\lambda_1,\ldots,\lambda_n) \subset \End(\fP_{nm}), 
    \]
    generalizing \eqref{eq::intro_gl_duality_connections}, see \cite{HuangMukhin} and \cite{MTV} for the symmetric case. Moreover, by \emph{loc. cit.}, this result can be viewed as a generalization of the classical Capelli identity. 
    
    Similarly, on the $\mathfrak{so}_{2n}$-side of $(\mathfrak{so}_{2n}, \rO_m)$-duality one can consider $\cA_\lambda^{\mathfrak{so}_{2n}}(z_1, \dots, z_n)\subset \End(\fP_{nm})$. Note that $\lambda \in \h^*_{\mathfrak{so}_{2n}}$. We expect that on the $O_m$-side of $(\mathfrak{so}_{2n},\rO_m)$-duality this algebra of operators can be generated by Hamiltonians constructed in a way which uses the symmetric pair $\rO_m\subset \mathrm{GL}_m$.
    
    A candidate construction for such operators would come from the \emph{cyclotomic} analog of the higher Gaudin Hamiltonians defined in \cite{VicedoYoung} and \cite{VicedoYoundIrregularSing}. We conjecture that the cyclotomic Hamiltonians for $\rO_m\subset \mathrm{GL}_m$ coincide, under $(\mathfrak{so}_{2n},\rO_m)$-duality, with the purely dynamical Hamiltonians for $\mathfrak{so}_{2n}$. Some evidence is that the construction of \emph{loc. cit.} applied to a particular element of the center gives operators closely related to the boundary KZ connection which is dual to the Casimir connection on the $\so_{2n}$-side by \cref{thm:orthogonal_classical_duality}.
    
    However, it is less clear how to construct operators on the $\rO_m\subset \GL_m$-side which correspond to the (higher) Gaudin Hamiltonians on the $\mathfrak{so}_{2n}$-side of $(\mathfrak{so}_{2n},\rO_m)$-duality. A naive guess, based on other bispectral duality results, is that this construction should stem from a boundary analog of the quantum shift of argument algebra $\mathcal{A}^{\mathfrak{k}\subset \mathfrak{g}}_{\lambda}$ where $\lambda\in \mathfrak{h}^*$ for $\mathfrak{h}\subset \mathfrak{g}$ a Cartan.
    
    Beyond bispectral duality results, it would be interesting to extend some general results about the usual higher Gaudin Hamiltonians to the cyclotomic setting, and the $(\so_{2n},\rO_m)$-duality could provide a useful insight into such problems.  An instance of a general result is that, under certain conditions on $\lambda$ and $z_1,\ldots,z_m$, the algebra $\cA^{\g}_{\lambda}(z_1,\ldots,z_m)$ acts with simple spectrum on the tensor product of irreducible finite-dimensional representations, see \cite[Corollary 5]{FFR}. We do not know if there is a similar result for the cyclotomic Gaudin subalgebras. If there is such simple spectrum result, a further line of inquiry would be to relate the cyclotomic Gaudin subalgebras to a boundary analog of crystals \cite{Watanabe} as in \cite{HKRW}.
  
	
	\subsection{Structure of the paper}
	In Section \ref{s:lie-theory} we introduce the necessary notations, in particular we introduce split symmetric pairs $\k\subset \g$. In \cref{sect:boundary_braid_group}, we recall the braid group action on finite-dimensional $\g$-modules, then introduce the boundary analog of the braid group action on integrable representations of $\k$. In \cref{sect:dynamical_twist}, we recall the dynamical fusion construction of \cite{EtingofVarchenkoExchange}, then in \cref{sect:boundary_fusion}, we introduce and study a boundary analog of dynamical fusion. In \cref{subsect:sl2_dw_diag}, we recall the construction of the dynamical Weyl group \cite{EtingofVarchenkoDynamicalWeyl}, then in Section \ref{s:boundary-dynamical-Weyl} we study the boundary dynamical Weyl group. In \cref{sect:Conections}, we recall some basic facts on connections and introduce a boundary analog of the Casimir connection. In \cref{sect::gl_duality}, we recall the bispectral $(\gl_n,\gl_m)$-duality. Finally, in \cref{sect:orthogonal_howe}, we study the orthogonal bispectral duality.

        \subsection*{Acknowledgments}
	
	This project started as a collaboration with Hunter Dinkins and we are thankful for his input. We benefited from discussions with Pavel Etingof, Valerio Toledano Laredo, Leonid Rybnikov, Jasper Stokman, and Milen Yakimov. Thanks also to Stefan Kolb, and Weiqiang Wang for helpful correspondence about the rank $1$ formulas discussed in Remark \ref{R:iquantum-was-there-first}. E.B. was partially supported by the NSF MSPRF-2202897.
	
\section{Lie theory notation}\label{s:lie-theory}
	
	We begin be establishing conventions for the rest of the paper. Here we follow \cite[Section 2]{ReshetikhinStokman}.

    \subsection{Cartan involution}
	Let $\g_{\mathbb{R}}$ be a real simple Lie algebra with Killing form $\kappa_{\mathbb{R}}(-,-)$. There is a \emph{Cartan involution} $\theta_{\mathbb{R}}:\g_{\mathbb{R}}\rightarrow \g_{\mathbb{R}}$, which by definition is a Lie algebra automorphism of order $2$, such that $-\kappa_{\mathbb{R}}(-,\theta_{\mathbb{R}}(-))$ is positive definite. Cartan involutions are unique up to conjugation by an inner automorphism of $\g_{\mathbb{R}}$. 
	
	We fix a Cartan involution $\theta_{\mathbb{R}}$. There is a corresponding Cartan decomposition 
	\[
	\g_{\mathbb{R}} = \k_{\mathbb{R}} \oplus \mathfrak{p}_{\mathbb{R}},
	\]
	where $\k_{\mathbb{R}}=\g_{\mathbb{R}}^{\theta}$ and $\theta_{\mathbb{R}}$ acts on $\mathfrak{p}_{\mathbb{R}}$ by $-1$. 
 
    We further assume that the Lie algebra $\g_{\mathbb{R}}$ is \emph{split}, i.e., $\mathfrak{p}_{\mathbb{R}}$ contains a Cartan subalgebra $\h_{\mathbb{R}}\subset \g_{\mathbb{R}}$. The restriction of $\kappa_{\mathbb{R}}(-,-)$ to $\h_{\mathbb{R}}$ is positive definite. In particular, $\kappa_{\mathbb{R}}(-,-)$ identifies $\h_{\mathbb{R}}$ with $\h_{\mathbb{R}}^*$ and there is an induced inner product $\kappa^*(-,-)$ on $\h_{\mathbb{R}}^*$. Write $\langle-,-\rangle_{\mathbb{R}}$ for the natural pairing between $\h_{\mathbb{R}}^*$ and $\h_{\mathbb{R}}$.
	
	Let $\g, \k, \mathfrak{p}$, and $\h$ denote the complexifications of $\g_{\mathbb{R}}$, $\k_{\mathbb{R}}$, $\mathfrak{p}_{\mathbb{R}}$, and $\h_{\mathbb{R}}$. We have $\theta = \C\otimes \theta_{\mathbb{R}}$, an order two automorphism of $\g$ such that $\g^{\theta} = \k$ and $\mathfrak{p}$ is the $-1$ eigenspace of $\theta$. We also write $\kappa(-,-)$ and $\langle-,-\rangle$ for the complexifications of $\kappa_{\mathbb{R}}(-,-)$ and $\langle -,-\rangle_{\mathbb{R}}$.

    \subsection{Roots and generators}\label{ss:roots-and-gens}
 
	The Lie algebra $\g$ is simple with Cartan subalgebra $\h$ and corresponding root space decomposition
	\[
	\g= \h\oplus \bigoplus_{\alpha\in R}\g_{\alpha}.
	\]
    For $\alpha \in R$, we denote by $t_{\alpha}\in \h$ the element corresponding to $\alpha\in \h^*$ under the identification $\h\xrightarrow{\cong}\h^*$ induced by $\kappa$, i.e., $t_{\alpha}\in \h$ is the unique element such that
    \[
    \kappa(t_{\alpha}, h) = \langle\alpha,h\rangle \qquad \text{for all $h\in \h$}.
    \]
    We have 
    \[
    [x, y] = \kappa(x, y)t_{\alpha}
    \]
    for $x\in \g_{\alpha}$ and $y\in \g_{-\alpha}$.
 
 The root space decomposition restricts to a decomposition
	\[
	\g_{\mathbb{R}} = \h_{\mathbb{R}}\oplus \bigoplus_{\alpha\in R} (\g_{\mathbb{R}}\cap \g_{\alpha}).
	\]
	This means $\alpha$ is real valued on $\h_{\mathbb{R}}^*$ and $t_{\alpha}\in \h_{\mathbb{R}}$. We fix generators $E_{\alpha}\in \g_{\mathbb{R}}\cap \g_{\alpha}$, for all $\alpha \in R$, such that $[E_{\alpha},E_{-{\alpha}}] = t_{\alpha}$, i.e., $\kappa_{\mathbb{R}}(E_{\alpha}, E_{-\alpha})=1$, and $\theta_{\mathbb{R}}(E_{\alpha}) = - E_{-\alpha}$ (it is possible by \cite[\S 25.2]{Humphreys}). Note that 
	\[
	\k_{\mathbb{R}} = \mathbb{R}\cdot \{E_{\alpha} - E_{-\alpha}\}_{\alpha \in R} \qquad \text{and} \qquad \mathfrak{p}_{\mathbb{R}} =\h_{\mathbb{R}}\oplus \mathbb{R} \cdot \{E_{\alpha} + E_{-\alpha}\}_{\alpha \in R}. 
	\]
	
    Next, we fix a choice of positive roots $R_+\subset R$. Let $\Pi = \{\alpha_i\}\subset R_+$ be the associated set of simple roots. We have,
	\[
	\g = \n_- \oplus \h \oplus \n,
 \]
 where
 \[
 \n_-  = \bigoplus_{\alpha\in R_+} \g_{-\alpha}, \qquad \h = \bigoplus_{\alpha_i\in \Pi}\C\cdot t_{\alpha_i} ,\qquad \text{and} \qquad \n = \bigoplus_{\alpha \in R_+} \g_{\alpha}.
	\]
    We have generators
	\begin{equation}
		\label{eq::root_gens_sl2}
		E_{\alpha} \in \g_{\alpha} \qquad \text{and} \qquad F_{\alpha} := E_{-\alpha} \in \g_{-\alpha},
	\end{equation}
	for $\alpha\in R_+$, satisfying 
	\[
	[E_{\alpha}, F_{\alpha}] = t_{\alpha}.
	\]

	\begin{notation}
	For $\sl_2$ there is one positive root $\alpha$. We will write
    \[
    \n_- = \C\cdot F, \qquad \h=\C\cdot t, \qquad \text{and} \qquad \n_+ = \C\cdot E,
    \]
    where $E= E_{\alpha}$, $t= t_{\alpha}$, and $F= F_{\alpha}$.
	\end{notation}

    \subsection{Iwasawa decomposition}
 
    The Lie algebra $\g_{\mathbb{R}}$ has the following \emph{Iwasawa decomposition}: 
    \[
    \g_{\mathbb{R}} = \k_{\mathbb{R}}\oplus \h_{\mathbb{R}} \oplus \bigoplus_{\alpha \in R_+}(\g_{\mathbb{R}}\cap \g_{\alpha}).
    \]
    Thus, $\g = \k\oplus \b$, where $\b = \h\oplus \n_+$ is the Borel subalgebra and $\k = \g^{\theta}$ are the $\theta$ fixed points in $\mathfrak{g}$. Explicitly, 
	\[
	\k = \bigoplus_{\alpha\in R_+} \C \cdot B_{\alpha},	
	\]
	where 
	\begin{equation}
		\label{eq::b_gens_sl2}
		B_{\alpha} = F_{\alpha} + \theta(F_{\alpha}) = F_{\alpha} - E_{\alpha}.
	\end{equation}
 
	\begin{notation}
		For $\sl_2$, we denote the generator of $\k = \so_2$ simply by $B := F - E$.   
	\end{notation}

    \subsection{Chevalley involution and generators}
    \label{subsect:chevalley_involution}

    To relate the generators $E_{\alpha}$, $F_{\alpha}$ and $t_{\alpha}$ to the standard Chevalley generators, set
    \[
    \alpha^{\vee}:=\frac{2}{\kappa(t_{\alpha},t_{\alpha})}t_{\alpha},
    \qquad \text{so}\qquad \langle \beta, \alpha^{\vee}\rangle = \frac{2\kappa^*(\alpha,\beta)}{\kappa^*(\alpha,\alpha)},
    \]
    and set
    \[
    e_{\alpha} :=\sqrt{\frac{2}{\kappa(t_{\alpha},t_{\alpha})}}E_{\alpha} \qquad \text{and} \qquad f_{\alpha} :=\sqrt{\frac{2}{\kappa(t_{\alpha},t_{\alpha})}}F_{\alpha}, \qquad \text{so} \qquad [e_{\alpha}, f_{\alpha}]= \alpha^{\vee}.
    \]
    We also write $b_{\alpha}:=f_{\alpha} - e_{\alpha}$.

    \begin{remark}
    Note that $\theta \colon \g \rightarrow \g$ is determined on simple root generators by
	\[
	\theta(e_{\alpha_i}) = -f_{\alpha_i}, \qquad \theta(f_{\alpha_i}) = - e_{\alpha_i}, \qquad \text{and} \qquad \theta|_{\h_i} = -\id_{\h_i}.
	\]
    Thus, $\theta$ is conjugate to the standard Chevalley involution, which swaps $e_{\alpha_i}$ and $f_{\alpha_i}$ and acts as $-1$ on $\h$, the existence of which is a Corollary of Serre's presentation of $\g$. In general, Chevalley involution refers to the automorphism of $\g$ corresponding to the Cartan involution of the split real form. 
    \end{remark}

    \subsection{Integrable representations}\label{ss:integrable}
 
     	The Lie algebra $\k$ is reductive which follows from the classification result \cite[Appendix C]{KnappLieGroups}. Fix a connected reductive group $K$ with the Lie algebra $\k$. We call a $\k$-representation \defterm{integrable} if it comes from a representation of $K$. By a ``finite-dimensional representation of $\g$'' we understand a \emph{Harish-Chandra $(\g,K)$-module}, i.e., a $\g$-module which is integrable over $\k$. In what follows, we will not refer to $K$ explicitly.

        We are studying the symmetric pair $\k\subset \g$. The tensor product of an integrable $\k$-module with a finite dimensional $\g$-module is integrable. Thus, we think of the category of integrable $\k$-modules as a module category over finite dimensional $\g$-modules via the coproduct
        \[
        \Delta\colon \U(\k)\rightarrow \U(\k) \otimes \U(\g), \qquad x\mapsto x\otimes 1 + 1\otimes x, \qquad \text{for $x\in \k$}.
        \]
        Our choice to ignore that $\Delta$ has image in $\U(\k)\otimes \U(\k)$ becomes clearer once we introduce dynamical symmetric pairs.
        
    \subsection{Weights and representations}

    Write $\{ \omega_i\}\subset \h^*$ for the fundamental weights, which induce coordinates on $\h^*$:
	\[
	\h^* \ni \lambda = \lambda_1 \omega_1 + \ldots + \lambda_n \omega_n, \qquad \lambda_i = \langle \lambda, \alpha_i^{\vee}\rangle. 
	\]
	We also denote by 
	\[
	X=\bigoplus_{i=1}^n \mathbb{Z}\omega_i \qquad \text{and} \qquad X_+=\bigoplus_{i=1}^n\mathbb{Z}_{\ge 0}\omega_i
	\]
	the set of integral weights and dominant integral weights respectively. 
 
    Denote by $W$ the Weyl group of $\g$, with reflection generators $\{s_{\alpha}\}_{\alpha\in R}$, where $s_{\alpha}$ acts on $\h_{\mathbb{R}}^*$ by reflecting through the hyperplane $H_{\alpha} = \{\lambda | \kappa^*_{\mathbb{R}}(\lambda, \alpha) = 0\}$. We extend this action to $\h^*$. The Weyl group $W$ is generated by simple reflection generators $\{s_i:=s_{\alpha_i}\}_{\alpha_i\in \Pi}$. Recall that $\rho:=\sum_{i=1}^n \omega_i$. The dot action of $W$ on $\h^*$ is defined by:
	\begin{equation}\label{E:dot-action}
		w\cdot \lambda := w(\lambda + \rho) - \rho, \ w\in W.
	\end{equation}

	It is convenient to work with the scalar multiple of the Killing form on. Let $\alpha_{\text{short}}$ be a short root and define 
        \[
        (-,-):=\frac{2}{\kappa^*(\alpha_{\text{short}},\alpha_{\text{short}})}\kappa(-,-).
        \]
    This is a non-degenerate $\g$-invariant form on $\g$. We also write $(-,-)$ to denote the induced form on $\h^*$, which is $W$-invariant and for any short root $\alpha$ satisfies
        \[
        (\alpha, \alpha) = (t_{\alpha}, t_{\alpha}) = \frac{2}{\kappa^*(\alpha, \alpha)}\kappa(t_{\alpha},t_{\alpha}) = 2.
        \]
    In fact, since $W$ acts irreducibly on $\h^*$ these two properties characterize $(-,-)$.

	\section{Boundary Weyl groups}\label{sect:boundary_braid_group}

    A (split) symmetric pair $\k\subset \g$ will have an associated boundary Weyl group and braid group action on integrable $\k$-modules. We recall the approach to the usual Weyl group, then expound on the parallel approach to the boundary Weyl group. 
	
	\subsection{Braid group action for Weyl group}
	
	By definition, the Weyl group $W$ acts on the Cartan subalgebra $\h \subset \g$. There is also an action of the braid group of $W$ on finite-dimensional representations of $\g$. For all the statements and proofs, we refer the reader to \cite{Lusztig} for the $q$-version. 
	
	Let $\g = \sl_2$. Consider the operator
	\[
	\tilde{s}^{\sl_2} := \exp(-e) \exp(f) \exp(-e),
	\]
	lying in a certain completion of $\U(\sl_2)$. Its action is well-defined on any finite-dimensional $\sl_2$-module $V$. Moreover, the Weyl group operator is group-like, i.e., $\Delta(\tilde{s}^{\sl_2}) = \tilde{s}^{\sl_2} \otimes \tilde{s}^{\sl_2}$. \begin{example}
	    For instance, if $V = V(1)$ is the standard vector representation, then $\tilde{s}^{\sl_2} = \begin{bmatrix}
		0 & -1 \\ 1 & 0
	\end{bmatrix}.$
	\end{example} 
	
	For a general simple Lie algebra $\g$, define
	\[
	\tilde{s}^{\g}_{\alpha_i} := \exp(-e_{{\alpha_i}}) \exp(f_{\alpha_i}) \exp(-e_{\alpha_i}),
	\]
	for any simple root $\alpha_i$. In what follows, we will also use the notation $\tilde{s}^{\g}_{i} := \tilde{s}^{\g}_{\alpha_i}$ Their action is well-defined on any finite-dimensional representation of $\g$. For any pair of simple roots $\alpha_i,\alpha_j$, they satisfy the \emph{braid relation}
	\begin{equation}
		\label{eq::usual_weyl_braid_relation}
		\ldots \tilde{s}^{\g}_{\alpha_i} \tilde{s}^{\g}_{\alpha_j} \tilde{s}^{\g}_{\alpha_i} = \ldots \tilde{s}^{\g}_{\alpha_j} \tilde{s}^{\g}_{\alpha_i} \tilde{s}^{\g}_{\alpha_j},
	\end{equation}
	depending on the root system of rank 2 that they generate. In particular, they define a braid group action on any finite-dimensional representation of $\g$.
	
	For any element $w\in W$, let $w = s_{i_1} \ldots s_{i_k}$ be a reduced decomposition. Define $\widetilde{w}:= \tilde{s}_{i_1} \ldots \tilde{s}_{i_k}$. By \eqref{eq::usual_weyl_braid_relation}, this element does not depend on the reduced decomposition. 
	
	\begin{remark}
		\label{remark:simply_connected_group}
		Let $G$ be the simply connected group with the Lie algebra $\g$. Then the operators $\widetilde{w}$ for $w\in W$ can be considered as elements of $G$. In particular, the group-like property $\Delta(\widetilde{w}) = \widetilde{w}\otimes \widetilde{w}$ is automatic. The group $G$ acts on $\g$ by the adjoint representation, thus by viewing $\widetilde{w}\in G$ we have $\mathrm{Ad}_{\widetilde{w}}:\g\rightarrow \g$. This agrees with the action of $\widetilde{w}$ on $\g$ obtained from viewing $\g$ as the adjoint representation of $\g$.
	\end{remark}

	\subsection{Braid group action for boundary Weyl group: \texorpdfstring{$\so_2\subset \sl_2$}{so(2) inside sl(2)}}
	\label{subsect:boundary_weyl_group}
	
	The action of 
	\begin{equation}
		\label{eq::weyl_operator_sl2}
		\tilde{s}^{\sl_2} = \begin{bmatrix} 0 & -1\\ 1 & 0\end{bmatrix} \in SL_2
	\end{equation}
	on a finite dimensional $\sl_2$-module is realized by the operator $ \exp(-e)\exp(f)\exp(-e)\cdot$. Since $\tilde{s}^{\sl_2}\in SO(2)$, it is natural to consider $\tilde{s}^{\sl_2}$ acting on $V|_{\so_2}$ and try to express the action in terms of some operators depending on $b\in \so_2$.  
	
	\begin{example}
		\label{example:boundary_weyl_sl2_vector_rep}
		Consider the $\sl_2$-module $V(1)$, with weight basis $v_+$ and $v_-$. We have $b= f-e\in \so_2$ acting by $b\cdot v_+ = v_-$ and $b\cdot v_- = -v_+$. Thus, $V(1)|_{\so_2} = \C\{v_+-iv_-\}\oplus \C\{v_++iv_-\}\cong \C_{i}\oplus \C_{-i}$. In this representation, we have $ \tilde{s}=b$, so $\tilde{s}$ acts on $\C_i$ by $i$ and on $\C_{-i}$ by $-i$. 
	\end{example}
	
	The previous example is a bit misleading. Since $b\in \so_2$ is a Lie algebra element it acts on tensor products by $b\otimes 1+ 1\otimes b$, while $\tilde{s}^{\sl_2}$ acts on tensor products by $\tilde{s}^{\sl_2}\otimes \tilde{s}^{\sl_2}$, one should not expect their actions to agree on all $\C_x$. 
	
	\begin{notation}
		\label{notation:analytic_continuation_operators}
		Let $X$ be any integrable $\so_2$-module, e.g. $X$ is the restriction of a finite dimensional $\sl_2$-module. In particular, $X$ decomposes as $X = \bigoplus_i \C_{x_i}$, where the generator $b\in \so_2$ acts on $\C_{x_i}$ by multiplication with $x_i$. For any meromorphic analytic function $f(x)$, we define the operator $f(b)$ by its restriction $f(b)|_{\C_{x_i}} := f(x_i)$ to the one-dimensional irreducible components.
	\end{notation}
	
	\begin{prop}\label{P:tilde-s-equal-exp}
		For any finite-dimensional module $V$ of $\sl_2$, the action of $\tilde{s}^{\sl_2}$ is equal to the action of $\exp\left(\frac{\pi b}{2}\right)$ understood in the sense of \cref{notation:analytic_continuation_operators}. 
		\begin{proof}
			By \cref{example:boundary_weyl_sl2_vector_rep}, the actions agree on $V(1)$. Since $\Delta(b) = b \otimes 1 + 1 \otimes b$ is a sum of two commuting operators, the operator $\exp\left(\frac{\pi b}{2}\right)$ is group-like, i.e., $\Delta(\exp\left(\frac{\pi b}{2}\right)) = \exp\left(\Delta(\frac{\pi b}{2})\right)=\exp\left(\frac{\pi b}{2}\right) \otimes \exp\left(\frac{\pi b}{2}\right)$. Since the operator $\tilde{s}$ is group-like as well, their actions agree on any tensor product $V(1)^{\otimes n}$. As any irreducible representation of $\sl_2$ can be realized as a subrepresentation of a suitable tensor product, the claim is then a consequence of complete reducibility of finite dimensional $\sl_2$-modules. 
		\end{proof}
	\end{prop}
	
	Since $\tilde{s}\in SO(2)$ it makes sense to consider its action in any integrable $\so_2$-module. 
	
	\begin{defn}
		\label{def:boundary_weyl_sl2}
		The \defterm{boundary Weyl group operator} $\tilde{s}^{\so_2}$ is the operator acting on any integrable $\so_2$-module by $\exp\left(\frac{\pi b}{2}\right)$.
	\end{defn}
	
	\begin{notation}
		\label{notation:boundary_weyl_sl2_1dim_rep}
		Consider the one-dimensional $\so_2$-representation $\C_{-ik}$. We denote by $\tilde{s}^{\so_2}(k) := \exp\left(-\frac{i\pi k}{2}\right)$ the action of $\tilde{s}^{\so_2}$ on $\C_{-ik}$. 
	\end{notation}

	
	
	For imaginary integral eigenvalues, the operator $\tilde{s}$ has a series presentation.
	
	\begin{defn}\label{N:divided powers}
		For $x\in \C$, introduce even and odd divided powers by the formulas
		\[
		x_0^{(2k)}:=\frac{(x^2-(2k-2)^2)(x^2-(2k-4)^2)\dots (x^2-0^2)}{(2k)!}
		\]
		and
		\[
		x_1^{(2k+1)}:=\frac{(x^2-(2k-1)^2)(x^2-(2k-3)^2)\dots (x^2-1^2)x}{(2k+1)!}.
		\]
		Define
		\begin{align}
			\label{eq::bweyl_series}
			\begin{split}
				s_0(x) &= \sum\limits_{k=0}^{+\infty} (-1)^k x_0^{(2k)},\qquad \text{for $x\in 2\Z$}, \qquad \text{and} \\
				s_1(x) &= \sum\limits_{k=0}^{+\infty} (-1)^k x_0^{(2k+1)},\qquad \text{for $x\in 1+2\Z$}.
			\end{split}
		\end{align}
	\end{defn}
	
	\begin{remark}\label{R:iquantum-was-there-first}
		Experts in the theory of $\iota$quantum groups will be familiar with the above formulas. Non-experts should compare the divided power formulas in Definition \ref{N:divided powers} with the $\iota$divided power formulas in \cite[Conjecture 4.13]{Bao-Wang-new-approach} (the conjecture that this is the $\iota$canonical basis is proven in \cite{Berman-Wang-formulas-for-divided}) and compare the formulas for the operators in equation \eqref{eq::bweyl_series} with \cite[Section 16.3]{Zhang-Weinan-thesis}.
	\end{remark}
	
	\begin{prop}
		\label{prop:bweyl_series}
		We have $\tilde{s}^{\so_2}(2x) = s_0(2x)$ and $\tilde{s}^{\so_2}(2x+1) =-is_1(2x+1)$ for all $x\in \Z$.
		\begin{proof}
			Observe that $s_0(-2x) = s_0(2x)$ and $s_1(-2x-1) = - s_1(2x+1)$, as well as $\tilde{s}^{\so_2}(-2x) = \tilde{s}^{\so_2}(2x)$ and $\tilde{s}^{\so_2}(-2x-1) = -\tilde{s}^{\so_2}(2x+1)$. Therefore, it is enough to prove the statement for positive integral $x$.
			
			It is an exercise to verify that for integral $x$:
			\begin{gather*}
				(2x+2)_0^{(2k)} - (2x)_0^{(2k)} = 2\cdot (2x+1)_1^{(2k-1)} \qquad \text{and} \qquad  
				(2x+1)^{(2k+1)}_1 - (2x-1)_1^{(2k+1)} = 2\cdot (2x)_0^{(2k)}.
			\end{gather*}
			Thus, 
			\begin{gather*}
				s_0(2x+2) - s_0(2x) = -2 s_1(2x+1) \qquad \text{and} \qquad 
				s_1(2x+1) - s_1(2x-1) = 2 s_0(2x).
			\end{gather*}
			At the same time, one can check that 
			\begin{gather*}
				\tilde{s}^{\so_2}(2x+2) - \tilde{s}^{\so_2}(2x) = -2i \tilde{s}^{\so_2}(2x+1) \qquad \text{and} \qquad 
				i \tilde{s}^{\so_2}(2x+1) - i \tilde{s}^{\so_2}(2x-1) = 2\tilde{s}(2x). 
			\end{gather*}
			We see that $s_0(2x)$ and $-is_1(2x+1)$ satisfy the same inductive relations as $\tilde{s}^{\so_2}(2x)$ and $\tilde{s}^{\so_2}(2x+1)$. Therefore, the statement follows from the equalities: $\tilde{s}^{\so_2}(0) = s_0(0)=1$ and $\tilde{s}^{\so_2}(1) = -is_1(1) = -i$. 
		\end{proof}
	\end{prop}
	
	\subsection{Braid group action for boundary Weyl group: \texorpdfstring{$\k\subset \g$}{k in g}}
	Similarly to the usual Weyl group operators, we define the boundary Weyl group for the split symmetric pair $(\g,\k)$ in terms of simple reflection. In what follows, we are in the setting of \cref{sect:boundary_fusion}.
	
	\begin{defn}
		Let $\alpha \in \h^*$ be a simple root. Define the \defterm{boundary Weyl group operator} $\tilde{s}^{\k}_{\alpha} := \exp\left(\frac{\pi b_{\alpha}}{2}\right)$,  understood in the sense of \cref{notation:analytic_continuation_operators}. When $\alpha = \alpha_i$, we will also denote $\tilde{s}^{\k}_i := \tilde{s}^{\k}_{\alpha_i}$.
	\end{defn}
	
	Let $w\in W$ be any Weyl group element and $s_{i_1} \ldots s_{i_l}$ be a reduced expression for $w$. Associated to the reduced expression, we have the operator $\tilde{s}^{\k}_{i_1} \ldots \tilde{s}^{\k}_{i_l}$. 
	
	\begin{remark}\label{R:realtive-agrees-with-usual}
		For a finite dimensional $\g$-module $V$, we have the usual Weyl group operators $\widetilde{w}^{\g}$ acting on $V$. By \cref{P:tilde-s-equal-exp}, the action of $\tilde{s}^{\k}_{i_1} \ldots \tilde{s}^{\k}_{i_l}$ on $V|_{\k}$ coincides with the action of $\widetilde{w}^{\g}$. In particular, if $X$ is an integrable $\k$-module and $V$ is a finite dimensional $\g$-module, then, recalling the conventions for $\Delta$ set forth in Section \ref{ss:integrable}, we find
        \begin{equation}\label{E:kotimesg}
        \Delta(\tilde{s}_i^{\k})|_{X\otimes V} = \tilde{s}_i^{\k}|_X\otimes \tilde{s}_i^{\g}|_V.
        \end{equation}
  
        These observations suggest we try to deduce the braid relations for the boundary Weyl group operators from the braid relations for the usual Weyl group operators. This is the tactic in the following proof, which provides a model for the proof of Theorem \ref{thm:mainthm1}.
	\end{remark}
	
	\begin{prop}
		\label{prop:boundary_nondyn_braid_rel}
		If $s_{j_1}\ldots s_{j_l}$ is another reduced expression for $w$, then
		\[
		\tilde{s}^{\k}_{i_1} \ldots \tilde{s}^{\k}_{i_l} = \tilde{s}^{\k}_{j_1} \ldots \tilde{s}^{\k}_{j_l}.
		\]  
		\begin{proof}
			If we show the braid relations for $W$ hold for the operators $\tilde{s}_{\alpha}^{\k}$, then the Proposition will follow. 
   
            Since each $\tilde{s}_{\alpha}^{\k}$ is group-like, it follows from equation \eqref{E:kotimesg} that 
			\[
			\Delta(\tilde{s}^{\k}_{i_1} \ldots \tilde{s}^{\k}_{i_l}) = \tilde{s}^{\k}_{i_1} \ldots \tilde{s}^{\k}_{i_l} \otimes \tilde{s}^{\g}_{i_1} \ldots \tilde{s}^{\g}_{i_l}.
			\]
			Thus, if $V$ is a finite dimensional representation of $\g$ and $X$ is an integrable representation of $\k$, then thanks to Remark \ref{R:realtive-agrees-with-usual}, we have
			\begin{equation}\label{E:wk=wk-otimes-wg}
				\tilde{s}^{\k}_{i_1} \ldots \tilde{s}^{\k}_{i_l}|_{X\otimes V} = \tilde{s}^{\k}_{i_1} \ldots \tilde{s}^{\k}_{i_l}|_{X}\otimes \widetilde{w}^{\g}|_{V}.
			\end{equation}
			Since the operators $\tilde{s}_{\alpha}^{\g}$ satisfy the braid relations, it is enough to prove the proposition for integrable representations $X_1,\ldots,X_l$ of $\k$ such that any other irreducible representation can be obtained as a direct summand of $X_i \otimes V$ for some finite dimensional $\g$-representation $V$. Moreover, it is enough to prove this for rank 2 algebras. This is done in \cref{appendix:rank_two}.
		\end{proof}
	\end{prop}
	
	\begin{defn}
		Define the boundary Weyl group operator for $w\in W$ to be $\widetilde{w}^{\k}:=\tilde{s}^{\k}_{i_1} \ldots \tilde{s}^{\k}_{i_l}$, where $w= s_{i_1}\ldots s_{i_l}$ is any reduced decomposition.
	\end{defn}

    	The following lemma shows that the braid group action on $\g$ preserves the basis $\{E_{\alpha}, E_{-\alpha}\}_{\alpha\in R}$ from Section \ref{ss:roots-and-gens}, up to a sign.
	\begin{lm}
		\label{lm:braid_group_action_g_sign}
		For any simple root $\alpha_i$ and root $\gamma$, we have $\mathrm{Ad}_{\tilde{s}_i^{\g}}(E_{\gamma}) = C_{\gamma} E_{s_i(\gamma)}$ such that $C_{\gamma} = C_{-\gamma} = \pm 1$. 
		\begin{proof}
			It is clear from the weight considerations that $\mathrm{Ad}_{\tilde{s}_i^{\g}}(\g_{\gamma}) \subset \g_{s_i(\gamma)}$, in particular, $\mathrm{Ad}_{\tilde{s}_i^{\g}}(E_{\gamma}) = C_{\gamma} E_{s_i(\gamma)}$. Since $\mathrm{Ad}_{\tilde{s}^{\g}_i}$ preserves the bilinear form on $\g$, we have
			\[
			C_{\gamma} C_{-\gamma} =(C_{\gamma}E_{s_i(\gamma)}, C_{- \gamma}E_{s_i(-\gamma)}) = (\mathrm{Ad}_{\tilde{s}^{\g}_i}(E_{\gamma}),\mathrm{Ad}_{\tilde{s}^{\g}_i}(E_{-\gamma})) = (E_{\gamma},E_{-\gamma}) = 1.
			\]
			Moreover, \cref{P:tilde-s-equal-exp} implies $\mathrm{Ad}_{\tilde{s}^{\g}_i} = \mathrm{Ad}_{\tilde{s}^{\k}_i}$, hence $\mathrm{Ad}_{\tilde{s}^{\g}_i}$ commutes with the Cartan involution $\theta:\g\rightarrow \g$. Therefore, 
			\[
			C_{\gamma}^2 = (\mathrm{Ad}_{\tilde{s}^{\g}_i}(E_{\gamma}), \theta \circ \mathrm{Ad}_{\tilde{s}^{\g}_i}(E_{\gamma})) = (\mathrm{Ad}_{\tilde{s}^{\g}_i}(E_{\gamma}),\mathrm{Ad}_{\tilde{s}^{\g}_i} \circ \theta (E_{\gamma})) = C_{\gamma} C_{-\gamma} = 1.
			\]
			The lemma follows.
		\end{proof}
	\end{lm}

	\section{Background on dynamical fusion}\label{sect:dynamical_twist}
	
	In this subsection, we briefly recall the construction of the dynamical twist of \cite{EtingofVarchenkoExchange}.

\subsection{Fusion operator}
	
	Since $\h$ is an abelian Lie algebra, we can identify its universal enveloping algebra $\U(\h)$ with the algebra $\cO(\h^*)$ of polynomial functions on the dual space $\h^*$. In what follows, we will often need to restrict to a certain open subset of $\h^*$.
	\begin{defn}
		A weight $\lambda\in \h^*$ is called \defterm{generic} if $\langle \lambda, \alpha^{\vee}\rangle \not\in \Z$ for every coroot $\alpha^{\vee}$.
	\end{defn}
	
	\begin{notation}
		\label{notation:generic_versions}
		We denote by $\mathcal{M}(\h^*)$ the space of meromorphic functions on $\h^*$ with poles on the hyperplanes $\{ \langle \alpha^{\vee}, \lambda \rangle = k\}$ for all $\alpha \in R$ and $k \in \Z$.
		In particular, we have an algebra embedding $\U(\h) \cong \cO(\h^*) \hookrightarrow \cM(\h^*)$.
	\end{notation}
	
	We will also consider a generic version of $\U(\b_{\pm})$.
	
	\begin{defn}\label{defn:generic_versions}
		Define $\U(\b_{\pm})^{\gen} := \U(\b_{\pm}) \otimes_{\U(\h)} \mathcal{M}(\h^*)$. Observe that there is an adjoint action of $\h$ on $\U(\b_{\pm})^{\gen}$, and the algebra structure on $\U(\b_{\pm})$ can be extended to $\U(\b_{\pm})^{\gen}$ by the rule $P(\lambda) x := x P(\lambda + \wt(x))$ for any homogeneous element $x \in \U(\b_{\pm})^{\gen}$ and any $P(\lambda) \in \mathcal{M}(\h^*)$.
	\end{defn}

	Let $\lambda$ be a character of $\h$. It can be extended to a one-dimensional representation $\C_{\lambda}$ of the Borel subalgebra $\b:= \h \oplus \n$. Denote by
	\begin{equation}
		\label{eq::verma}
		M_{\lambda} = \U(\g) \otimes_{\U(\b)} \C_{\lambda}
	\end{equation}
	the Verma module generated by the highest-weight vector $m_{\lambda}$. 
	
	The following definition recalls the usual \defterm{dynamical notation}.
	
	\begin{defn}
		\label{def:dynamical_notations}
		Let $V_1, V_2, \dots, V_d$ be $\g$-representations. For a function $f:\h^*\rightarrow \End(V_1\otimes V_2\otimes \cdots \otimes V_d)$, denote by $\lambda \mapsto f(\lambda - h|_{V_i})$, the function
		\[
		\h^*\rightarrow \End(V_1\otimes V_2\otimes \dots \otimes V_d)
		\]
		such that, if $v_1, v_2, \dots, v_d$ are weight vectors with $v_i\in V_i$, then 
		\[
		f(\lambda - h|_{V_i})\cdot v_1\otimes v_2\otimes \dots \otimes v_d = f(\lambda - \wt(v_i))\cdot v_1\otimes v_1\otimes \dots \otimes v_d.
		\]
	\end{defn}
	
	\begin{notation}\label{not:dynamical-notation-U=V=W}
		In case $V_1=V_2=\dots =V_d$, we will write $h^{(i)}$ instead of $h|_{V_i}$. 
	\end{notation}
	
	Let $V$ be a finite-dimensional representation of $\g$. For a weight $\nu \in \h^*$, we denote by $V[\nu]$ the corresponding weight space of $V$. Consider the space of intertwiners $\Hom_{\g} (M_{\lambda},M_{\mu} \otimes V)$. 
	
	\begin{thm}
		\label{thm:vertex_operators}
		\begin{itemize}
			\item \cite[Theorem 8]{EtingofVarchenkoExchange} For generic $\lambda$ and $\mu$, there is a natural isomorphism
			\begin{equation}\label{E:expectation-value-iso}
				\Hom_{\g} (M_{\lambda},M_{\mu} \otimes V) \cong V[\lambda - \mu],
			\end{equation}
			such that for any $v\in V[\lambda-\mu]$, there is a unique intertwining operator $\Phi_{\lambda}^v$ satisfying
			\begin{equation}\label{E:Phi-triangular}
				\Phi_{\lambda}^v(m_{\lambda}) \in m_{\mu} \otimes v + \bigoplus_{\nu < \mu} M_{\mu} [\nu] \otimes V.
			\end{equation}
			
			\item There exists a universal element $\Xi(\lambda)$ in a completion of  $\U(\n_-) \otimes \U(\b_+)^{\gen}$, such that 
                \[
                \Phi^v_{\lambda}(m_{\lambda}) = \Xi(\lambda) (m_{\mu} \otimes v).
                \]
		\end{itemize}
		\begin{proof}
			The second assertion follows directly from \cite[Proposition 2.1]{EtingofStyrkas}: for every $\mu$, the singular vectors in $M_{\mu} \otimes V$ satisfying \eqref{E:Phi-triangular} can be obtained by applying the inverse of the Shapovalov form. Since the matrix coefficients of the latter are rational functions in $\mu$ and $\mu = \lambda - h|_V$ in the notations of \cref{def:dynamical_notations}, we conclude. 
		\end{proof}
	\end{thm}
	
	\begin{example}
		\label{example:sl2_extremal_vectors}
		Let $\g = \sl_2$. Consider the series
		\[
		\Xi(\lambda) = \sum\limits_{k=0}^{\infty} \frac{(-1)^k}{k!} f^k \otimes e^k \frac{1}{\prod_{i=0}^{k-1} (\lambda - h^{(2)}-i)} = 1\otimes 1 + \sum_{k=1}^{\infty}(-1)^k\frac{f^{k}}{k!}\otimes \frac{e^{k}}{k!} \bigg(\prod_{i=0}^{k-1}\frac{i+1}{\lambda - h^{(2)}-i} \bigg).
		\]
		One can easily check that for a homogeneous element $v\in V$ of weight $\lambda - \mu$, the action of $\Xi(\lambda)$ on $m_{\mu} \otimes v$ is well-defined, moreover,
		\[
		e\cdot \Xi(\lambda)(m_{\mu} \otimes v) = 0 \quad \text{and} \quad 
		\Xi(\lambda)(m_{\mu}\otimes v)\in m_{\mu} \otimes v + \bigoplus_{\nu < \mu} M_{\mu} [\nu] \otimes V.
		\]
		Thus, the map $m_{\lambda}\mapsto \Xi(\mu)(m_{\mu}\otimes v)$ defines an intertwiner $\Phi_{\lambda}^v \colon M_{\lambda} \rightarrow M_{\mu} \otimes V$. 
	\end{example}
	
	Let $W$ be another finite-dimensional representation of $\g$ and $w\in W[\mu - \nu]$. The composition
	\[
	M_{\lambda} \xrightarrow{\Phi_{\lambda}^v} M_{\mu} \otimes V \xrightarrow{\Phi_{\mu}^w \otimes \id_V} M_{\nu} \otimes W \otimes V
	\]
	defines an intertwiner. Assuming the weights $\lambda$, $\mu$, and $\nu$ are generic, this intertwiner can be identified with $\Phi_{\lambda}^u$ for some $u \in (W\otimes V)[\lambda - \nu]$.
	
	\begin{defn}\cite[Subsection 2.4]{EtingofVarchenkoExchange}
		\label{def:fusion_gg}
		The \defterm{fusion operator} $J_{W,V}(\lambda) \colon W \otimes V \rightarrow W \otimes V$ is uniquely defined by
		\[
		(\Phi_\mu^w \otimes \id_V) \circ \Phi_{\lambda}^v = \Phi_{\lambda}^{J_{W,V}(\lambda)(w\otimes v)}
		\]
		for homogeneous $v\in V, w\in W$, such that $\wt (v) = \lambda - \mu$, and extended by linearity to $W \otimes V$.
	\end{defn}   
	
	\begin{thm}
		\label{thm::fusion_operator_properties}
		The fusion operator $J_{W,V}(\lambda)$ has the following properties:
		\begin{itemize}
			\item {\cite[Theorem 14]{EtingofVarchenkoExchange}} For a triple $U,V,W$ of finite-dimensional $\g$-representations, it satisfies the \defterm{dynamical twist} equation
			\begin{equation*}
				J_{V \otimes W, U} (\lambda) \circ (J_{V,W} (\lambda - h|_U) \otimes \id_U) = J_{V,W\otimes U} \circ (\id_V \otimes J_{W,U}(\lambda)).
			\end{equation*}
			
			\item {\cite[Section 8]{EtingofSchifmannLecturesDynamical}} There exists a universal invertible element $J(\lambda)$ in a completion of $\U(\n_-) \otimes \U(\b)^{\gen}$ whose action on any $W \otimes V$ is given by $J_{W,V}(\lambda)$. 
		\end{itemize}
	\end{thm}

\subsection{ABRR equation}

    	The universal element $J(\lambda)$ can be computed inductively the so-called ABRR equation, see \cite{ABRR}.

	\begin{defn}
		If $\lambda \in \h^*$, then we write $\overline{\lambda}\in \h$, to denote the element of the Cartan subalgebra such that $\mu(\overline{\lambda}) = (\mu, \lambda)$, for $\mu\in \h^*$. 
	\end{defn}
	
	\begin{example}
		If $\mathfrak{g}= \mathfrak{sl}_2$, with simple root $\alpha$ and fundamental weight $\omega = \alpha/2$, then $\overline{\alpha} = h$ and $\overline{\omega} = \frac{h}{2}$. 
	\end{example}
	
	Let $\{x_i\} \subset \h$ be an orthonormal basis.
	
	\begin{defn}
		For $\lambda \in \h^*$, define $\gamma(\lambda) = \overline{\lambda} + \overline{\rho} - \frac{1}{2}\sum_i x_i^2\in \h$.
	\end{defn}
	
	If $\mu\in \h^*$, then $\gamma(\lambda)|_{V[\mu]} = (\lambda+\rho, \mu) - \frac{(\mu, \mu)}{2}$ for all $\mathfrak{g}$ modules $V$.
	
	\begin{thm}[{\cite{ABRR}}]
		\label{thm:abrr}
		The universal fusion operator $J(\lambda)$ is a unique element in a completion of $\U(\n_-) \otimes \U(\b_+)^{\gen}$ satisfying the \defterm{ABRR equation}
		\[
		[J(\lambda), 1\otimes \gamma(\lambda)] = \left(\sum\limits_{\alpha \in R_+} F_{\alpha} \otimes E_{\alpha}\right) J(\lambda).
		\]
	\end{thm}
	
	\begin{example}
		\label{example:abrr_sl2}
		Following \cite[Section 8]{EtingofSchifmannLecturesDynamical}, one can construct the universal dynamical twist for $\g=\sl_2$ as follows. Set $J(\lambda) = 1 + \sum_{n\geq 1} J^{(n)}(\lambda)$, where $J^{(n)}(\lambda)$ lies in the corresponding weight components $\U(\n_-)[-2n] \otimes \U(\b_+)^{\gen}[2n]$. The ABRR equation reads
		\[
		\frac{1}{2} [1 \otimes ((\lambda+1) h - \frac{h^2}{2}, J(\lambda)-1] = -(f\otimes e)J(\lambda),
		\]
		which gives the recurrence relation
		\[
		1 \otimes ((\lambda+1)n - nh +n^2)J^{(n)}(\lambda) = (-f\otimes e)J^{(n-1)},
		\]
		so that
		\[
		J^{(n)}(\lambda) = \frac{(-1)^n}{n!} f^n \otimes \frac{1}{(\lambda - h +n+1)\ldots (\lambda -h + 2n)} e^n.
		\]
	\end{example}

	\section{Boundary dynamical fusion}
	\label{sect:boundary_fusion}

	In this section, we study a boundary analog of the dynamical fusion construction of \cref{sect:dynamical_twist}.
	
	\subsection{Boundary fusion operator}
	
	Let $X$ be an integrable $\k$-representation. According to \cite{ReshetikhinStokman}, it turns out that a boundary dynamical analog of the intertwiners of \cref{sect:dynamical_twist} is the space of $\k$-intertwiners $\Hom_{\k} (M_{\lambda},X)$. For instance, the following result is similar to \cref{thm:vertex_operators}.
	\begin{prop}\cite[Proposition 5.4]{ReshetikhinStokman}
		\label{prop:boundary_intertwiners}
		The natural map
		\begin{equation}
			\label{eq::boundary_intertwiners}
			\Hom_{\k}(M_{\lambda},X) \xrightarrow{\sim} X, \qquad \phi \mapsto \phi(m_{\lambda}).
		\end{equation}
		is an isomorphism with the inverse $X \ni \xi \mapsto \phi^{\xi}_{\lambda}$ uniquely defined by $\phi^{\xi}_{\lambda}(m_{\lambda}) = \xi$. 
		\begin{proof}
			The Iwasawa decomposition implies $\g = \k \oplus \h\oplus \mathfrak{n}$, so the PBW theorem implies that the Verma module $M_{\lambda}$ is a free $\U(\k)$-module of rank one generated by $m_{\lambda}$. 
		\end{proof}
	\end{prop}
	
	\begin{remark}\label{R:f-phi-xi}
		If $f\in \End_{\k}(X)$, then $f\circ \phi_{\lambda}^{\xi} = \phi_{\lambda}^{f(\xi)}$, since both are elements of $\Hom_{\k}(M_{\lambda}, X)$ sending $m_{\lambda}$ to $f(\xi)$.
	\end{remark}
	
	\begin{example}
		\label{example:sl2_hom_verma}
		Let us describe explicitly the isomorphism of \eqref{eq::boundary_intertwiners} in the case of $\g = \sl_2$. We learned this from \cite[Subsection 5.5]{ReshetikhinStokman}, but since we use slightly different notations and conventions, we repeat the construction here. 
		
		Let $X$ be an integrable $\so_2$-representation and $\xi \in X$ some element. Consider the intertwiner $\phi^{\xi}_{\lambda}$. Define
		\[
		\xi_k := \phi^{\xi}_{\lambda} \left( \frac{f^k}{k!}\cdot m_{\lambda} \right).
		\]
		For instance, $\xi_0 = \xi$ and $\xi_1 = b\xi$. It follows that the vectors $\xi_k$ satisfy the following recurrence relation:
		\[
		b\xi_k = (k+1)\xi_{k+1} - (\lambda-k+1)\xi_{k-1}.
		\]
		In particular, there exist polynomials $p_{k}(b,\lambda) \in \U(\k)$ satisfying the recurrence relation
		\begin{equation}
			\label{eq::mp_pols}
			p_{k+1}(b,\lambda) = \frac{bp_{k}(b,\lambda) + (\lambda-k+1)p_{k-1}(b,\lambda)}{k+1},
		\end{equation}
		with initial conditions $p_0(b,\lambda) = 1$ and $p_1(b,\lambda)=b$ such that $\xi_k = p_{k}(b,\lambda)\xi$. These polynomials are closely related to the \emph{Meixner--Pollaczek polynomials}, see \cite[(5.17)]{ReshetikhinStokman}. 
	\end{example}
	
	Likewise, there is an analog of the fusion operator of \cref{def:fusion_gg}. Namely, let $V$ be an integrable representation of $\sl_2$. Consider a weight vector $v\in V$ and some element $\xi \in X$. The tensor product $X \otimes V$ has a natural $\k$-module structure, so by \cref{prop:boundary_intertwiners} we have a canonical identification
	\[
	\Hom_{\k} (M_{\lambda}, X \otimes V) \cong X \otimes V.
	\]
	Now let $\lambda,\mu$ be generic weights. Consider a weight vector $v\in V[\lambda - \mu]$ and an element $\xi \in X$. It follows from \cref{thm:vertex_operators} and \cref{prop:boundary_intertwiners} that any $\k$-intertwiner in $\Hom_{\k} (M_{\lambda}, X \otimes V)$ can be uniquely presented as the composition
	\[
	M_{\lambda} \xrightarrow{\Phi_{\lambda}^v} M_{\mu} \otimes V \xrightarrow{\phi^{\xi}_{\mu} \otimes \id_V} X \otimes V.
	\]

    The following definition first appears in \cite[Lemma 6.22]{ReshetikhinStokman}.
 
	\begin{defn}
		\label{def:fusion_gk}
		The \defterm{boundary fusion operator} $G_{X,V}(\lambda)$ is uniquely defined by
		\[
		(\phi_{\mu}^{\xi} \otimes \id_V) \circ \Phi_{\lambda}^v = \phi_{\lambda}^{G_{X,V}(\lambda)(\xi \otimes v)}
		\]
		for homogeneous $v\in V$ and extended by linearity to $X \otimes V$. 
	\end{defn}

	The next proposition is a boundary analog of \cref{thm::fusion_operator_properties}.
	
	\begin{prop}
		The boundary fusion operator has the following properties:
		\begin{itemize} 
			\item \cite[Corollary 6.23]{ReshetikhinStokman} Let $V,W$ be finite-dimensional $\g$-modules and $X$ an integrable $\k$-module. Then
			\[
			G_{X\otimes V,W}(\lambda) \circ (G_{X,V}(\lambda - h|_W) \otimes \id_W) = G_{X,V\otimes W} (\lambda) \circ (\id_X \otimes J_{V,W}(\lambda)).
			\]
			
			\item Using \cref{defn:generic_versions}, there exists a universal invertible element $G(\lambda)$ in a certain completion of $\U(\k) \otimes \U(\b)^{\gen}$ whose action on any $X \otimes V$ is given by $G_{X,V}(\lambda)$. 
		\end{itemize}
		\begin{proof}
			For the proof of the first statement,  given suitable $\mu,\nu\in \h^*$, consider the following commutative diagram.
			\[
			\xymatrix@R+3pc@C+8pc{
				M_{\lambda} \ar[r]^-{\Phi_{\lambda}^w}  \ar[dr] & M_{\mu} \otimes W \ar[d]^-{\phi_{\mu}^{G_{X,V}(\mu)(\xi \otimes v)}\otimes \id_W} \ar[r]^-{\Phi_{\mu}^v \otimes \id_W} & M_{\nu} \otimes V \otimes W \ar[dl]^-{\phi_{\nu}^{\xi}\otimes \id_W} \\
				& X \otimes V \otimes W &
			}
			\]
			The left-hand side of the formula corresponds to the inner left triangle, while the right-hand side corresponds to the outer right triangle. 
			
			To prove the second statement, let $X$ be an integrable $\k$-module and $V$ a finite-dimensional $\g$-representation. Recall that in the notations of \cref{def:fusion_gk}, we have
			\[
			G_{X,V}(\xi \otimes v) = (\phi_{\mu}^{\xi} \otimes \id_V) \Phi_{\lambda}^v(m_{\lambda}). 
			\]
			By \cref{thm:vertex_operators}, the singular vector $\Phi_{\lambda}^v(m_{\lambda}) \in M_{\mu} \otimes V$ can be computed using the universal element $\Xi(\lambda)$ in a completion of $\U(\n_-) \otimes \U(\b_+)^{\gen}$. Observe that $\phi^{\xi}_{\mu}$ comes from the action of the universal homomorphism given by the composition
			\[
			\U(\n_-) \xrightarrow{\act_{\n_-}} M_{\mu} \xrightarrow{\act_{\k}^{-1}} \U(\k),
			\]
			where $\act_{\n_-}$ is defined by $\U(\n_-) \ni x \mapsto x\cdot m_{\mu}$, similarly for $\act_{\k}$. The matrix coefficients are element of $\mathcal{M}(\h^*)$ that can be rearranged to the second tensor factor. Invertibility follows from upper-triangularity in the second tensor component with respect to the $\h$-weight grading.
		\end{proof}
	\end{prop}
	
	\begin{example}
		\label{example:fusion_gk_sl2}
		Let $\g = \sl_2$. According to \cref{example:sl2_extremal_vectors}, the boundary fusion operator is given by
		\[
		G_{X,V}(\lambda) (\xi \otimes v) = \sum\limits_{k=0}^{\infty} (-1)^k \phi^{\xi}_{\lambda - h|_V} \left( \frac{f^k}{k!} m_{\lambda - h|_V} \right) \otimes e^k \frac{1}{\prod_{i=0}^{k-1}(\lambda - h|_V -i)} \cdot v. 
		\]
		where we use the normal ordering on the first tensor component: the operator $h|_V$ acts \emph{first}. Using \cref{example:sl2_hom_verma}, we obtain
		\begin{equation}
			\label{eq::boundary_fusion_sl2}
			G_{X,V}(\lambda) =  \sum\limits_{k=0}^{\infty} (-1)^k  p_k(b,\lambda - h|_V) \otimes e^k \frac{1}{\prod_{i=0}^{k-1}(\lambda - h|_V -i)}.
		\end{equation}
		The elements $p_k(b,\lambda - h|_V)$ are polynomials in $b$ with coefficients in $\mathcal{M}(\h^*)$ that can be rearranged to the second tensor factor.
		
		For instance, for the vector representation $V = V(1)$, the boundary fusion operator is given by the $\U(\k)$-valued matrix
		\begin{equation}
			\label{eq::fusion_gk_v1}
			\begin{bmatrix}
				1 & - \frac{b}{\lambda+1} \\ 0 & 1
			\end{bmatrix}.
		\end{equation}
		Another example is when $X = \C$ is the trivial representation of $\k$. \cref{example:sl2_hom_verma} implies that 
		\[
		G_{\C,V}(\lambda) = \sum\limits_{k = 0}^{\infty} \frac{1}{(2k)!!} \cdot  e^{2k} \frac{1}{\prod_{i=0}^{k-1} (\lambda - h|_V - 2i-1)},
		\]
		where $(2k)!!$ is the double factorial. Observe its similarity to the quantum $K$-matrix in rank one \cite[(3.39)]{DobsonKolb}. 
	\end{example}

	\begin{remark}
		The boundary dynamical twist equation can be written as
		\begin{equation}\label{E:re-write-boundary-dyn-twist}
			G_{X, V}(\lambda - h|_W)\otimes \id_W = G_{X\otimes V, W}(\lambda)^{-1}\circ G_{X, V\otimes W}(\lambda) \circ \bigg(\mathrm{id}_{X}\otimes J_{V, W}(\lambda)\bigg)
		\end{equation}
		which is useful below. In particular, this allows one to remove the $\lambda - h|_W$ dependence.
		
	\end{remark}

	\subsection{Boundary ABRR equation}
	
	The boundary fusion operator can be constructed inductively using an analog of the ABRR equation of \cref{thm:abrr}; in what follows, we will use the notations of the theorem.
	
	Let $C = \sum_{i}x_i^2 + \sum_{\alpha\in R_+}(E_{\alpha}F_{\alpha} + F_{\alpha}E_{\alpha})$ be the quadratic Casimir element of $\g$. Recall that $C$ acts on $M_{\lambda}$ as $(\lambda, \lambda + 2\rho)$. Note that we can rewrite $C$ as
	\[
	C = \sum_i x_i^2 + 2\overline{\rho} + 2\sum_{\alpha\in R_+} F_{\alpha} E_{\alpha}.
	\]       
	
	\begin{notation}
		If $x\in \U(\mathfrak{g})$, then we write $x^{(1)}:=x\otimes 1\in \U(\mathfrak{g})\otimes \U(\mathfrak{g})$ and $x^{(2)}:=1\otimes x \in \U(\mathfrak{g})\otimes \U(\mathfrak{g})$
	\end{notation}
	
	\begin{prop}
		The universal boundary fusion operator $G(\lambda)$ is a unique solution in a completion of $\U(\k) \otimes \U(\b_+)^{\gen}$ to the \defterm{boundary ABRR equation}
		\begin{equation}\label{E:boundary-ABRR}
			[G(\lambda),1\otimes \gamma(\lambda)] = \left(\sum\limits_{\alpha\in R_+} B_{\alpha} \otimes E_{\alpha} - 1 \otimes E_{\alpha}^2\right) G(\lambda).
		\end{equation}
		\begin{proof}
			Uniqueness can be proved using the same arguments as in \cite[Theorem 8.1]{EtingofSchifmannLecturesDynamical}. Let us show that the operators $G_{X,V}(\lambda)$ satisfy the boundary ABRR equation for every finite-dimensional $\g$-representation $V$ and integrable $\k$-representation $X$; it will imply the statement for $G(\lambda)$. 
			
			Let $v\in V$ be a homogeneous vector and consider the element $(C\otimes 1)\Phi_{\lambda}^v(m_{\lambda}) \in M_{\mu} \otimes V$. On the one hand, the element is equal to 
			\begin{equation}
				\label{eq_prop::casimir_one}
				(\mu,\mu + 2\rho) \cdot \Phi_{\lambda}^v(m_{\lambda}), \qquad \text{where} \qquad \mu = \lambda - \wt v
			\end{equation}
			On the other hand, it is equal to 
			\begin{equation}\label{E:Ctimes1-cdot-Phi}
				\left(\sum_i  (x_i^{(1)})^2 + 2\rho^{(1)} + 2\sum_{\alpha\in R_+} (B_{\alpha}^{(1)} E_{\alpha}^{(1)} + (E_{\alpha}^{(1)})^2) \right) \cdot \Phi_{\lambda}^v(m_{\lambda}) \in M_{\mu} \otimes V.
			\end{equation}
			Since $E_{\alpha}\cdot m_{\lambda} = 0$, it follows from $\Phi_{\lambda}^v$ being an intertwiner that $E_{\alpha}^{(1)}\cdot \Phi_{\lambda}^v(m_{\lambda}) = -E_{\alpha}^{(2)}\cdot \Phi_{\lambda}^v(m_{\lambda})$. Likewise, writing $\lambda_i := \lambda(x_i)$, we have 
			\[
			x_i^{(1)}\cdot \Phi_{\lambda}^v(m_{\lambda}) = (\lambda_i - x_i^{(2)})\cdot \Phi_{\lambda}^v(m_{\lambda}) \qquad \text{and} \qquad  \overline{\rho}^{(1)}\cdot \Phi_{\lambda}^v(m_{\lambda}) = \big((\lambda,\rho) - \overline{\rho}^{(2)}\big)\cdot \Phi_{\lambda}^v(m_{\lambda}).
			\]
			Thus, \eqref{E:Ctimes1-cdot-Phi} is equal to
			\[
			\left(\sum_i (\lambda_i^2 - 2\lambda_i x_i^{(2)} + (x_i^{(2)})^2) + 2(\lambda,\rho) - 2\overline{\rho}^{(2)} + 2\sum_{\alpha\in R_+} (- B_{\alpha}^{(1)} E_{\alpha}^{(2)} + (E_{\alpha}^{(2)})^2) \right)\cdot \Phi_{\lambda}^v(m_{\lambda})
			\]
			which we can rewrite as
			\begin{equation}
				\label{eq_prop::casimir_two}
				\left((\lambda, \lambda) -2\overline{\lambda}^{(2)} +\sum_i (x_i^{(2)})^2 + 2(\lambda,\rho) - 2\overline{\rho}^{(2)} + 2\sum_{\alpha\in R_+} (- B_{\alpha}^{(1)} E_{\alpha}^{(2)} + (E_{\alpha}^{(2)})^2) \right)\cdot \Phi_{\lambda}^v(m_{\lambda})
			\end{equation}
			
			Now, let $X$ be an integrable $\k$-module and let $\xi\in X$. We apply $\phi_{\mu}^{\xi}\otimes \id_V$ to \eqref{eq_prop::casimir_one} and \eqref{eq_prop::casimir_two}. Using the $\U(\k)$-intertwining property of $\phi_{\mu}^{\xi}$ and equating \eqref{eq_prop::casimir_one} with \eqref{eq_prop::casimir_two}, we get the equality
			\begin{align*}
				(\mu, \mu+2\rho)\phi_{\mu}^{G_{X,V}(\lambda)\cdot \xi\otimes v} &= \left((\lambda, \lambda) +2(\lambda, \rho)\right)\phi_{\mu}^{G_{X,V}\cdot \xi\otimes v}\\
				&+ \left(-2\otimes \overline{\lambda} + \sum( 1\otimes x_i^2) - 2\otimes \overline{\rho}\right)\phi_{\mu}^{G_{X,V}\cdot \xi\otimes v} \\
				&+2\sum_{\alpha\in R_+} (-B_{\alpha}\otimes E_{\alpha} + 1\otimes E_{\alpha}^2)\phi_{\mu}^{G_{X,V}(\lambda)\cdot \xi\otimes v}.
			\end{align*}
			After multiplying this equality by $-\frac{1}{2}$, rearrangement and simplification of terms yields
			\begin{equation}\label{E:ABRR-second-to-last}
				\begin{split}
					\left((\lambda+\rho, \lambda-\mu) - \frac{(\lambda-\mu, \lambda-m)}{2}\right)\phi_{\mu}^{G_{X,V}(\lambda)\cdot \xi\otimes v} &- 1\otimes \left(\overline{\lambda} - \frac{1}{2}\sum_ix_i^2 + \overline{\rho}\right)\phi_{\mu}^{G_{X,V}\cdot \xi\otimes v} \\
					&=\sum_{\alpha\in R_+} (B_{\alpha}\otimes E_{\alpha} - 1\otimes E_{\alpha}^2) \phi_{\mu}^{G_{X,V}(\lambda)\cdot \xi\otimes v}.
				\end{split}
			\end{equation}
			Since $\wt(v)= \lambda - \mu$, we also have
			\[
			\gamma(\lambda)\cdot v = (\lambda+\rho, \lambda-\mu) - \frac{(\lambda-\mu, \lambda-\mu)}{2}\cdot v.
			\]
			Applying both sides of equation \eqref{E:ABRR-second-to-last} to $m_{\lambda}$, we conclude that
			\[
			[G_{X,V}(\lambda),1\otimes \gamma(\lambda)]\cdot \xi\otimes v = \sum_{\alpha\in R_+} (B_{\alpha} \otimes E_{\alpha} - 1 \otimes E_{\alpha}^2) \cdot G_{X,V}(\lambda)\cdot \xi\otimes v,
			\]
			as required.
		\end{proof}
	\end{prop}

	\begin{example}
		Let $\g=\sl_2$. Recall that the boundary fusion operator $G(\lambda)$ is given by the formula \eqref{eq::boundary_fusion_sl2}. Let us re-derive it using the ABRR equation.
		
		By rearranging the $\mathcal{M}(\h^*)$-coefficients to the first tensor factor, let us write the universal $G(\lambda)$ operator in the form 
		\[
		G(\lambda) = \sum\limits_{k=0}^{+\infty} G^{(k)}, \quad G^{(k)}(\lambda) = (-1)^k q_k(b,\lambda,h^{(2)}) \otimes e^k\cdot \frac{1}{\prod_{i=0}^k (\lambda - h - i)}.
		\]
		for some polynomials $q_k(b,\lambda,h^{(2)})$ in $b\in \U(\k)$ with values in $\mathcal{M}(\h^*)$ such that $q_0 = 1$, and $h^{(2)}$ is the weight operator on the second tensor component. Proceeding as in \cref{example:abrr_sl2}, we derive the following equation:
		\[
		1 \otimes ((\lambda+1)(k+1) - (k+1)h - (k+1)^2) G^{(k+1)} = -(b\otimes e) G^{(k)} + (1 \otimes e^2) G^{(k-1)}.
		\]
		Since the second component of $G^{(k+1)}$ is homogeneous of weight $2(k+1)$, we also have 
		\[
		1 \otimes ((\lambda+1)(k+1) - (k+1)h + (k+1)^2) G^{(k+1)} = G^{(k+1)} (1 \otimes ((\lambda+1)(k+1) - (k+1)h - (k+1)^2)).
		\]
		It gives the following equation on the polynomials $q_k$:
		\[
		q_{k+1} = \frac{b\cdot q_k + (\lambda - h^{(2)} - k +1) q_{k-1}}{k+1}.
		\]
		Observe that this is the same inductive relation as for the Meixner--Pollaczek polynomials $p_k(b,\lambda - h^{(2)})$ of \eqref{eq::mp_pols}. Therefore, the ABRR equation gives the operator
		\[
		G(\lambda) = \sum\limits_{k=0}^{\infty} (-1)^k p_{k}(b,\lambda - h^{(2)}) \otimes e^k \frac{1}{\prod_{i=0}^{k-1} (\lambda - h - i)},
		\]
		as we derived in \cref{example:fusion_gk_sl2}. 
	\end{example}

    \begin{remark}
        Jasper Stokman shared with us unpublished notes in which he derives a $q$-analog of the boundary ABRR equation in the case of $U_q^{\iota}(\so_2)\subset U_q(\sl_2)$ \cite{Stokman}. Stokman informed us this will appear in work in preparation by himself, Valentin Buciumas, and Hadewijch de Clercq.
    \end{remark}
		
		\section{Background on dynamical Weyl group}
		\label{subsect:sl2_dw_diag}
		The dynamical fusion operators give rise to a dynamical quantum group (Hopf algebroid) \cite{EtingofVarchenkoExchange}. It is observed in \cite{EtingofVarchenkoDynamicalWeyl} that there is a dynamical Weyl group which is analogous to the quantum Weyl group for the usual Drinfeld--Jimbo quantum group. In this section, we recall the construction of the dynamical Weyl group operators. In what follows, we use definitions and notations of \cref{sect:dynamical_twist}. In particular, $\g$ is a simple Lie algebra.

		\subsection{Dynamical Weyl group}
		
		As in \cref{sect:dynamical_twist}, let $V$ be a finite-dimensional $\g$-representation, and consider the space $\Hom_{\g}(M_{\lambda}, M_{\mu} \otimes V)$. Observe that the intertwiner $\Phi_{\lambda}^v$ from Theorem \ref{thm:vertex_operators} is defined for all sufficiently dominant integral weights.
  
        For a simple reflection $s=s_{\alpha}$, such that $\lambda(\alpha^{\vee})\in \mathbb{Z}_{\ge 0}$, the vector $f_{\alpha}^{\lambda(\alpha^{\vee})+1} m_{\lambda}$ is a singular vector in $M_{\lambda}$ of weight $s\cdot \lambda$, where the dot action is as in equation \ref{E:dot-action}. This implies there is a unique $\g$-intertwiner 
        \[
        i_{s\cdot \lambda}^{\lambda}:M_{s\cdot \lambda}\rightarrow M_{\lambda}\qquad \text{such that} \qquad m_{s\cdot \lambda}\mapsto \frac{f_{\alpha}^{\lambda(\alpha^{\vee})+1}}{(\lambda(\alpha^{\vee})+1)!}m_{\lambda}.
        \]
        Define $A^{\g,s}_{V}(\lambda)\in \End(V)$ by
		\[
		\Phi_{\lambda}^v\circ i_{s\cdot \lambda}^{\lambda} = \Phi_{s\cdot \lambda}^{A_{V}^{\g, s}(\lambda)\cdot v},
		\]
		for all $v\in V$.

		\begin{example}
			\label{example:dw_sl2_vect_rep}
			Let $\g=\sl_2$ and $\alpha$ be the simple root with corresponding simple reflection $s$. For the vector representation $V(1)$ with weight basis $\{v_1, v_{-1}\}$, we have
			\[
			A_{V(1)}^{\sl_2,s}(\lambda) = \begin{bmatrix}
				0 & -\frac{\lambda+2}{\lambda+1} \\ 1 & 0
			\end{bmatrix},
			\]
			where we identified $\h^* \cong \C$ via the fundamental weight, see \cite[Lemma 5]{EtingofVarchenkoDynamicalWeyl} for the $q$-version.
		\end{example}

		\begin{notation}
			We will denote $A^{\g,\alpha_i}_V(\lambda) := A^{\g,s_i}_{V}(\lambda)$.
		\end{notation}
		
		The following remark recalls some technical results about embeddings of Verma modules, which allow us to extend this construction to arbitrary elements $w\in W$. 
		
		\begin{remark}
			Consider an element $w\in W$ and let $w = s_{i_k} s_{i_{k-1}} \ldots s_{i_1}$ be a reduced decomposition. Set $\alpha^1 = \alpha_{i_k}$ and $\alpha^j = (s_{i_1} \ldots s_{i_{j-1}})(\alpha_{i_j})$, and let $n_j = \langle \lambda + \rho, (\alpha^j)^{\vee}\rangle$. Then the vector 
			\[
			m_{w\cdot \lambda}^{\lambda} := \frac{f_{\alpha_{i_k}}^{n_k}}{n_k !} \ldots \frac{f_{\alpha_{i_1}}^{n_1}}{n_1!} m_{\lambda}
			\]
			is singular and does not depend on the reduced decomposition of $w$, see \cite[Lemma 4]{TarasovVarchenko}. There is a corresponding embedding of Verma modules $M_{w\cdot \lambda} \hookrightarrow M_{\lambda}$.
		\end{remark}
		
		\begin{prop}\cite[Lemma 5]{TarasovVarchenko}
			For every $v\in V[\lambda-\mu]$, there is a unique vector $A^{\g,w}_{V}(\lambda)v \in V[w(\lambda-\mu)]$ such that
			\[
			\Phi_{\lambda}^v (m^{\lambda}_{w\cdot \lambda}) \in m^{\mu}_{s \cdot \mu} \otimes A^{\g,w}_{V}(\lambda)v + \bigoplus_{\nu < w\cdot \lambda} M_{\mu}[\nu] \otimes V.
			\]
			It defines a linear operator $A^{\g,w}_{V}(\lambda) \colon V \rightarrow V$ depending rationally on $\lambda$.
		\end{prop} 
		
		In fact, these operators are completely defined by simple reflections $s_{\alpha_i}$.
		\begin{prop}\cite[Lemma 6]{TarasovVarchenko}
			\label{prop:dynamical_weyl_group_braid_relations}
			The operators $\{A^{\g,s_i}_V(\lambda) \}_{\alpha_i\in \Pi}$ satisfy the \emph{braid relations}
			\begin{equation}\label{E:dw-braid-relns}
				\ldots A^{\g,s_{i}}_V(s_{j}s_{i}\cdot \lambda)A^{\g,s_{j}}_V(s_{i}\cdot \lambda) A^{\g,s_{i}}_V(\lambda) = \ldots A^{\g,s_{j}}_V(s_{i}s_{j}\cdot \lambda)A^{\g,s_{i}}_V(s_{j}\cdot \lambda) A^{\g,s_{j}}_V(\lambda),
			\end{equation}
			where the number of terms on each side of equation \ref{E:dw-braid-relns} is equal to the order of $(s_{i}s_{j})$ in $W$. In particular, if $w_1,w_2\in W$ with $l(w_1 w_2) = l(w_1) + l(w_2)$, then 
			\[
			A^{\g,w}_V (\lambda) = A^{\g,w_1}_V (w_2 \cdot \lambda) A^{\g,w_2}_V (\lambda).
			\]
		\end{prop}
		
		We will need the following tensor property of the dynamical Weyl group operator.
		\begin{prop}\label{P:weyl-fusion-compatible}
			\cite[Lemma 4]{EtingofVarchenkoDynamicalWeyl}
			\label{prop:dw_tensor_property}
			For a pair $U,V$ of finite dimensional representations of $\g$ and a simple reflection $s\in W$, we have
			\begin{equation}\label{E:dw_tensor-property}
				A^{\g,s}_{U\otimes V} = J_{U,V}(s\cdot \lambda) \circ \big(A^{\g,s}_U(\lambda - h|_V) \otimes A^{\g,s}_V(\lambda)\big) \circ J_{U,V}(\lambda)^{-1}.
			\end{equation}
		\end{prop}
		
		
		In what follows, we will also need a power series presentation of the dynamical Weyl group operators. For a root $\alpha$, consider the series 
		\begin{equation}
			\label{eq::b_operator}
			\mathcal{B}^{\g}_{\alpha}(t) = \sum_{k\geq 0} f_{\alpha}^k e_{\alpha}^k \frac{1}{k!} \prod_{j=0}^{k-1} \frac{1}{t - h_{\alpha} - j},
		\end{equation}
		see \cite[equation (5)]{TarasovVarchenko}. Let $\alpha_i$ be a simple root. Recall that $\tilde{s}_{\alpha_i} = \exp(-e_{\alpha_i})\exp(f_{\alpha_i}) \exp(-e_{\alpha_i})$.
		
		\begin{thm}\cite[Theorem 8]{TarasovVarchenko}
			\label{prop:dynamical_weyl_group_b_operator}
			If $V$ is an finite-dimensional representation of $\mathfrak{g}$, then
			\[
			A^{\g,s_i}_V(\lambda) = \tilde{s}_{\alpha_i} \cdot \mathcal{B}^{\g}_{\alpha_i}(\langle \lambda,\alpha^{\vee}_i \rangle)|_V
			\]
		\end{thm}
		
		The braid relation for the dynamical Weyl group gives a natural generalization of the \emph{Yang-Baxter equation} that we will discuss later in the paper, see \eqref{eq::gl_rmatrix_qybe}.
		\begin{defn}\cite[Definition 2.1]{CherednikQKZ}
			\label{def:generalized_r_matrix}
			Let $V$ be a vector space. Consider a collection $\{G_{\alpha}(\mu)\}_{\alpha\in R^+, \mu \in \C}$ of operators on $V$. Define functions $G^{\alpha}:\h^*\rightarrow \End(V)$ by $G^{\alpha}(\lambda):= G_{\alpha}(\langle \lambda,\alpha^{\vee}\rangle)$. Such a collection is called a \defterm{closed $R$-matrix} if, for every pair of roots $\alpha,\beta$, they satisfy one of the following relations:
			\begin{align*}
				G^{\alpha} G^{\beta} &= G^{\beta} G^{\alpha}, \\
				G^{\alpha} G^{\alpha+\beta} G^{\alpha} &= G^{\beta} G^{\alpha+\beta} G^{\alpha}, \\
				G^{\alpha} G^{\alpha+\beta} G^{\alpha+2\beta} G^{\beta}&=  G^{\beta} G^{\alpha+2\beta} G^{\alpha+\beta} G^{\alpha},\\
				G^{\alpha} G^{3\alpha+\beta} G^{2\alpha + \beta} G^{3\alpha+2\beta} G^{\alpha+\beta} G^{\beta} &= G^{\beta} G^{\alpha+\beta} G^{3\alpha+2\beta} G^{2\alpha + \beta} G^{3\alpha+\beta} G^{\alpha},
			\end{align*}
			depending on the rank 2 root system formed by $\{\alpha,\beta\}$. 
		\end{defn}
		
		\begin{thm}\cite[Section 2.7(VI)]{TarasovVarchenko}
			\label{thm:b_operators_generalized_r_matrix}
			The operators $\{\cB^{\g}_{\alpha}(\lambda - \rho + \frac{1}{2} h|_V)\}$ form a closed $R$-matrix.
		\end{thm}
		
		\section{Boundary dynamical Weyl group}\label{s:boundary-dynamical-Weyl}
		
		\subsection{Boundary dynamical Weyl group: \texorpdfstring{$\mathfrak{so}_2\subset \mathfrak{sl}_2$}{so2 in sl2}}
		
		In this subsection, we construct a boundary analog of the dynamical Weyl group as in \cref{subsect:sl2_dw_diag}.
		
		Recall that if $\langle \lambda, \alpha^{\vee}\rangle \in \mathbb{Z}_{\ge 0}$, then there is a submodule of $M_{\lambda}$ generated by $\frac{f^{\lambda+1}}{(\lambda+1)!}m_{\lambda}$. This determines an injective homomorphism $\iota_{s\cdot \lambda}^{\lambda}:M_{s\cdot \lambda}\rightarrow M_{\lambda}$, where $m_{s\cdot \lambda}\mapsto \frac{f^{\lambda+1}}{(\lambda+1)!}m_{\lambda}$. 
		
		\begin{defn}\label{D:A-so2}
			Let $X$ be an integrable $\so_2$-representation. Consider $\xi \in X$ and the corresponding $\so_2$-intertwiner $\phi_{\lambda}^{\xi} \in \Hom_{\so_2}(M_{\lambda},X)$ as in \cref{prop:boundary_intertwiners}. Assume that $\langle \lambda, \alpha^{\vee}\rangle \in \mathbb{Z}_{\ge 0}$. Define $A_{X}^{\so_2}(\lambda)\in \End_{\C}(X)$ by
			\begin{equation}\label{E:A-so2}
				A^{\so_2}_X(\lambda)\xi := \phi_{\lambda}^{\xi} \left( \frac{f^{\lambda+1}}{(\lambda+1)!} m_{\lambda} \right).
			\end{equation}
		\end{defn}

		\begin{prop}
			Under the hypotheses of Definition \ref{D:A-so2}, we have
			\begin{equation}
				\label{eq::dw_boundary_integer}
				A^{\so_2}_X(\lambda)\xi = \frac{(b-\lambda i)(b-\lambda i+2i) \ldots (b+\lambda i)}{(\lambda+1)!} \xi.
			\end{equation}
			\begin{proof}
				Since $X$ is assumed to be integrable, $b$ acts semisimply on $X$, so it is enough to assume that $X = \C_x$ is a one-dimensional module of $\so_2$ with character $x$ and $\xi=1_x$.
				
				According to \cref{example:sl2_hom_verma}, the action of $A_{\C_x}(\lambda)$ is given by the polynomial $p_{\lambda+1}(x,\lambda)$. It is clear that its degree is $\lambda+1$ and that the coefficient in front of $x^{\lambda+1}$ is equal to $1/(\lambda+1)!$. Therefore, it is enough to prove that the roots of $p_{\lambda+1}(x, \lambda)$ are 
				\begin{equation}
					\label{prop_eq::roots}
					\{\lambda i,\lambda i-2i,\ldots,-\lambda i+2i,-\lambda i\}.
				\end{equation}
				By degree considerations, we are reduced to proving that $p_{\lambda+1}(x,\lambda)=0$ for all $x$ in the set \eqref{prop_eq::roots}. 
				
				Consider the short exact sequence of $\sl_2$-modules
				\[
				0 \rightarrow M_{s\cdot \lambda} \xrightarrow{\iota_{s\cdot \lambda}^{\lambda}} M_{\lambda} \xrightarrow{\pi} V_{\lambda} \rightarrow 0,
				\]
				where $V_{\lambda}$ is the irreducible finite-dimensional representation with highest weight $\lambda$. Let us apply a left exact functor $\Hom_{\k}(-,\C_x)$ to obtain:
				\[
				0 \rightarrow \Hom_{\k}(V_{\lambda},\C_x) \xrightarrow{\pi\circ (-)} \Hom_{\k} (M_{\lambda},\C_x) \xrightarrow{\iota_{s\cdot \lambda}^{\lambda}\circ (-)} \Hom_{\k}(M_{s\cdot \lambda}, \C_x),
				\]
				Note that $M_{s\cdot \lambda}$ is generated by $m_{s\cdot \lambda}$, so $\iota_{s\cdot \lambda}^{\lambda}$ is determined by evaluating on $m_{s\cdot \lambda}$, and that equation \eqref{E:A-so2} shows that $A_{\C_x}^{\so_2}(\lambda)$ is determined by evaluating $\phi_{\lambda}^{1_x}$ on $\iota_{s\cdot \lambda}^{\lambda}(m_{s\cdot \lambda})$. Thus, it suffices to prove that for $x$ in the set \eqref{prop_eq::roots}, the image of $\phi_{\lambda}^{1_x}$, under $\iota_{s\cdot \lambda}^{\lambda}\circ (-)$, is equal to zero in $\Hom_{\k}(M_{s\cdot \lambda}, \C_x)$. By exactness, it is enough to show that $\phi_{\lambda}^{1_x}$ is in the image of $\pi\circ (-)$. 
				
				Since $b=f-e$ has eigenvalues $i$ and $-i$ in the natural representation $V_1\cong\C^2$, $b$ is conjugate to $ih$ inside $\sl_2$ and therefore acts semisimply in the finite dimensional module $V_{\lambda}$ with the eigenvalues \eqref{prop_eq::roots}. Thus, for $x$ in the set \eqref{prop_eq::roots}, we have $\dim \Hom_{\k}(V_{\lambda}, \C_x) = 1$. Also, from Proposition \ref{prop:boundary_intertwiners} we have $\dim \Hom_{\k}(M_{\lambda}, \C_x) = \dim \C_x = 1$. By exactness, $\pi\circ (-)$ is injective, and therefore is surjective.
			\end{proof}
		\end{prop}
		
		\begin{defn}
			Euler's \emph{gamma function} $\Gamma(z)$ is defined as the analytic continuation of the function $z\mapsto \int_0^{\infty}t^{z-1}\exp(-t)dt$, for $z\in \mathbb{R}_{\ge 0}$, to a meromorphic function on $\C$.
		\end{defn}
		
		\begin{remark}
			For us, the most important property of the Gamma function is that 
			\[
			\Gamma(z+1) = z\Gamma(z).
			\]
			The gamma function also satisfies the Euler reflection formula:
			\begin{equation}\label{E:e-r-formula}
				\Gamma(1-z) \Gamma(z) = \frac{\pi}{\sin(\pi z)}.
			\end{equation}
		\end{remark}

		Observe that if $b$ acts on $\xi$ with eigenvalue $x$, we have from formula \eqref{eq::dw_boundary_integer}
		\[
		A^{\so_2}_X(\lambda)\xi = (-2i)^{\lambda+1}\frac{(\frac{ix+\lambda}{2})(\frac{ix+\lambda}{2}-1) \ldots (\frac{ix+\lambda}{2}-\lambda)}{(\lambda+1)!}\xi = \frac{(-2i)^{\lambda+1}}{\Gamma(\lambda+2)}\frac{\Gamma(\frac{ix + \lambda}{2} + 1)}{\Gamma(\frac{ix+\lambda}{2} - \lambda)}\xi.
		\]
		
		\begin{defn}
			Euler's \emph{beta function} is defined by $\B(z_1, z_2):=\frac{\Gamma(z_1)\Gamma(z_2)}{\Gamma(z_1+z_2)}$. 
		\end{defn}
		
		Thus,
		\[
		A^{\so_2}_X(\lambda)\xi = \frac{(-2i)^{\lambda+1}}{\Gamma(\lambda+2)} \frac{\Gamma\left( \frac{ix+\lambda}{2} + 1 \right)}{\Gamma \left( \frac{ix-\lambda}{2} \right)}\xi = \frac{(-2i)^{\lambda+1}}{\lambda+1} \Beta \left( \frac{ix-\lambda}{2},\lambda+1 \right)^{-1}\xi.
		\]
		
		In general, for any meromorphic analytic function $f$, one can define the action of the operator $f(b)$ on $\C_x$ by $f(x)$ and extend it to an arbitrary integrable $\so_2$-module by complete reducibility.
		
		\begin{defn}
			Let $X$ be an integrable $\so_2$-module. The \defterm{boundary dynamical Weyl group} operator is 
			\[
			A^{\so_2}_X(\lambda) := \frac{(-2i)^{\lambda+1}}{\lambda+1} \Beta \left( \frac{ib-\lambda}{2}, \lambda+1 \right)^{-1},
			\]
			where $(-2i)^{\lambda} := \exp \Big(\left( -\frac{i\pi}{2} + \log(2) \right) \cdot \lambda \Big)$.
		\end{defn}
		
		\begin{notation}
			\label{notation:boundary_dynamical_weyl_group_scalar_notation}
			Let $X = \C_x$ be an $\so_2$-representation with basis $1_{x}$, such that $b\cdot 1_x = x1_x$. In what follows, we write
			\begin{equation}\label{E:AxL-notation}
				A(x,\lambda) := \frac{(-2i)^{\lambda}}{\lambda+1} \Beta \left( \frac{ix-\lambda}{2},\lambda+1 \right)^{-1}
			\end{equation}
			to denote the scalar by which $A^{\so_2}_{\C_x}(\lambda)$ acts on $1_x\in \C_x$.
		\end{notation}
		
		Observe that
		\begin{equation}
			\label{eq::boundary_x_shift}
			A(x-2i,\lambda) = \frac{ix+\lambda+2}{ix-\lambda} A(x,\lambda).
		\end{equation}
		
		The next lemma follows from the standard properties of the beta function:
		\begin{lm}
			\label{lm:boundary_dw_shift}
			The boundary dynamical Weyl group operator satisfies the following functional equations:
			\[
			A(x-i,\lambda) = \frac{x-i(\lambda+1)}{\lambda+1} A(x,\lambda-1), \qquad A(x+i,\lambda) = \frac{x + i(\lambda+1)}{\lambda+1} A(x,\lambda-1).
			\]
			In particular,
			\begin{equation}
				\label{eq::boundary_lambda_shift}
				A(x,\lambda+2) = \frac{x^2+\lambda^2}{\lambda(\lambda+1)} A(x,\lambda).
			\end{equation}
		\end{lm}
		
		It follows from equation \eqref{E:e-r-formula} that
		\begin{equation}
			\label{eq::boundary_dyn_weyl_minus_weight}
			A(-x,\lambda) = -\frac{\sin \left(\pi \frac{ix+\lambda}{2} \right)}{\sin \left(\pi \frac{ix-\lambda}{2} \right)} A(x,\lambda)
		\end{equation}
		
		\subsection{Divided power formula}
		
		Recall that for the (non-dynamical) boundary Weyl group, there is a series presentation \eqref{eq::bweyl_series} for integral eigenvalues of $b$. In this subsection, we present an analogous formula in the dynamical case.
		
		\begin{defn}
			Define $a_0(x,\lambda):=\frac{A(-ix, \lambda)}{A(0, \lambda)}$, for $x\in 2\mathbb{Z}$, and $a_1(x,\lambda):=\frac{A(-ix, \lambda)}{A(-i,\lambda)}$, for $x\in 1+2\mathbb{Z}$. 
		\end{defn}
		
		Note that $a_0(0, \lambda) =1$ and $a_1(x, \lambda) =1$. One also finds that
		\begin{equation}\label{E:a(xL)}
			a_0(x,\lambda) = \frac{x+\lambda}{x-2-\lambda}\cdot \frac{x-2+\lambda}{x-4-\lambda}\cdots \frac{2+\lambda}{-\lambda} \qquad \text{and} \qquad a_1(x, \lambda) = \frac{x+\lambda}{x-2-\lambda}\cdot\frac{x-2+\lambda}{x-4-\lambda}\cdots \frac{3+\lambda}{1-\lambda}.
		\end{equation}
		
		Recall the divided powers from Notation \ref{N:divided powers}.
		
		\begin{defn}\label{D:simplified-dynamical-series}
			Define
			\begin{equation}\label{E:s(xL)0}
				s_0(x,\lambda)= 1 - \frac{(\lambda+1)}{(\lambda)}x_0^{(2)} + \frac{(\lambda+1)(\lambda-1)}{(\lambda)(\lambda-2)}x_0^{(4)} - \frac{(\lambda+1)(\lambda-1)(\lambda-3)}{(\lambda)(\lambda-2)(\lambda-4)}x_0^{(6)} \pm \dots, \qquad \text{for $x\in 2\mathbb{Z}$},
			\end{equation}
			and
			\begin{equation}\label{E:s(xL)1}
				s_1(x,\lambda) = x_1^{(1)}- \frac{(\lambda)}{(\lambda-1)}x_1^{(3)} + \frac{(\lambda)(\lambda-2)}{(\lambda-1)(\lambda-3)}x_1^{(5)} - \frac{(\lambda)(\lambda-2)(\lambda-4)}{(\lambda-1)(\lambda-3)(\lambda-5)}x_1^{(7)} \pm \dots \qquad \text{for $x\in 1+2\mathbb{Z}$}.
			\end{equation}    
		\end{defn}
		
		\begin{remark}
			Observe that there is a well-defined limit of both series $a_0(x,\lambda)$ and $a_1(x,\lambda)$ when $\lambda \rightarrow \infty$ which coincides with the non-dynamical series \eqref{eq::bweyl_series}. 
		\end{remark}
		
		The following theorem is a dynamical analog of \cref{prop:bweyl_series}.
		
		\begin{thm}\label{T:simplified-dynamical-series}
			We have the following equalities:
			\[
			a_0(x,\lambda) = s_0(x, \lambda) \qquad \text{and} \qquad a_1(x, \lambda) = s_1(x, \lambda). 
			\]
		\end{thm}
		\begin{proof}
			By \eqref{eq::boundary_dyn_weyl_minus_weight}, we have
			\[
			a_0(-x,\lambda) = a_0 (x,\lambda), \qquad a_1(-x,\lambda)= -a_1(x,\lambda).
			\]
			Observe that the series $s_0(x,\lambda)$ and $s_1(x,\lambda)$ satisfy the same property. Therefore, it is enough to show the statement for positive $x$. 
			
			It follows from equation \eqref{eq::boundary_x_shift} and \cref{lm:boundary_dw_shift} that 
			\[
			a_0 (x+2,\lambda) - a_0 (x,\lambda) = -2\cdot \frac{\lambda+1}{\lambda} a_1(x+1,\lambda-1) 
			\]
			and
			\[
			a_1(x+2,\lambda) - a_1(\lambda) = -2\cdot a_0(x+1,\lambda).
			\]
			In particular, for any positive integral $x$, functions $a_0(x,\lambda)$ and $a_1(x,\lambda)$ can be constructed inductively starting from $a_0(0,\lambda) = a_1(1,\lambda) = 1$. At the same time, it is an exercise to verify that the series $s_0(x,\lambda)$ and $s_1(x,\lambda)$ satisfy the same relations. We also have $a_0(0,\lambda) = s_0(0,\lambda)=1$ and $a_1(1,\lambda) = s_1(1,\lambda)=1$. Therefore, $s_i(x, \lambda)$ and $a_i(x, \lambda)$, $i\in \{0,1\}$, coincide for any integral $x$.
		\end{proof}
		
		\subsection{Boundary dynamical Weyl group: \texorpdfstring{$\k\subset \g$}{k in g}}
		
		We define the boundary dynamical Weyl group for the split symmetric pair $(\g, \k)$ in terms of simple reflection generators. 
		
		\begin{notation}
			For a root $\beta$, write $\k_{\beta}$ to denote the Lie subalgebra generated by $b_{\beta}$.
		\end{notation}
		
		The isomorphism $\g_{\beta}\cong \sl_2$ induces an isomorphism of the pairs: $\k_{\beta}\subset \g_{\beta}$ and $\so_2\subset \sl_2$. 
		
		\begin{remark}
			Although the pairs $\k_{\beta}\subset \g_{\beta}$ are isomorphic to $\so_2\subset \sl_2$ for all roots $\alpha\in R$, our definition of the boundary dynamical Weyl group is in terms of Verma modules induced from $\b$, which depends on a choice of simple roots $\Pi$. Consequently, in the following definition we restrict to simple roots. 
		\end{remark}
		
		\begin{defn}\label{D:A-k-in-general}
			For each simple root $\alpha\in \Pi$ we define a family of linear operators $A^{\k, \alpha}(\lambda)$, such that $\lambda \in \h$, as follows. Let $X$ be an integrable $\k$-module. For $\xi \in X$, set
			\begin{equation}\label{E:A-in-general-defn}
				A_X^{\k, \alpha}(\lambda) \xi := A_{X}^{\k_{\alpha}}(\langle \lambda, \alpha^{\vee}\rangle).
			\end{equation}
		\end{defn}
		
		We will also use the notation
		\begin{equation}
			\label{eq::boundary_dynamical_weyl_group_scalar_notation}
			A^{\k}(b_{\alpha},t) := \frac{(-2i)^{t+1}}{t+1} \B \left( \frac{ib_{\alpha} - t}{2}, t+1 \right)^{-1},
		\end{equation}
		as in \cref{notation:boundary_dynamical_weyl_group_scalar_notation} for \emph{all} roots, not necessarily simple. 
		
		The following is the boundary analog of equation \eqref{E:dw_tensor-property}, or equation \eqref{E:wk=wk-otimes-wg}.
		\begin{prop}
			\label{prop:boundary_tensor}
			For every integrable representation $X$ of $\k$ and finite dimensional representation $V$ of $\g$, we have
			\begin{equation}
				\label{eq::boundary_dyn_weyl_tensor}
				A_{X\otimes V}^{\k,\alpha}(\lambda) = G_{X, V}(s_{\alpha}\cdot \lambda)\circ \big(A_{X}^{\k,\alpha}(\lambda - h|_V)\otimes A_V^{\g,\alpha}(\lambda)\big)\circ G_{X, V}(\lambda)^{-1}.
			\end{equation}
		\end{prop}
		
		\begin{proof}
			For sufficiently positive integral $\lambda$, the equality follows directly from the construction of the (boundary) dynamical Weyl group operators via Verma modules. Observe that the matrix coefficients of the operators $G_{X,V}(\lambda)$ and $A^{\g,\alpha}_V(\lambda)$ are rational functions in $\lambda$. Therefore, in terms of matrix coefficients, the lemma is equivalent to equalities of the form
			\[
			f_1(\lambda) A(x_1,\lambda) + \ldots + f_n(\lambda) A(x_n,\lambda) = 0
			\]
			for some rational functions $f_k(\lambda)$. Moreover, by integrability and \eqref{eq::boundary_x_shift}, we can choose the collection $\{x_k\}$ such that $x_k$'s are pairwise non-congruent modulo $2\sqrt{-1}$ for every $k$. Since these equalities are satisfied for all sufficiently positive integral $\lambda$, they are satisfied for all $\lambda$ by \cref{prop:beta_lin_independence_rat_func}.
		\end{proof}
		
		\begin{example}
			Let $\g = \sl_2$, with standard generators $\{e,h,f\}$, and $\k = \so_2 = \mathrm{span}(b)$, where $b:=f-e$. Denote by $\{v_+,v_-\}$ the standard $h$-weight basis of $V(1)$. 
   
            The eigenvectors of $b$ are $v_+ - iv_-$, with eigenvalue $i$, and $v_+ + iv_-$, with eigenvalue $-i$. Therefore, the matrix of $A^{\so_2}_{\C_x \otimes V(1)}(\lambda)$ in the standard basis is 
			\[
			\begin{bmatrix}
				\frac{A(x+i,\lambda) + A(x-i,\lambda)}{2} & -\frac{A(x+i,\lambda) - A(x-i,\lambda)}{2i} \\
				\frac{A(x+i,\lambda) - A(x-i,\lambda)}{2i} & \frac{A(x+i,\lambda) + A(x-i,\lambda)}{2}
			\end{bmatrix}
			\]
			We use \cref{example:fusion_gk_sl2} and \cref{example:dw_sl2_vect_rep} to compute that the matrix of
			\[
			G_{\C_x,V(1)} (s\cdot \lambda) \circ A_{\C_x}^{\so_2} (\lambda - h|_{V(1)}) \otimes A_{V(1)}^{\sl_2}(\lambda) \circ G_{X,V(1)}(\lambda)^{-1}
			\]
			in the standard basis is equal to
			\[
			\begin{bmatrix}
				\frac{xA(x,\lambda-1)}{\lambda+1} & \frac{x^2}{(\lambda+1)^2} A(x,\lambda-1) - \frac{\lambda+2}{\lambda+1} A(x,\lambda+1) \\ A(x,\lambda-1) & \frac{xA(x,\lambda-1)}{\lambda+1}
			\end{bmatrix}.
			\]
			Thanks to \cref{lm:boundary_dw_shift}, the two matrices coincide.
		\end{example}

		The following theorem is the boundary analog of \cref{prop:dynamical_weyl_group_braid_relations}. 
		
		\begin{thm}
			\label{thm:boundary_braid_relations}
			Let $X$ be an integrable $\k$-module. The operators $\{A^{\k,\alpha_i}_X(\lambda)\}_{\alpha_i\in\Pi}$ satisfy the braid relations
			\begin{equation}\label{E:bdw-braid-relns}
				\ldots A^{\k,\alpha_{i}}_X(s_{j}s_{i}\cdot \lambda)A^{\k,\alpha_{j}}_X(s_{i}\cdot \lambda) A^{\k,\alpha_{i}}_X(\lambda) = \ldots A^{\k,\alpha_{j}}_X(s_{i}s_{j}\cdot \lambda)A^{\k,\alpha_{i}}_X(s_{j}\cdot \lambda) A^{\k,\alpha_{j}}_X(\lambda).
			\end{equation}
		\end{thm}
		\begin{proof}
			Rewriting \cref{prop:boundary_tensor} gives that
			\begin{equation}
				\label{eq:lm_boundary_tensor_conj}
				G_{X,V}(s_{\alpha}\cdot \lambda)^{-1} A^{\k,\alpha}_{X\otimes V}(\lambda) G_{X,V}(\lambda) =   A_X^{\k,\alpha} (\lambda - h|_V) \otimes A_V^{\g,\alpha}(\lambda).
			\end{equation}
			Assume that the braid relation is proven for $A_X^{\k,\alpha} (\lambda)$. We know by \cite{EtingofVarchenkoDynamicalWeyl} that the braid relation for the usual dynamical Weyl group operators $A_V^{\g,\alpha}(\lambda)$ is satisfied as well. We will show that for any two simple roots $\alpha_i, \alpha_j\in \Pi$
			\begin{equation}\label{braid-XtimesV}
				\ldots A^{\k,\alpha_{j}}_{X\otimes V}(s_{{i}}\cdot \lambda) A^{\k,\alpha_{i}}_{X\otimes V}(\lambda) = \ldots A^{\k,\alpha_{i}}_{X\otimes V}(s_{{j}}\cdot \lambda) A^{\k,\alpha_{j}}_{X\otimes V}(\lambda).
			\end{equation}
			Since $A^{\g,\alpha}_{V}(\lambda)$ maps the $\mu$-weight space $V[\mu]$ to the $s_{\alpha}\mu$-weight space $V[s_{\alpha}\mu]$, we have 
			\begin{gather*}
				\ldots \bigg(A_X^{\k, \alpha_j}(s_{i}\cdot \lambda - h|_V)\otimes A_V^{\g, \alpha_j}(s_{i}\cdot \lambda)\bigg)\circ \bigg(A_X^{\k,\alpha_{i}} (\lambda - h|_V) \otimes A_V^{\g,\alpha_{i}}(\lambda)\bigg) \\
				= \\
				\bigg(\ldots A^{\k,\alpha_{j}}_{X}(s_i\cdot \lambda - h|_V) \circ A^{\k,\alpha_{i}}_{X}(\lambda - h|_V)\bigg) \otimes \bigg( \ldots A_V^{\g, \alpha_j}(s_i\cdot \lambda)\circ A_V^{\g,\alpha_{i}}(\lambda)\bigg) \\
				= \\
				\bigg(\ldots A^{\k,\alpha_{i}}_{X}(s_j\cdot \lambda - h|_V) \circ A^{\k,\alpha_{j}}_{X}(\lambda - h|_V)\bigg) \otimes \bigg( \ldots A_V^{\g, \alpha_i}(s_j\cdot \lambda)\circ A_V^{\g,\alpha_{j}}(\lambda)\bigg) \\
				= \\
				\ldots \bigg(A_X^{\k, \alpha_i}(s_{j}\cdot \lambda - h|_V)\otimes A_V^{\g, \alpha_i}(s_{j}\cdot \lambda)\bigg)\circ \bigg(A_X^{\k,\alpha_{j}} (\lambda - h|_V) \otimes A_V^{\g,\alpha_{j}}(\lambda)\bigg)
			\end{gather*}
			Where the second equality follows from our assumption. This shows that the the braid relation is satisfied by the operators from the right hand side of \eqref{eq:lm_boundary_tensor_conj} so the braid relation is also satisfied by the operators
			\[
			G_{X,V}(s_{\alpha}\cdot \lambda)^{-1} A^{\k,\alpha}_{X\otimes V}(\lambda) G_{X,V}(\lambda)
			\]
			from the left-hand side of \eqref{eq:lm_boundary_tensor_conj}. Equation \eqref{braid-XtimesV} follows immediately.
			
			Therefore, it is enough to prove the braid relations for irreducible representations $X_1,\ldots,X_n$ such that any integrable representation of $\k$ is a direct summand of $X_i \otimes V$ for some $i$ and a suitable representation $V$ of $\g$. Note that it is enough to prove this when $\g$ is a rank 2 algebra. This is done in \cref{appendix:rank_two}.
		\end{proof}
		
		\subsection{Boundary \texorpdfstring{$B$}{B}-operators}\label{subsect:k-B-operator}
		
		As in \cref{subsect:sl2_dw_diag}, we can define the $B$-operators. Recall from \cref{def:boundary_weyl_sl2} that the boundary Weyl group operator is defined by $\tilde{s}^{\k}_{\alpha} = \exp\left({\frac{\pi b_{\alpha}}{2}}\right)$. Since $\sin(x) = (\exp(ix)- \exp(-ix))/{2i}$,  we find that
		\begin{equation}
			\label{eq::sin_boundary_weyl_group}
			\sin \left(\pi \frac{\mu - ib_{\alpha}}{2}\right) = \frac{\exp(\frac{i\pi\mu}{2})\exp(\frac{\pi b_{\alpha}}{2}) - \exp(\frac{-i\pi \mu}{2})\exp(\frac{-\pi b_{\alpha}}{2})}{2i} = \frac{\tilde{s}^{\k}_{\alpha} \exp(\frac{i\pi \mu}{2}) - (\tilde{s}^{\k}_{\alpha})^{-1} \exp(-\frac{i\pi \mu}{2})}{2i},
		\end{equation}
		for $\mu \in \C$.
		
		\begin{defn}
			\label{def:boundary_b_operator}
			The \defterm{boundary $B$-operator} for a (not necessarily simple) root $\beta$ is defined as follows
			\[
			\cB^{\k}_{\beta} (t) := \left(\sin \left(\pi \frac{t -1 - ib_{\beta}}{2}\right) \right)^{-1} A^{\k}(b_{\beta},t-1) = \frac{(-2i)^t}{\pi t} \Beta\left( \frac{t+ib_{\beta}}{2}, \frac{t-ib_{\beta}}{2} \right)
			\]
			Here, $A^{\k}(b_{\beta}, t-1)$ is as in notation \eqref{eq::boundary_dynamical_weyl_group_scalar_notation}.
		\end{defn}

		\begin{remark}
			The shift of variable by $-1$ corresponds to the $\rho$-shift. It differs from the $B$-operators \eqref{eq::b_operator} for $\g$, but it will be more convenient for the $(\so_{2n},\rO_m)$-duality of \cref{subsect::quantum_orthogonal_duality}.
		\end{remark}
		
		\begin{remark}
			A similar beta factor arises for intertwining operators among principal series representations of $\SL(2,\mathbb{R})$, see \cite{Wallach} (or the lecture notes \cite{Garrett} for the same beta factor). Indeed: as indicated in \cite{ReshetikhinStokman}, if we denote by $K$ the \emph{compact real group} with Lie algebra $\k$, then the space of $K$-finite vector in a principal series representation is isomorphic to $K$-finite vectors in the completion of the dual Verma module. 
		\end{remark}
		
		Thanks to the next lemma, these operators behave well with respect to the boundary Weyl group operators of \cref{sect:boundary_braid_group}.
		\begin{lm}
			\label{lm:boundary_b_operator_weyl_group_action}
			For any $\beta\in R_+$ and simple root $\alpha_i$, we have
			\[
			(\tilde{s}^{\k}_{i})^{\pm} \cdot  \cB^{\k}_{\beta}(t)\cdot  (\tilde{s}^{\k}_{i})^{\mp}=  \cB^{\k}_{s_{i}(\beta)} (t).
			\]
			\begin{proof}
				We want to compute the action of conjugation by $\tilde{s}_i^{\k}$, the adjoint action of $\tilde{s}_i^{\k}$, on the $B$ operator associated to $\beta$. By \cref{lm:braid_group_action_g_sign}, we obtain that $\mathrm{Ad}_{\tilde{s}^{\k}_i}(b_{\beta}) = \pm b_{s_i(\beta)}$. Application of $\mathrm{Ad}_{\tilde{s}^{\k}_i}$ to the boundary $B$-operator associated to $\beta$ then gives
				\[
				\mathrm{Ad}_{\tilde{s}^{\k}_i} \cB^{\k}_{\beta}(t) = \frac{(-2i)^t}{\pi t} \Beta\left( \frac{t\pm ib_{s_i(\beta)}}{2}, \frac{t\mp ib_{s_i(\beta)}}{2} \right).
				\]
				As the beta function is symmetric, the lemma follows for $\tilde{s}^{\k}_{i}$. Similar arguments work for $(\tilde{s}^{\k}_{i})^{-1}$.
			\end{proof}
		\end{lm}
		
		\begin{cor}
			\label{cor:sine_as_braid_group}
			For any simple root $\alpha_i$ and $u\in \C$, we have 
			\[
			\cB^{\k}_{\beta}(t) \sin \left(\pi \frac{u-1 - ib_{\alpha_i}}{2} \right)= \sin\left(\pi\frac{u-1 - ib_{\alpha_i}}{2} \right)\cB^{\k}_{s_i(\beta)}(t).
			\]
			\begin{proof}
				Follows from equation \eqref{eq::sin_boundary_weyl_group} and \cref{lm:boundary_b_operator_weyl_group_action}.
			\end{proof}
		\end{cor}

		In fact, the additional sine contribution satisfies the following property.
		\begin{lm}
			\label{lm:sine_braid_relation}
			Let $S_{\alpha_i}(\lambda) = \sin\left(\pi\frac{\langle \lambda-\rho,\alpha_i^{\vee}\rangle - b_{\alpha_i}}{2}\right)$
			The collection of operators $\{S_{\alpha_i}(\lambda)\}$ for simple roots $\{\alpha_i\}$ satisfies the dynamical braid relation 
			\begin{align}
				\begin{split}
					\label{eq::sin_braid_relation}
					\ldots S_{\alpha_i}(s_j s_i\lambda) S_{\alpha_j}(s_i\lambda) S_{\alpha_i}(\lambda) = \ldots S_{\alpha_j}(s_i s_j\lambda) S_{\alpha_i}(s_j\lambda) S_{\alpha_j}(\lambda).
				\end{split}
			\end{align}
			\begin{proof}
				It is enough to prove the lemma for rank 2 algebras which is done case-by-case in \cref{appendix:rank_two}.
			\end{proof}
		\end{lm}

		The following theorem is a boundary analog of \cref{thm:b_operators_generalized_r_matrix}. 
		
		\begin{thm}
        \label{thm:broundary_b_ops_r_matrix}
			The operators $\{\cB^{\k}_{\beta}(\lambda)\}$ form a closed $R$-matrix in the sense of \cref{def:generalized_r_matrix}. 
			\begin{proof}
				Observe the relations in \cref{def:generalized_r_matrix} are equivalent to the same relations for $G_{\alpha}^{-1}$. So, it is enough to prove the statement for 
				\[
				\cB^{\k}_{\alpha}(\lambda)^{-1} = \sin \left(\pi \frac{\langle \lambda - \rho, \alpha^{\vee} \rangle - ib_{\alpha}}{2}\right) A^{\k,\alpha}(\lambda-\rho)^{-1}. 
				\]
				It follows easily from \cref{thm:boundary_braid_relations} that the operators $\tilde{A}^{\k,\alpha}(\lambda)^{-1}:= A^{\k,\alpha}(\lambda-\rho)^{-1}$ satisfy the \emph{unshifted} dynamical braid relation
				\[
				\ldots \tilde{A}^{\k,\alpha}(s_{\beta}s_{\alpha}\lambda)^{-1} \tilde{A}^{\k,\beta}(s_{\alpha}\lambda)^{-1}\tilde{A}^{\k,\alpha}(\lambda)^{-1} = \ldots \tilde{A}^{\k,\beta}(s_{\alpha}s_{\beta}\lambda)^{-1} \tilde{A}^{\k,\alpha}(s_{\beta}\lambda)^{-1}\tilde{A}^{\k,\beta}(\lambda)^{-1},
				\]
				and the same is true for the sine operators $S_{\alpha}(\lambda)$ by \cref{lm:sine_braid_relation}. Multiply both sides by
				\begin{align*}
					S_{\alpha}(\lambda) S_{\beta}(s_{\alpha}\lambda) S_{\alpha}(s_{\beta}s_{\alpha} \lambda) \ldots = S_{\beta}(\lambda) S_{\alpha}(s_{\beta}\lambda) S_{\beta}(s_{\alpha}s_{\beta} \lambda) \ldots
				\end{align*}
				on the right. By pushing the sine terms to the left accordingly and using \cref{cor:sine_as_braid_group}, we conclude. 
			\end{proof}
		\end{thm}

  \begin{remark}
      Note that the relation between $B$-operators and dynamical Weyl group is different in the boundary case compared to the one for $\g$ of \cref{prop:dynamical_weyl_group_b_operator}. \cref{thm:broundary_b_ops_r_matrix} was the main motivation to use the sine factor instead of the boundary Weyl group operator.
  \end{remark}

        \section{Boundary Casimir connection}\label{sect:Conections}
		
		We briefly recall the basic notions concerning connections following \cite[Lecture 7]{EtingofFrenkelKirillov}. Then we introduce a new example of a flat connection, which is the boundary analog of the Casimir connection. 
		
		\subsection{Background on connections and differential operators}\label{subsect::background-connection}
		
		Let $M$ be a complex manifold and let $E$ be a finite dimensional complex vector bundle on $M$. 
		
		\begin{remark}
			For our purposes, the only examples to consider are when $M= \C^n\setminus\{H_{\alpha_1}, \dots, H_{\alpha_{\ell}}\}$, for $H_{\alpha_i}$ some codimension $1$ linear subspaces, and when $E$ a trivial vector bundle of the form $M\times \C^r$.  
		\end{remark}
		
		For an open subset $U\subset M$, write $\Gamma(U, E)$ for the holomorphic sections of $E$ over $U$. If $f\in \cO_M$ is a holomorphic function and $s\in \Gamma(U, E)$, then $fs\in \Gamma(U, E)$.
		
		We write $TM$ for the tangent bundle of $M$, so a vector field is a section of $TM$. We also write $T^*M$ for the cotangent bundle.
		
		\begin{defn}
			A \defterm{connection} is a map of local sections of vector bundles $\nabla:\Gamma(U, E)\rightarrow \Gamma(U, E\otimes T^*M)$ such that
			\begin{equation}
				\nabla(fs) = s\otimes df + f\nabla(s).
			\end{equation}
		\end{defn}
		
		Thus, for each vector field $\xi\in \Gamma(M, TM)$, we have 
		\[
		\nabla_{\xi}:\Gamma(U, E)\rightarrow \Gamma(U, E), \qquad s\mapsto \langle \nabla(s), \xi\rangle. 
		\]
		The condition in equation \eqref{E:diff-op-version-conection} then becomes equivalent to $\nabla_{\xi}$ being a differential operator such that
		\begin{equation}\label{E:diff-op-version-conection}
			\nabla_{\xi}(fs) =  \partial_{\xi}(f)s + f\nabla_{\xi}(s).
		\end{equation}
		
		In what follows, we will only consider manifolds $M$ with a chosen coordinate system $(z_1,\ldots,z_n)$ and trivial bundles $E = \underline{V} := M \times V$ for some vector space $V$. Then $\nabla_{\partial/\partial z_i} = \frac{\partial}{\partial z_i} +A_i(z_1, \dots, z_n)$, for $i=1, \dots, n$, where $A_i(z_1, \dots, z_n)$ is an $\cO_M$-linear endomorphism of the fiber $E_{z_1, \dots, z_n}$. Vice versa, a connection can be recovered from the data of such operators.
		
		There is a natural group of transformations acting on the space of connections; for brevity, we recall the definition in the setup as above.
		\begin{defn}
			\label{def:gauge_transformation}
			A \defterm{gauge transformation} of a connection $\nabla$ on a trivial vector bundle $\underline{V}$ is a function $g(z) \colon M \rightarrow \GL(V)$ acting by
			\[
			\nabla \mapsto \nabla^g, \qquad \frac{\partial}{\partial z_i} + A_i \mapsto \frac{\partial}{\partial z_i} + g A_i g^{-1} - \frac{\partial g}{\partial z_i} g^{-1}. 
			\]
			Two connections are called \defterm{gauge equivalent} if they differ by a gauge transformation.
		\end{defn}
		
		In particular, if a finite group $\Gamma$ acts on the fiber $V$, then it defines gauge transformations by the rule
		\[
		\frac{\partial}{\partial z_i} + A_i \mapsto \frac{\partial}{\partial z_i} + \gamma A_i \gamma^{-1}, \ \gamma \in \Gamma.
		\]
		
    Assume that $\Gamma$ acts on the base $M$. Then there is an induced action on connections via $\nabla(z) \mapsto \nabla(\gamma(z))$ for $\gamma \in \Gamma$. If the action is free, we can consider the quotient space $M/\Gamma$. The following proposition is an easy corollary of definitions.
		\begin{prop}
			\label{prop::connection_to_quotient}
			If $\nabla(\gamma(z)) = \nabla^{\gamma^{-1}}(z)$, for all $\gamma\in \Gamma$, then $\nabla$ descends to the quotient $M/\Gamma$.
		\end{prop}
		
		
		\begin{defn}
			A \defterm{flat connection} is a conection $\nabla$ such that $[\nabla_{\xi}, \nabla_{\eta}]=\nabla_{[\xi, \eta]}$, for all vector fields $\xi, \eta$.  
		\end{defn}
		
		In coordinates, $\nabla$ is a flat connection if and only if $[\nabla_{\partial/\partial z_i},\nabla_{\partial/\partial z_j}] = 0$ for all $1\le i,j \le n$.

		\subsection{Boundary Casimir connection}
		\label{subsect:boundary_casimir}
		In this subsection, we are in the setup of \cref{sect:boundary_fusion}. For instance, we consider an arbitrary simple Lie algebra $\g$ with a Cartan involution $\theta$.

		Tautologically, we consider any root $\alpha\in \h^*$ as a function on $\h$. Denote by $\h^{\reg} := \h \backslash \bigcup_{\alpha \in R_+} \{h|\langle h, \alpha\rangle=0\}$ the complement to the root hyperplanes. 
		
		\begin{defn}\label{D:boundary-Casimir}
			The \defterm{boundary Casimir connection} is a meromorphic connection on $\h^{\reg}$ given by
			\[
			\nabla^{\bCas} = \kappa\cdot \mathrm{d} - \sum_{\alpha\in R_+} \frac{\mathrm{d} \alpha}{\alpha} B_{\alpha}^2,
			\]
			where $\kappa\in \C^*$ is any nonzero number.
		\end{defn}
		
		We will prove the following result.
		\begin{thm}
			\label{thm::bcasimir_flat}
			The boundary Casimir connection is flat for any $\kappa$.
		\end{thm}
		
		For $\lambda, \mu\in \h^{\reg}$, let 
		\[
		T(\lambda,\mu) = \sum_{\alpha\in R_+} \frac{\langle \alpha,\lambda \rangle}{\langle \alpha,\mu\rangle} B_{\alpha}^2.
		\]
		One can easily see that \cref{thm::bcasimir_flat} follows from the following proposition.
		\begin{prop}
			\label{prop::bcasimir_flat_commute}
			For any $\lambda,\nu,\mu \in \h$, we have
			\[
			[T(\lambda,\mu),T(\nu,\mu)] = 0.
			\]
		\end{prop}
		
		The following lemma is an analog of \cite[Lemma 2.3]{FelderMarkovTarasovVarchenko}.
		
		\begin{lm}
			\label{lm:bcasimir_alpha_beta_commutator}
			Let $\alpha,\beta \in R$ be roots such that $\alpha \neq \pm \beta$. Denote by $S=S(\alpha,\beta)$ the set of integers $j$ such $\alpha + j \beta$ is a root. Then
			\[
			\sum_{j\in S(\alpha,\beta)} [B_{\alpha}^2,B_{\beta+j\alpha}^2] = 0.
			\]
			\begin{proof}
				We have
				\[
				B_{\beta+j\alpha}^2 = E_{-\beta - j\alpha}^2 + E_{\beta + j\alpha}^2 - E_{\beta+j\alpha} E_{-\beta-j\alpha} - E_{-\beta-j\alpha} E_{\beta+j\alpha}.
				\]
				By \cite[Lemma 2.3]{FelderMarkovTarasovVarchenko}, we have
				\[
				\sum_{j\in S} [E_{\pm\alpha}, E_{\beta+j\alpha} E_{-\beta-j\alpha}] = 0.
				\]
				Similarly, one can show that 
				\[
				\sum_{j\in S} [E_{\pm\alpha}, E_{-\beta-j\alpha} E_{\beta+j\alpha}] = 0.
				\]
				Therefore, it is enough to prove that 
				\[
				\sum_{j\in S} [B_{\alpha}^2, E_{\beta + j \alpha}^2 + E_{-\beta - j \alpha}^2] = 0.
				\]
				We will show that 
				\[
				\sum_{j\in S} [B_{\alpha}, E_{\beta + j \alpha}^2] = 0.
				\]
				Similar arguments work for other part.
				
				Let $N_{\alpha,\gamma} \in \C$ be the numbers such that $[E_{\alpha},E_{\gamma}] = N_{\alpha,\gamma} E_{\alpha+\gamma}$. Then a summand of the expression above is equal to
				\begin{align*}
					&N_{-\alpha,\beta+j\alpha} E_{\beta+j\alpha} E_{\beta+(j-1)\alpha} + N_{-\alpha,\beta+j\alpha} E_{\beta+(j-1)\alpha} E_{\beta+j\alpha} - \\
					-&N_{\alpha,\beta+(j+1)\alpha} E_{\beta+j\alpha} E_{\beta+j\alpha} - N_{\alpha,\beta+j\alpha} E_{\beta+j\alpha} E_{\beta+(j+1)\alpha}.
				\end{align*}
				
				One can see that the sum is equal to zero if $N_{\alpha,\beta+j\alpha} = N_{-\alpha,\beta+(j+1)\alpha}$. Observe that the Cartan involution $\theta$ satisfies $(\theta(x),\theta(y)) = (x,y)$. Therefore, we have
				\begin{align*}
					N_{\alpha,\beta+j\alpha} &= \big([E_{\alpha},E_{\beta+j\alpha}],E_{-\beta-(j+1)\alpha}\big) \\
					&= \big([\theta(E_{\alpha}),\theta(E_{\beta+j\alpha})],\theta(E_{-\beta-(j+1)\alpha})\big) \\
					&= -\big([E_{-\alpha},E_{-\beta-j\alpha}],E_{\beta+(j+1)\alpha}\big) \\
					&= \big([E_{-\alpha},E_{\beta+(j+1)\alpha}],E_{-\beta-j\alpha}\big) \\
					&= N_{-\alpha,\beta+(j+1)\alpha}.
				\end{align*}
			\end{proof}
			
			\begin{proof}[Proof of \cref{prop::bcasimir_flat_commute}]
				As in the proof of \cite[Proposition 2.2]{FelderMarkovTarasovVarchenko}, consider 
				\[
				[T(\lambda,\mu),T(\nu,\mu)] = \sum_{\alpha,\beta\in R_+} \frac{\langle \alpha, \lambda\rangle \langle \beta,\nu \rangle}{\langle \alpha,\mu\rangle \langle \beta, \mu \rangle} [B^2_{\alpha},B^2_{\beta}] = \sum_{\substack{\alpha, \beta\in R_+ \\ \alpha \ne \beta}}\frac{\langle \alpha, \lambda\rangle \langle \beta,\nu \rangle}{\langle \alpha,\mu\rangle \langle \beta, \mu \rangle} [B^2_{\alpha},B^2_{\beta}].
				\]
				Fixing $\lambda, \nu\in \h^{\reg}$ we obtain a function of $\mu\in \h^{\reg}$, let us show that this expression is a regular function. Then, since the function converges to zero at infinity, we will conclude that it vanishes identically.
				
				The residue\footnote{Recall for example that if $f$ is a rational function on $\C^n$, the residue along $z_1=0$ is computed as follows. Write $f = z_1^{-k}f_{-k} + z_1^{-k+1}f_{-k+1}+ \dots$, where each $f_i$ is a rational function of $z_2, \dots, z_n$, then $\mathrm{res}_{z_1=0}f = f_{-1}$. In particular, if $f$ has only a simple pole at $z_1=0$, i.e., if $f= z_1^{-1}f_{-1} + f_{0} + z_1f_1 + \dots$, then $\mathrm{res}_{z_1=0}f = (z_1f)|_{z_1=0}$.} at $\langle \alpha, \mu \rangle = 0$ is given by
				\begin{equation}\label{E:sum-over-gamma}
					\sum_{\substack{\beta\in R_+ \\ \beta \ne \alpha}}\frac{\langle \alpha, \lambda\rangle \langle \beta,\nu \rangle}{\langle \beta, \mu \rangle} [B^2_{\alpha},B^2_{\beta}] + \sum_{\substack{\beta\in R_+ \\ \beta \ne \alpha}}\frac{\langle \beta, \lambda\rangle \langle \alpha,\nu \rangle}{\langle \beta, \mu \rangle} [B^2_{\beta},B^2_{\alpha}]= \frac{1}{2}\sum_{\substack{\gamma\in R \\ \gamma \neq \pm \alpha}} \frac{\langle \alpha,\lambda \rangle \langle \gamma,\nu \rangle - \langle \gamma, \lambda \rangle \langle \alpha, \nu \rangle}{\langle \gamma,\mu \rangle} [B_{\alpha}^2, B_{\gamma}^2]
				\end{equation}
				(observe that $B_{-\gamma}^2 = B_{\gamma}^2$). The sum over $\gamma$ of the form $\beta + j\alpha$ for $j\in S(\alpha,\beta)$ is 
				\[
				\frac{1}{2}\sum_{j \in S} \frac{\langle \alpha,\lambda \rangle \langle \beta,\nu \rangle - \langle \beta, \lambda \rangle \langle \alpha, \nu \rangle}{\langle \gamma,\mu \rangle} [B_{\alpha}^2, B_{\beta + j\alpha}^2] = 0
				\]
				by \cref{lm:bcasimir_alpha_beta_commutator}. Since the sum in equation \eqref{E:sum-over-gamma} can presented as the sum of such terms, the residue is zero. We deduce that the function is regular.
				
			\end{proof}
		\end{lm}
		
		It follows from \cref{P:tilde-s-equal-exp} and \cref{lm:braid_group_action_g_sign} that 
		\[
		\mathrm{Ad}_{\tilde{s}^{\k}_i} (B_{\alpha}^2) = B_{s_i(\alpha)}^2.
		\]
		In particular, by \cref{prop::connection_to_quotient}, we see that the boundary Casimir connection descends to the quotient $\h^{\reg}/W$. Recall that the monodromy of the usual Casimir connection is given by the quantum Weyl group action \cite{ToledanoLaredoQuantumWeyl}. Analogously, we propose the following conjecture.
		\begin{conjecture}
			\label{conjecture:bcas_monodromy}
			The monodromy of the boundary Casimir connection is given by the quantum boundary Weyl group action.
		\end{conjecture}
		
		The conjecture is based on the following evidence, in the case of $\so_m\subset \sl_m$.
		
		\begin{remark}
			The $(\gl_n,\gl_m)$-duality, which we recall in \cref{sect::gl_duality}, exchanges the usual Casimir connection with the Knizhnik-Zamolodchikov connection. Moreover, the monodromy of the KZ connection is equivalent to the quantum $R$-matrix by the Kohno--Drinfeld theorem \cite{Drinfeld,Kohno}. Finally, by \cite[Theorem 6.5]{ToledanoLaredoKohnoDrinfeldWeylGroup}, the action of the $R$-matrix for $U_q(\gl_n)$ is equal to the quantum Weyl group action for $U_q(\gl_m)$. These three observations are combined in the proof of \cite[Theorem 7.1]{ToledanoLaredoKohnoDrinfeldWeylGroup} to show that the monodromy of the usual Casimir connection is given by the quantum Weyl group action. 
   
            Under the orthogonal Howe duality of \cref{sect:orthogonal_howe}, the KZ connection for $\so_{2n}$ is exchanged with the boundary Casimir connection for $\rO_m\subset \GL_m$. The Kohno--Drinfeld theorem again implies that the monodromy of the KZ connection for $\so_{2n}$ is equivalent to the $R$-matrix for $U_q(\so_{2n})$. Finally, we computed, for $n\le 4$, that under Howe duality the action of the $U_q(\so_{2n})$ $R$-matrix is equal to the $U_q^{\iota}(\so_2)$ $\iota$quantum Weyl group operators \cite{Zhang-Weinan-thesis}. 
		\end{remark}

		\section{\texorpdfstring{$(\gl_n,\gl_m)$}{(gl(n),gl(m))}-duality}
		\label{sect::gl_duality}

		In this section, we recall the $(\gl_n,\gl_m)$-duality for exterior powers following \cite{TarasovUvarov}.
		
		\subsection{Generalities on \texorpdfstring{$\gl_n$}{gl(n)}} In what follows, we recall the definition of the general linear algebra $\gl_n$ and of its skew-polynomial representation $\fP_n$.
		\subsubsection{Lie algebra and representations}
		Consider the Lie algebra $\gl_n$ with the basis of matrix units 
            \[
            \{ E_{ij}| i,j=1\ldots n\} \qquad\text{satisfying} \qquad [E_{ij},E_{kl}] = \delta_{jk} E_{il} - \delta_{li} E_{kj}.
            \]
        We identify the Cartan subalgebra $\h$ with $\C^n$ via the basis $\{E_{ii} | i = 1,\ldots, n\}$ and denote by $\{\epsilon_i\} \subset \h^*$ the dual basis. Then the roots are $\epsilon_i - \epsilon_j$ with the root spaces $\g_{\epsilon_i - \epsilon_j} = \C\cdot E_{ij}$. The positive roots are $R_+ = \{ \epsilon_i - \epsilon_j| i < j\}$ with simple roots
		    \[
		  \alpha_1 = \epsilon_1 - \epsilon_2, \qquad \ldots \qquad \alpha_{n-1} = \epsilon_{n-1} - \epsilon_n.
		  \]
		
		For a sequence $\mathbf{l} = (l_1 \leq \ldots \leq l_n)$ of non-decreasing \emph{rational} numbers, we denote by $V(\mathbf{l})$ the corresponding irreducible representation of weight $\mathbf{l}$. For instance, the weight of the $k$-th exterior power $\Lambda^k \C^n$ of the vector representation is $(1,\ldots,1,0,\ldots,0)$, where $1$ is repeated $k$-times. We will also need a rational power of the determinant representation: for $\alpha\in \mathbb{Q}$, we denote by $\Det^{\alpha}$ the one-dimensional representation of weight $(\alpha,\ldots,\alpha)$ explicitly given by
		\begin{equation}
			\label{eq::determinant_representation}
			\Det^{\alpha} := \C\cdot \omega^{\alpha}, \qquad E_{ij}\cdot \omega^{\alpha} = \delta_{ij} \alpha\cdot \omega^{\alpha}.
		\end{equation}
		
		\subsubsection{Skew polynomials}
		\label{subsect:skew_polynomials}
		
		Denote by $\fP_k$ the space of anticommuting polynomials in $k$ variables $\{x_i | i = 1,\ldots, k\}$. The algebra $\fP_k$ is $\mathbb{Z}_{\ge 0}$ graded, with generators $x_i$ in degree $1$. In what follows, we denote by $\mathbf{1} \in \fP_n$ the identity polynomial.
		
		Let $\partial_i$ be the left (super) derivation such that $\partial_i (x_j) = \delta_{ij}$ and $\partial_i(f g)= \partial_i(f)g + (-1)^{\deg(f)}f\partial_i(g)$. There is a $\gl_k$-representation structure on $\fP_k$ given by 
		\begin{equation}
			\label{eq::skew_polynomials_gl_rep}
			E_{ij}\cdot g(x) = x_i \partial_j g(x),\ g(x) \in \fP_k.
		\end{equation}
		As a $\gl_k$-representation the $\mathbb{Z}_{\ge 0}$-grading corresponds to the decomposition $\fP_k = \oplus_{l=0}^k \Lambda^l \C^k$ of exterior powers of the vector representation. 
		
		\subsection{Representation}
		\label{subsect:gl_duality_representation}
		
		For $\fP_{nm}$, we denote the variables by $x_{i\alpha}$ for $i=1,\ldots, n, \alpha = 1, \ldots, m$, and by $\partial_{i\alpha}$ the corresponding derivations. We denote 
		\[
		E_{ij}^{(\alpha)}:=x_{i\alpha} \partial_{j\alpha}, \qquad E_{\alpha\beta}^{(i)} := x_{i\alpha} \partial_{i\beta}.
		\]
		One can easily see that the operators $\{E_{ij}^{(\alpha)}|i,j=1\ldots n\}$ satisfy the $\gl_n$-relations; we denote the corresponding action by $\gl_n^{(\alpha)}$ (similarly for $\gl_m^{(i)}$). The formulas
		\begin{align*}
			E_{ij} \cdot g(x) &= \sum_{\alpha=1}^m E_{ij}^{(\alpha)} g(x),\ E_{ij} \in \gl_n, \\
			E_{\alpha\beta} \cdot g(x) &= \sum_{i=1}^n E_{\alpha\beta}^{(i)} g(x), \ E_{\alpha\beta} \in \gl_m.
		\end{align*}
		provide an $\gl_n \oplus \gl_m$-action (in particular, they commute). In what follows, we denote the $\gl_n$-action either by Latin indices or $\langle n \rangle$ and the $\gl_m$-one either by Greek indices or $\langle m \rangle$. 
		
		Since the generators $E_{ij}$ and $E_{\alpha\beta}$ act by even degree operators, there is an isomorphism $\fP_{nm} \cong \fP_n^{\otimes m}$ as $\gl_n$-representations and $\fP_{nm} \cong \fP_m^{\otimes n}$ as $\gl_m$-representations. To understand how $\fP_{nm}$ looks as a $\gl_n\oplus \gl_m$-representation, there is the following, for example see \cite[Theorem 5.18]{ChengWang}.
		
		\begin{thm}
			\label{thm:gln_glm_duality}
			For any partition $\mathbf{l}$, denote by $\mathbf{l}'$ its transpose. Then
			\[
			\fP_{nm} \cong \bigoplus_{\substack{\mathbf{l} = (l_1\geq \ldots\geq l_n) \\ l_1 \leq m}} V(\mathbf{l})^{\langle n \rangle} \otimes V(\mathbf{l}')^{\langle m \rangle},
			\]
			where $V(\mathbf{l})^{\langle n \rangle}$ is the irreducible representation of $\gl_n$ of highest weight $\mathbf{l}$ and $V(\mathbf{l}')^{\langle m \rangle}$ is the irreducible representation of $\gl_m$ of highest weight $\mathbf{l}'$.
		\end{thm}
		
		\subsection{Classical duality}
		\label{subsect:gln_classical_duality}
		
		Fix a non-zero complex number $\kappa$.
		
		On the $\gl_m$-side, we denote by
		\begin{equation}
			\label{eq::split_casimir}
			\Omega_{\gl_m}^{(i,j)} = \sum_{1\le \alpha,\beta\le n} E_{\alpha\beta}^{(i)} E_{\beta\alpha}^{(j)}
		\end{equation}
		the \emph{$2$-tensor Casimir operator} acting on $\fP_{nm}$. Under the identification $\fP_{nm} \cong \fP_{m} \otimes \ldots \otimes \fP_{m}$, $\Omega_{\gl_m}^{(i,j)}$ acts on the $(i,j)$-component of the tensor product. Consider the \emph{KZ operators} 
		\[
		\nabla^{\KZ,\langle m \rangle}_{z_1},\ldots, \nabla^{\KZ,\langle m \rangle}_{z_n}
		\]
		which are differential operators on $\C^n\setminus \{z_{i}=z_{j}|i\ne j\in \{1, \dots, n\}\}$ defined by:
		\begin{equation}
			\label{eq::kz_operators}
			\nabla^{\KZ,\langle m \rangle}_{z_{i}}(z,\kappa) = \kappa \partial_{z_{i}}  - \sum\limits_{\substack{ 1\le j \le n\\ j\ne i }} \frac{\Omega_{\gl_m}^{(i,j)}}{z_{i} - z_{j}}.
		\end{equation}
		It is well-known that they define a flat connection \cite{KnizhnikZamolodchikov}. 
		
		On the $\gl_n$-side of the duality, we take the Cartan subalgebra $\h_{\gl_n} \subset \gl_n$ with coordinates $(\lambda_1,\ldots,\lambda_n)$. Consider the \emph{differential dynamical} operators 
		\[
		\nabla_{\lambda_{1}}^{\Cas,\langle n \rangle},\ldots,\nabla_{\lambda_{n}}^{\Cas,\langle n \rangle}
		\]
		which are differential operators on $\h_{\gl_n}^{\reg}$ defined by:
		\begin{equation}
			\label{eq::casimir_operators}
			\nabla_{\lambda_{\alpha}}^{\Cas,\langle n \rangle}(\lambda,\kappa)  = \kappa \partial_{\lambda_{i}} - \sum_{\substack{1\le \beta \le n \\ j\ne i}} \frac{E_{ij} E_{ji} - E_{ii}}{\lambda_{i} - \lambda_{j}}
		\end{equation}
		In fact, they define (a version of) the \emph{Casimir connection} which is known to be flat as well, see \cite{FelderMarkovTarasovVarchenko}.
		
		We can identify $\C^n\rightarrow \h_{\gl_n}$ by $z_{i} \mapsto \lambda_{i}$, for $i = 1, \dots, n$. 
		
		\begin{thm}\cite[Theorem 4.4]{TarasovUvarov}
			\label{eq::gl_classical_bispectrality}
			The following relation holds 
			\[
			\nabla^{\KZ,\langle m \rangle}_{z_i}(z,\kappa) = \nabla^{\Cas,\langle m \rangle}_{\lambda_{i}}(-z,-\kappa).
			\]
		\end{thm}
		
		\begin{remark}\label{R:down-to-earth}
			In more down to earth terms, Theorem \ref{eq::gl_classical_bispectrality} says that in $\End(\fP_{nm})$ the operators
			\[
			\Omega_{\gl_m}^{(i,j)} \qquad \text{and} \qquad \big(E_{ij}E_{ji} - E_{ii}\big)
			\]
			differ by a sign. This follows from using $\partial_{i\beta} x_{j\beta} = - x_{j\beta} \partial_{i\beta} + \delta_{ij}$ to show that 
			\[
			\Omega_{\gl_m}^{(i,j)} = \sum_{\alpha,\beta} E_{\alpha\beta}^{(i)} E_{\beta\alpha}^{(j)} = \sum_{\alpha,\beta} x_{i\alpha} \partial_{i\beta} x_{j\beta} \partial_{j\alpha}
			\]
			is the same as 
			\[
			- E_{ij} E_{ji} + E_{ii} = -(\sum_{\alpha} x_{i\alpha} \partial_{j\alpha} ) (\sum_{\beta} x_{j\beta} \partial_{i\beta}) + \sum_{\alpha} x_{i\alpha} \partial_{i\alpha} = -\sum_{\alpha,\beta} x_{i\alpha} \partial_{i\beta} x_{j\beta} \partial_{j\alpha} + \sum_{\alpha} x_{i\alpha} \partial_{i\alpha}. 
			\]
		\end{remark}
		
		\subsection{Quantum duality}
		
		In this subsection, we recall the quantum version of \cref{subsect:gln_classical_duality}. For that, we will need to introduce the \emph{Yangian} $\rY(\gl_m)$.
		
		\subsubsection{Yangian}
		\label{subsubsect:gl_yangian}
		For all the statements and proofs, we refer the reader to \cite{MolevYangians}.
		
		Consider the \emph{Yang $R$-matrix} 
		\begin{equation}
			\label{eq::yang_r_matrix}
			R(u) = \id + \frac{P}{u},
		\end{equation}
		which is an $\End(\C^m \otimes \C^m)$-valued function of $u$, where $P = \sum_{\alpha,\beta} E_{\alpha\beta} \otimes E_{\beta\alpha}$ is the permutation operator. The Yangian $\rY(\gl_m)$ is generated by $t_{\alpha\beta}^{(k)}$ for $\alpha,\beta=1,\ldots, m$ and $k\geq 1$ subject to the  relations
		\[
		[t_{\alpha\beta}^{(r+1)},t_{\gamma\delta}^{(s)}] - [t_{\alpha\beta}^{(r)},t_{\gamma\delta}^{(s+1)}] = t_{\gamma\beta}^{(r)}t_{\alpha\delta}^{(s)}- t_{\gamma\beta}^{(s)}t_{\alpha\delta}^{(r)}.
		\]
		These relations are equivalent to the so-called \emph{RTT}-relation. Consider the generating series 
		\[
		t_{\alpha\beta}(u) := \delta_{\alpha\beta} + \sum_{k > 0}t_{\alpha\beta}^{(k)} u^{-k}\in \rY(\gl_m)[[u^{-1}]],
		\]
		and combine them into the matrix $T(u) = (t_{\alpha\beta}(u))_{\alpha,\beta=1}^m \in \End(\C^m) \otimes \rY(\gl_m)[[u^{-1}]]$, then the RTT-relation reads
		\[
		R_{12}(v-u) T_1(u) T_2(v) = T_2(v) T_1(u) R_{12}(v-u) \in \End(\C^m \otimes \C^m) \otimes \rY(\gl_m)[[u^{-1}]].
		\]
		
		The Yangian is a Hopf algebra with the coproduct
		\begin{equation}
			\label{eq::yangian_coproduct}
			\Delta(t_{\alpha\beta}(u)) = \sum_{\delta=1}^n t_{\alpha\delta}(u) \otimes t_{\delta\beta}(u).
		\end{equation}
		
		For any $z\in \C$ and $f(u) \in 1 + u^{-1} \C[[u^{-1}]]$, the following operations define automorphisms of the Yangian:
		\begin{align}
			\begin{split}
				\label{eq::yangian_automorphisms}
				f(u) &\colon T(u) \mapsto f(u) T(u), \\
				\tau_z &\colon T(u) \mapsto T(u-z), \\
				\tilde{\theta}&\colon T(u) \mapsto T^t(-u),
			\end{split}
		\end{align}
		where $T^t(u)\in \End(\C^n) \otimes \rY(\gl_m)$ is the transposition with respect to the $\End(\C^m)$-part.  
		
		There is a homomorphism $\ev\colon \rY(\gl_m) \rightarrow \U(\gl_m)$ called the \emph{evaluation homomorphism} defined by
		\begin{equation}
			\label{eq::yangian_evaluation}
			t_{\alpha\beta}(u) \mapsto \delta_{\alpha\beta} + \frac{E_{\beta\alpha}}{u}.
		\end{equation}
		For any $\gl_m$-representation $V$, we denote by $V(z)$ the pullback via the composition $\ev \circ \tau_z$. Observe that the automorphism $\tilde{\theta}$ of \eqref{eq::yangian_automorphisms} lifts the Cartan involution $\theta$ of \cref{sect:boundary_fusion}.
		
		For any two irreducible representations $V_1,V_2$ of $\gl_n$, the tensor products $V_1(z_1) \otimes V_2(z_2)$ and $V_2(z_2) \otimes V_1(z_1)$ are isomorphic irreducible representations of $\mathrm{Y}(\gl_m)$ provided $z_1 - z_2 \not\in \Z$. The isomorphism is given by $P_{12} \cR_{V_1 V_2}(z_2 - z_1)$, where $P_{12}$ is the flip operator and $\cR_{V_1 V_2}(z_2-z_1)$ is a rational $\End(V_1 \otimes V_2)$-valued function called the \emph{rational $R$-matrix}. In what follows, we will often not specify explicitly the evaluation parameters and set $t=z_2-z_1$.  Upon a choice of highest-weight vectors: $v_1\in V_1$ and  $v_2\in V_2$, $\cR_{V_1 V_2}(t)$ is uniquely determined by the relations 
		\begin{equation}
			\label{eq::r_matrix_properties}
			\begin{gathered}
				[\cR_{V_1 V_2}(t), x \otimes 1 + 1 \otimes x] = 0, \\
				\cR_{V_1 V_2}(t)\cdot (tE_{\alpha\beta} \otimes 1 + \sum_{\delta=1}^m E_{\alpha\delta} \otimes E_{\delta\beta}) = (tE_{\alpha\beta} \otimes 1 + \sum_{\delta=1}^m E_{\delta\alpha} \otimes E_{\beta\delta}) \cdot \cR_{V_1 V_2}(t), \\
				\cR_{V_1 V_2}(t)v_1 \otimes v_2 = v_1 \otimes v_2.
			\end{gathered}
		\end{equation}
		for any $x\in \gl_m$, see \cite{TarasovVarchenko}. It obeys the \emph{unitarity condition}
		\begin{equation}
			\label{eq::gl_rmatrix_unitarity}
			\cR_{12}(t) \cR_{21}(-t) = \id,
		\end{equation}
		and also satisfies the \emph{Yang-Baxter equation}
		\begin{equation}
			\label{eq::gl_rmatrix_qybe}
			\cR_{V_1 V_2} (z_1 - z_2) \cR_{V_1 V_3} (z_1 - z_3) \cR_{V_2 V_3} (z_2 - z_3) = \cR_{V_2 V_3} (z_2 - z_3) \cR_{V_1 V_3} (z_1 - z_3) \cR_{V_1 V_2} (z_1 - z_2).
		\end{equation}
		
		In fact, there exists a \emph{universal $R$-matrix} $\fR$ in a certain completion of the tensor square of $\rY(\gl_m)$, see \cite{Drinfeld}. Up to a scalar factor, the action of $\fR$ on $V_1(z_1) \otimes V_2(z_2)$ coincides with that of $\cR_{V_1 V_2}(z_2-z_1)$. 
		
		\subsubsection{Duality}
		\label{subsubsect:gln_quantum_duality}
		Recall that as a $\gl_m$-representation, the space of skew polynomials $\fP_{nm}$ is isomorphic to the tensor product $\fP_{m} \otimes \ldots \otimes \fP_{m}$. Denote by $\cR^{\langle m \rangle}_{ij}(t)$ the $R$-matrix acting on $(i,j)$-factors, where $t=z_1-z_2$ (we will not mention it explicitly in what follows). On the $\gl_n$ side, we denote by $\cB^{\langle n \rangle}_{\epsilon_{i} - \epsilon_{j}}$ the dynamical Weyl group operator of \eqref{eq::b_operator} corresponding to the root $\epsilon_{i} - \epsilon_{j}$ of $\gl_n$. 
		
		A quantum analog of the differential operators from \cref{subsect:gln_classical_duality} is as follows. Denote by $T_{z_i}$ the shift operator acting on meromorphic analytic functions $\C^n\rightarrow \fP_{nm}$ by
		\[
		[T_{z_{i}} f](z_1,\ldots,z_n) = f(z_1,\ldots, z_{i} + \kappa, \ldots, z_n). 
		\]
		Consider the difference operators
		\begin{align}
			\label{eq::gl_qkz_ops}
			\begin{split}
				Z^{\langle m \rangle}_{i}(z,\kappa) := \big(\cR^{\langle m \rangle}_{i n}(z_{i} - z_{n})\ldots \cR^{\langle m \rangle}_{i,i+1}(z_{i} - z_{i+1})\big)^{-1} \big( \cR^{\langle m \rangle}_{1i}(z_1 - z_{i} - \kappa) \ldots \cR^{\langle m \rangle}_{i-1,i}(z_{i-1}- z_{i} - \kappa) \big) T_{z_{i}}.
			\end{split}
		\end{align}
		They are called \emph{quantized Knizhnik-Zamolodchikov (qKZ)} operators. 
		
		On the other side, one can similarly define the shift operators $T_{\lambda_i}$ on functions $\h^{*} \rightarrow \fP_{nm}$, and the \emph{dynamical difference} operators
		\begin{align}
			\label{eq::gl_dyn_diff_ops}
			\begin{split}
				X^{\langle n \rangle}_{i}(\lambda,\kappa) := \big(\cB^{\langle n \rangle}_{\epsilon_{i} - \epsilon_n}(\lambda_{i} - \lambda_n) \ldots \cB^{\langle n \rangle}_{\epsilon_{i} - \epsilon_{i+1}}(\lambda_{i} - \lambda_{i+1})\big)^{-1} \big( \cB^{\langle n \rangle}_{\epsilon_1 - \epsilon_{i}}(\lambda_1 - \lambda_{i} - \kappa) \ldots \cB^{\langle n \rangle}_{\epsilon_{i-1} - \epsilon_{i}}(\lambda_{i-1} - \lambda_{i} - \kappa) \big)T_{\lambda_{i}}.
			\end{split}
		\end{align}
		Also, denote by 
		\[
		C_{ij}^{\langle n \rangle} (t) = \frac{\Gamma(t + E_{ii} +1) \Gamma(t - E_{jj})}{\Gamma(t + E_{ii} - E_{jj} + 1)\Gamma(t)}.
		\]
		
		\begin{thm}\cite[Theorem 4.4]{TarasovUvarov}
			\label{thm::gl_quantum_bispectrality}
			The following relation holds:
			\[
			\cR^{\langle m \rangle}_{ij} (t) = C_{ij}^{\langle n \rangle} (t) \cB^{\langle n \rangle}_{\epsilon_{i} - \epsilon_{j}} (-t).
			\]
			In particular, we have
			\[
			Z_{i}^{\langle m \rangle}(z,\kappa) = \frac{\prod_{1\leq j < i} C_{ji}^{\langle n \rangle}(z_{j} - z_{i} - \kappa)}{\prod_{m \geq j > i} C_{ij}^{\langle n \rangle}(z_{i} - z_{j})} \cdot X_{i}^{\langle n \rangle}(-z,-\kappa).
			\]
		\end{thm}

		\section{\texorpdfstring{$(\so_{2n},\rO_m)$}{(so(2n),O(m))}-duality}
		\label{sect:orthogonal_howe}
		
		Consider the $(\gl_n,\gl_m)$-duality from Section \ref{sect::gl_duality}. The Cartan involution $\theta$ acting on $\g= \sl_m$ is induced by the group automorphism of $\GL_m$ obtained by composing inversion with matrix transpose. Restricting the $\GL_m$ action on $\fP_{nm}$ to $\rO_m$ we find more operators in $\End_{\rO_m}(\fP_{nm})$ than just those coming from the $\gl_n$ action on $\fP_{nm}$. In fact, the centralizer $\End_{\rO_m}(\fP_{nm})$ has a Lie theoretic description: enlarge $\sl_n$ to $\so_{2n}$ and view the exterior power of the natural representation of $\sl_n$ as the spin representation for $\so_{2n}$. In this section we precisely recall this construction, then deduce various dualities, which parallel those recalled in Section \ref{sect::gl_duality}. 
		
		\subsection{Generalities on \texorpdfstring{$\fo_{k}$}{o(k)}}
		
		In this subsection, we recall the necessary facts about the orthogonal Lie groups and algebras. For the statements and proofs, we refer the reader to \cite[Section 5.3]{ChengWang}.
		
		\subsubsection{Lie algebra $\so_{k}$}
		\label{subsubsect:so_k}
		Let $k=2n$ be even. Consider $\C^{2n} = \Span(e_1,\ldots,e_n,e_{-1},\ldots,e_{-n})$ with the form $(e_i,e_{-j}) = \delta_{ij}$ and all other pairings zero. Denote by 
		\[
		M_{ij} := E_{ij} - E_{-j,-i} \qquad \text{for $i,j \in\{ -n, -n+1, \ldots, n-1, n\}$}.
		\]
		Under the isomorphism $\so_{2n} \cong \Lambda^2 \C^{2n}$, the element $M_{ij}$ corresponds to $e_i \wedge e_{-j}$.  
		
		We can choose a Cartan subalgebra $\h_{\so_{2n}} = \mathrm{diag}(h_1,\ldots,h_n,-h_1,\ldots,-h_n)$. Denote by $\epsilon_i \in \h_{\so_{2n}}^*$ the functional $\epsilon_i(h) = h_i$. Then the roots spaces of $\g = \so_{2n}$ are
		\[
		\g_{\epsilon_i - \epsilon_j} = \C\cdot M_{ij}, \qquad \g_{\epsilon_i + \epsilon_j} = \C\cdot M_{i,-j}, \qquad\text{and} \qquad \g_{-\epsilon_i-\epsilon_j} = \C\cdot M_{-j,i},
		\]
		for $i\neq j\in\{1, \dots, n\}$. Our choice of positive roots is
		\[
		R_+ = \{ \epsilon_i \pm \epsilon_j | 1\le i< j\le n\}
		\]
		with simple roots
		\[
		\alpha_1 = \epsilon_1 - \epsilon_2,\qquad \alpha_2 = \epsilon_2 - \epsilon_3,\qquad \ldots\qquad , \qquad \alpha_{n-1} = \epsilon_{n-1} - \epsilon_n, \qquad \text{and} \qquad \alpha_n = \epsilon_{n-1} + \epsilon_n.
		\]
		The corresponding Dynkin diagram is 
		
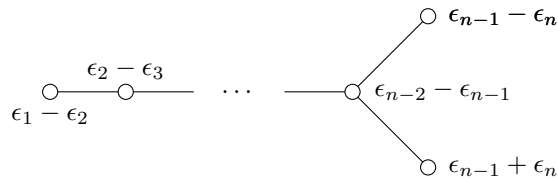
\begin{center}
			\begin{tikzpicture}
				\draw (5.07,0.07) -- (5.93,0.93);
				\draw (5.07,-0.07) -- (5.93,-0.93);
				\draw (1.1,0) -- (1.9,0);
				\draw (2,0) circle (0.1cm);
				\draw (2.1,0) -- (2.9,0);
				\node at (3.5,0) {\ldots};
				\draw (4.1,0) -- (4.9,0);
				\draw (5,0) circle (0.1cm);
				\draw (6,1) circle (0.1cm);
				\draw (6,-1) circle (0.1cm);
				\draw (1,0) circle (0.1cm);
				\node at (1,-0.3) {$\epsilon_1 - \epsilon_2$};
				\node at (2,0.3) {$\epsilon_2 - \epsilon_3$};
				\node at (6.2,0) {$\epsilon_{n-2} - \epsilon_{n-1}$};
				\node at (7,1) {$\epsilon_{n-1} - \epsilon_{n}$};
				\node at (7,1) {$\epsilon_{n-1} - \epsilon_{n}$};
				\node at (7,-1) {$\epsilon_{n-1} + \epsilon_{n}$};
			\end{tikzpicture}
		\end{center}
		The diagram automorphism induces an automorphism of $R_+$, acting by $\epsilon_{n-1} \pm \epsilon_n \mapsto \epsilon_{n-1} \mp \epsilon_n$ on the last two roots and by identity on the other ones. It also induces an outer automorphism of $\so_{2n}$ which we denote by $\tau_{n}$. In general, for $k=1,\ldots, n$, we define Lie algebra automorphisms $\tau_k \colon \so_{2n} \rightarrow \so_{2n}$ by
		\begin{equation}
			\label{eq::so2n_outer_auto}
			M_{ij} \mapsto M_{ij}, \qquad M_{i,\pm k} \mapsto M_{i,\mp k}, \qquad M_{\pm k,j} \mapsto M_{\mp k, j}, \qquad \text{and} \qquad M_{kk} \mapsto - M_{kk},
		\end{equation}
		for any $i,j\neq k$.
		
		We will consider the embedding
		\begin{equation}
			\label{eq::gl_embedding_so}
			\gl_n \hookrightarrow \so_{2n}, \qquad E_{ij} \mapsto M_{ij},\qquad \text{for $i,j\in \{1,\ldots, n\}$}. 
		\end{equation}
		
		Finite dimensional irreducible representations of $\so_{2n}$ are described by sequences of half-integral numbers $\vec{l} = (l_1,\ldots,l_n)$ satisfying $l_1 \geq \ldots \geq l_{n-1} \geq |l_n|$; we denote them by $V(\so_{2n},\vec{l})$. The corresponding representation integrates to the group $\SO(2n)$ if and only if all numbers are integral, i.e., when $\bl := (l_1,\ldots,l_{n-1},|l_n|)$ defines a partition.

		Let $k=2n+1$ be odd. Consider $\C^{2n+1} = \Span(e_1,\ldots,e_n,e_0,e_{-1},\ldots,e_{-n})$ with the form $(e_{i},e_{-j}) = \delta_{ij}$ with all the other pairings zero. As before, denote by $M_{ij} = E_{ij} - E_{-j,-i}$. Note that now indices can be zero. Choose the Cartan subalgebra $\h_{\so_{2n+1}}= \mathrm{diag}(h_1, \dots, h_n, 0, -h_1, \dots, -h_n)$ and write $\epsilon_i\in \h_{\so_{2n+1}}^*$ for the functional $\epsilon_i(h) = h_i$. The root spaces of $\so_{2n+1}$ are
		\begin{gather*}
			\g_{\epsilon_i - \epsilon_j}  = M_{ij}, \qquad
			\g_{\epsilon_i + \epsilon_j} = M_{i,-j}, \qquad \text{and} \qquad 
			\g_{-\epsilon_i-\epsilon_j}  = M_{-j,i}, \qquad \text{for $i\ne j\in \{1, \dots, n\}$}\\
			\g_{\epsilon_i} = M_{i,0}, \qquad \text{and} \qquad
			\g_{-\epsilon_i} = M_{-i,0} \qquad \text{for $i\in \{1, \dots, n\}$}.
		\end{gather*}
		Positive roots are 
		\[
		R_+ =  \{ \epsilon_i \pm \epsilon_j | 1\le i< j\le n\} \cup \{ \epsilon_i|1\le i\le n\}
		\]
		with simple roots
		\[
		\alpha_1 = \epsilon_1 - \epsilon_2,\qquad \alpha_2 = \epsilon_2 - \epsilon_3,\qquad \ldots, \qquad \alpha_{n-1} = \epsilon_{n-1} - \epsilon_n, \qquad \alpha_n =  \epsilon_n. 
		\]
		
		Finite dimensional irreducible representations of $\so_{2n+1}$ are described by sequences $\vec{l} = (l_1,\ldots,l_n)$ of half-integral numbers such that $l_1 \geq \ldots \geq l_n \geq 0$; we denote it by $V(\so_{2n+1},\vec{l})$. This representation integrates to a representation of $\SO(2n+1)$ if and only if all the numbers are integral, i.e., $\bl :=(l_1,\ldots,l_n)$ defines a partition.
		
		\subsubsection{Representations of $\rO_k$}
		\label{subsubsect:o_m_reps}
		
		Recall the Cartan involution $\theta$ from \cref{sect:boundary_fusion}. In the case of $\g=\gl_k$, it is given by $\theta(E_{\alpha\beta}) = -E_{\beta\alpha}$ for all $\alpha,\beta\in \{1,\ldots, m\}$. In fact, it can be extended to an involution on the group $\GL_k$ of invertible matrices by $\theta(g) = (g^{-1})^{\mathrm{T}}$. Denote by $\rO_k := \GL^{\theta}_k$ the fixed point subgroup.
		
		The goal of this subsection is to recall how to parameterize the isomorphism classes of irreducible finite dimensional representations of $\rO_k$ in terms of representations of $\SO_k$ via Clifford theory\footnote{Meaning the theory of induction and restriction from a finite index normal subgroup.}.
		
		The Lie algebra of $\rO_k$ is $\so_k = \gl^{\theta}_k$. The group $\SO_k$ is the identity component of $\rO_k$, so in particular $\SO_k$ is a connected Lie group. It follows that all finite dimensional representations of $\SO_k$ come from the Lie algebra $\so_k$. The group $\rO_k$ has two connected components. We denote by $\mathrm{det}$ the determinant representation of $\rO_k$ restricted from $\mathrm{det}:\GL_k\rightarrow \C^{\times}$. The representation $\mathrm{det}$ is nontrivial for $\rO_k$ but restricts to the trivial representation for $\SO_k$. 
		
		\begin{remark}
			To describe finite dimensional irreducible representations of $\rO_k$, we use the root systems of \cref{subsubsect:so_k}. However, one must be careful when doing so. Note that $\rO_k=\GL_k^{\theta}$ preserves the form on $\C^k$ such that $(e_i, e_j) = \delta_{ij}$.  But the root systems in \cref{subsubsect:so_k} correspond to Cartan subalgebras of $\so_k$ which preserve a different bilinear form on $\C^k$ than the one preserved by $\gl_k^{\theta}$. 
		\end{remark}

		When $k=2n+1$ is odd, then $-\id \in \rO_k$ and $-\id\notin \SO_k$, so $\rO_k$ is the direct product $\SO_k \times \Z/2$. In particular, all irreducible representations of $\rO_k$ are isomorphic to the external tensor product of an irreducible representation of $\SO_k$ with an irreducible representation for $\mathbb{Z}/2$. For a partition $\bl$, let $V(\rO_k,\bl)$ be the representation of $\rO_k$ obtained from $V(\SO_k,\bl)$ by making $-\id$ act trivially. 
		
		\begin{notation}\label{N:partition-combinatorics}
			Denote by $\bl'$ the conjugate partition and by $\tilde{\bl}$ the partition obtained from $\bl$ by replacing the first column of $\bl'$ by $k-\bl_1'$. If $\bl = (l_1\ge \dots \ge l_k)$ and $l_k>0$, then write $\mathrm{len}(\bl) := k$.    
		\end{notation}

		Formally, denote by $V(\rO_k,\tilde{\bl}) := V(\SO_k,\bl) \otimes \det$, the representation of $\rO_k$ obtained from $V(\SO_k, \bl)$ by making $-\id$ act by $\mathrm{det}(-\id) = -1$.
		
		When $k=2n$ is even, there is the following classification.
		\begin{prop}\cite[Proposition 5.35]{ChengWang}
			A simple $\rO_{2n}$-module $L$ is one of the following:
			\begin{enumerate}
				\item $L$ is a direct sum of two irreducible modules of highest weight $(l_1,\ldots,l_{n-1},l_n)$ and $(l_1,\ldots,l_{n-1},-l_n)$ for $l_n >0$. In this case, $L \cong L \otimes \det$.
				\item For $\vec{l}=(l_1,\ldots,l_{n-1},0)$, there are exactly two inequivalent $L$ that restrict to the module $V(\SO_{2n},\vec{l})$ of $\SO_{2n}$. If one of them is $L$, then the other one is $L \otimes \det$. 
			\end{enumerate}
		\end{prop}
		
		In particular, there is a uniform classification of irreducible $\rO_k$-modules.
		\begin{prop}\cite[Proposition 5.36]{ChengWang}
			A complete set of pairwise non-isomorphism $\rO_k$-modules consists of $V(\rO_k,\bl)$, where $\bl$ is a partition of length no greater than $k$ with $\bl_1' + \bl_2' \leq k$.
		\end{prop}
		
		\subsubsection{Spinor representation}
		\label{subsect:spinor_representation}
		We denote by $\Cl_n$ the Clifford algebra generated by $\{\gamma_i,\gamma^{\dagger}_i|i=1,\ldots, n\}$ with the relations
		\[
		[\gamma^{\dagger}_i,\gamma_j]_+ = \delta_{ij}, \qquad [\gamma^{\dagger}_i,\gamma^{\dagger}_j]_+ = [\gamma_i,\gamma_j]_+ = 0,
		\]
		where $[-,-]_+$ is the anticommutator, i.e., $[x,y]_+= xy+yx$. The algebra $\Cl_n$ is $\Z/2$-graded by 
            \[
            \deg(\gamma^{\dagger}_i) = \deg(\gamma_{i}) = 1 \qquad \text{for all}\qquad i=1, \dots, n.
            \]
        For categorical constructions (in particular, tensor products), we consider $\Cl_n$ as a \emph{super} vector space.
		
		The relevance of the Clifford algebra for us is that the assignment
		\begin{equation}
			\label{eq::so_spinor_representation}
			M_{ij} \mapsto \frac{[\gamma^{\dagger}_i,\gamma_j]}{2}, \qquad M_{-i,j} \mapsto \frac{[\gamma_i,\gamma_j]}{2}, \qquad M_{i,-j} \mapsto \frac{[\gamma^{\dagger}_i,\gamma^{\dagger}_j]}{2}
		\end{equation}
		defines a Lie algebra homomorphism $\so_{2n} \hookrightarrow \Cl_n$.
		
		There is an action of $\Cl_n$ on skew polynomials $\fP_n$, see \cref{subsect:skew_polynomials}, defined by
		\begin{equation*}
			\gamma^{\dagger}_i \mapsto x_i \qquad \text{and} \qquad \gamma_i \mapsto \partial_i.
		\end{equation*}
		In fact, $\fP_n$ is the \emph{spinor} module of $\Cl_n$. Restricting the action of $\Cl_n$ to $\so_{2n}$, via Equation \eqref{eq::so_spinor_representation}, gives the structure of an $\so_{2n}$-representation on $\fP_n$. There is a compatible $\Z/2$-grading on $\fP_n$ determined by $\deg(\mathbf{1}) = 0$ and $\deg(x_i)=1$ for all $i$. In particular, $\fP_n$ is a direct sum of the even $\fP^+_n$ and odd $\fP^-_n$ component. Moreover each component is an irreducible $\so_{2n}$-representation. The corresponding lowest-weight vectors are 
		
		\begin{align*}
			&\fP_n^+ \ni \mathbf{1},\ \wt(\mathbf{1}) = \left(-\frac{1}{2},\ldots,-\frac{1}{2},-\frac{1}{2}\right), \\
			&\fP_n^- \ni x_n,\ \wt(x_n) = \left(-\frac{1}{2},\ldots,-\frac{1}{2},\frac{1}{2}\right).
		\end{align*}
		
		Recall the embedding $\gl_n \hookrightarrow \so_{2n}$ of equation \eqref{eq::gl_embedding_so}. One can easily see that the restriction of $\fP_n$ under this action is isomorphic to $\Lambda^{\bullet}(\C^n) \otimes \Det^{-\frac{1}{2}}$, where $\mathrm{det}$ is the determinant representation of \eqref{eq::determinant_representation}. Note that $M_{ii} = \gamma^{\dagger}_i \gamma_i - 1/2$ and $\gamma^{\dagger}_i \gamma_i$ corresponds to $E_{ii}$ in $\gl_n$. In particular, as an $\sl_n$-representation under the embedding $\sl_n \hookrightarrow \gl_n$, we have $\fP_n \cong \Lambda^{\bullet}(\C^n)$.

        In what follows, we will consider an automorphism $\tau$ of $\Cl_n$ defined by 
        \begin{equation}
        \label{eq::fourier_transform_clifford}
            \gamma^{\dagger}_i \mapsto \gamma_i, \qquad \gamma_i \mapsto \gamma^{\dagger}_i
        \end{equation}
        for all $i$. We call it the \emph{Fourier transform}. It is induced from the automorphism $\tau^{\fP}$ of the spinor module $\fP_n$ constructed as follows. Let $\{y_1,\ldots,y_n\}$ be a collection of anticommuting variables which anticommute with $x_1,\ldots,x_n$ as well. The \emph{Berezin integral} is defined by 
        \[
            \int y_I \mathrm{d} y = \begin{cases} 1,\ I = \{n > \ldots > 1\}, \\ 0,\ \mathrm{otherwise} \end{cases}
        \]
        where $y_I := y_{i_k} \ldots y_{i_1}$ for the index set $I = \{i_k > \ldots > i_1 \}$. Consider the odd analog of the Fourier transform
        \begin{equation}
        \label{eq::fourier_transform_spinor}
            \tau^{\fP} \colon \fP_n \rightarrow \fP_n, \qquad f(x) \mapsto \int e^{\sum_i x_i y_i} f(y) \mathrm{d} y,
        \end{equation}
        where the exponential factor is given by the usual exponential series. Denote by $\tau^* \fP_n$ the pullback of the spinor module under the automorphism $\tau$ of Clifford algebra. 
        \begin{prop}
        \label{prop:spinor_involution_intertwiner}
            The Fourier transform $\tau^{\fP}$ of $\fP_n$ induces an isomorphism $\fP_n \xrightarrow{\sim} \tau^* \fP_n$ of Clifford algebra modules.
        \begin{proof}
            We need to show that 
            \[
                \tau^{\fP} (\partial_i f) = x_i \tau^{\fP}(f), \qquad \tau^{\fP} (x_i f) = \partial_i \tau^{\fP}(f)
            \]
            for any $f\in \Lambda^{\bullet} \C^n$. For any function $g(y_1,\ldots,y_n)$ and any $i=1\ldots n$, we have
            \[
                \int \partial_{y_i} g(y_1,\ldots,y_n) \mathrm{d} y = 0. 
            \]
            Also, $\partial_{y_i} e^{\sum_j x_j y_j} = -x_i \cdot e^{\sum_j x_j y_j}$. Since $e^{\sum_j x_j y_j}$ has even degree, we get
            \[
                0 = \int \partial_{y_i} (e^{\sum_j x_j y_j} f(y)) \mathrm{d} y = -x_i \int (e^{\sum_j x_j y_j} g) \mathrm{d} y + \int e^{\sum_j x_j y_j} \partial_{i} f(y) \, \mathrm{d} y, 
            \]
            and so $\tau^{\fP} (\partial_i f) = x_i \tau^{\fP}(f)$. The second property follows from $\partial_{x_i} e^{\sum_j x_j y_j} = y_i e^{\sum_j x_j y_j}$. 
        \end{proof}
        \end{prop}
		
		\subsection{Representation}
		\label{subsect:so_duality_representation}
		Consider $\Cl_{nm}$ with generators $\{\gamma^{\dagger}_{i\alpha}, \gamma_{i\alpha}|i=1, \ldots, n, \ \alpha=1, \ldots, m\}$ satisfying 
		\[
		\gamma_{i\alpha}^{\dagger} \gamma_{j\beta} + \gamma_{j\beta} \gamma_{i\alpha}^{\dagger} = \delta_{ij} \delta_{\alpha\beta}.
		\]
		For each $\alpha$, consider the Clifford subalgebras
		\begin{align*}
			\Cl_n^{(\alpha)} = \langle \gamma^{\dagger}_{k\alpha},\gamma_{l\alpha}| k,l \in \{1 \ldots n\} \rangle, 
		\end{align*}
		We denote the corresponding image of the orthogonal Lie algebra by $\so_{2n}^{(\alpha)}$. Similarly, we define $\Cl_m^{(j)}$ to be the Clifford algebra generated by $\gamma^{\dagger}_{j\alpha}, \gamma_{j\beta}$ for all $\alpha,\beta$. 
		
		Recall the spinor module $\fP_{nm}$ from \cref{subsect:spinor_representation}. The action of $\Cl_{nm}$ on $\fP_{nm}$ gives rise to a total action of $\so_{2n}$ of $\fP_{nm}$ by restricting along the coproduct:
		\begin{align*}
			\so_{2n}&\rightarrow \so_{2n}^{(1)}\otimes \dots  \otimes \so_{2n}^{(m)}\hookrightarrow \Cl_{nm} \\ x &\mapsto  \sum_{\alpha} x^{(\alpha)}=: x^{\langle n \rangle} .
		\end{align*}
		Since all the generators of $\so_{2n}$ have even degree in the Clifford algebra, there is an isomorphism
		\[
		\fP_{nm} \cong \underbrace{\fP_n \otimes \ldots \otimes \fP_n}_{m}.
		\]
		of $\so_{2n}$-representations. 
		
		On the $\rO_m$-side, we first consider $\fP_{nm} = \Lambda^{\bullet}(\C^n \otimes \C^m)$ as a $\GL_m$-representation, see \cref{subsect:gl_duality_representation}. In particular, we have the embeddings 
		\[
		\gl^{(i)}_{m} \hookrightarrow \Cl_{nm}, \qquad E^{(i)}_{\alpha\beta} = \frac{[\gamma^{\dagger}_{i\alpha}, \gamma_{i\beta}]}{2}, \qquad \text{for $i=1, \dots, n$,}
		\]
		and the total $\gl_m$ action is given by the sum $E_{\alpha\beta} \mapsto \sum_i E_{\alpha\beta}^{(i)}$. As in \cref{subsubsect:o_m_reps}, we denote by $\theta$ the Cartan involution, $\rO_m := \GL_m^{\theta}$ the fixed point subgroup, and by $B_{\alpha\beta} = E_{\beta\alpha} - E_{\alpha\beta}$ the basis of the fixed point Lie subalgebra $\so_m := \gl_m^{\theta}$. 
		
		One can easily see that the actions of $\so_{2n}$ and $\rO_m$ commute. In fact, there is an analog of \cref{thm:gln_glm_duality}, for instance see \cite[Corollary 5.41]{ChengWang}. The irreducible highest-weight $\so_{2n}$-modules below are with respect to the following choice of the Borel subalgebra:
		
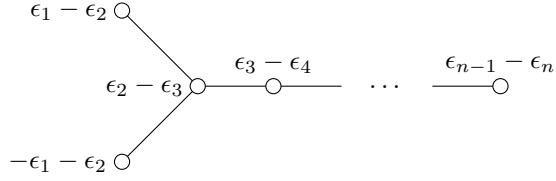
\begin{center}
			\begin{tikzpicture}
				\draw (0,1) circle (0.1cm);
				\draw (0,-1) circle (0.1cm);
				\draw (1,0) circle (0.1cm);
				\draw (0.07,0.93) -- (0.93,0.07);
				\draw (0.07,-0.93) -- (0.93,-0.07);
				\draw (1.1,0) -- (1.9,0);
				\draw (2,0) circle (0.1cm);
				\draw (2.1,0) -- (2.9,0);
				\node at (3.5,0) {\ldots};
				\draw (4.1,0) -- (4.9,0);
				\draw (5,0) circle (0.1cm);
				\node at (-0.7,1) {$\epsilon_1 - \epsilon_2$};
				\node at (-0.85,-1) {$-\epsilon_1 - \epsilon_2$};
				\node at (0.3,0) {$\epsilon_2 - \epsilon_3$};
				\node at (2,0.3) {$\epsilon_3 - \epsilon_4$};
				\node at (5,0.3) {$\epsilon_{n-1} - \epsilon_n$};
			\end{tikzpicture}
		\end{center}
		We continue to use Notation \ref{N:partition-combinatorics} when working with partitions.
		
		\begin{thm}
			\label{thm:o_duality}
			We have the following decomposition of $(\so_{2n},\rO_m)$-modules
			\[
			\fP_{nm} \cong \bigoplus_{\substack{\bl = (l_1 \geq \ldots \geq l_m) \\ \bl_1' + \bl_2' \leq m,\ \mathrm{len}(\bl') \leq n}} V \left(\so_{2n}, \sum_{i=1}^n (\bl_i' - \frac{m}{2}) \epsilon_i \right) \otimes V(\rO_m,\bl).
			\]
		\end{thm}

        Observe that we have an isomorphism $\Cl_{nm} \cong \Cl^{(1)}_m \otimes \ldots \otimes \Cl^{(n)}_m$ of superalgebras. Denote by $\tau_j$ the Fourier transform \eqref{eq::fourier_transform_clifford} along the $j$-th component. One can easily see that under the embedding $\so_{2n} \hookrightarrow \Cl_{nm}$, it corresponds to the Dynkin automorphism of $\so_{2n}$ from \eqref{eq::so2n_outer_auto}. On the $\rO_m$-side, it corresponds to a shifted version $\tau$ of the Cartan involution from \cref{subsect:chevalley_involution}:
        \[
            \tau \colon \U(\gl_m) \rightarrow \U(\gl_m),  \qquad E_{\alpha\beta} \mapsto -E_{\beta\alpha} + \delta_{\alpha,\beta}. 
        \]
        It is easy to check that $\tau_j$ induces the shifted Cartan involution on $\gl_m^{(j)}$. Similarly, under the isomorphism $\fP_{nm} \cong \fP_{m} \otimes \ldots \otimes \fP_m$, we denote by $\tau_j^{\fP}$ the spinor Fourier transform \eqref{eq::fourier_transform_spinor} of the $j$-th component.

		\subsection{Classical boundary duality}
		On the $\so_{2n}$-side, let us consider the \emph{differential dynamical} operators 
		\[
		\nabla^{\Cas,\langle n \rangle}_{\lambda_1},\ldots,\nabla^{\Cas,\langle n \rangle}_{\lambda_n}
		\]
		which are differential operators on $\h_{\so_{2n}}^{\reg}$ defined by:
		\begin{equation}
			\label{eq::so_casimir_connection}
			\nabla^{\Cas,\langle n \rangle}_{\lambda_i}(\lambda,\kappa) = \kappa \cdot \partial_{\lambda_i} - \sum_{j\neq i} \frac{M_{ij} M_{ji} - M_{ii}}{\lambda_i - \lambda_j} - \sum_{j \neq i} \frac{M_{i,-j} M_{-j,i} - M_{ii}}{\lambda_i + \lambda_j} -\frac{M_{ii}^2}{2\lambda_i}.
		\end{equation}
		One can see that these operators define a connection gauge equivalent (in the sense of \cref{def:gauge_transformation}) to the Casimir connection of \cite{FelderMarkovTarasovVarchenko} for $\so_{2n}$. 
		
		On the $\rO_m$-side, consider the \emph{boundary differential KZ operators} 
		\[
		\nabla^{\bKZ,\langle m \rangle}_{z_{1}},\ldots,\nabla^{\bKZ,\langle m \rangle}_{z_{n}}
		\]
		which are differential operators on $\C^n\setminus \{z_i= \pm z_j|i\ne j\in \{1, \dots, n\}\}$ defined by:
		\begin{equation}
			\label{eq::bkz_connection}
			\nabla^{\bKZ,\langle m \rangle}_{z_{i}} (z,\kappa)= \kappa\cdot \partial_{z_i} - \sum_{j\neq i} \frac{\Omega_{\gl_n}^{(i,j)} - m/2}{z_i - z_j} - \sum_{j\neq i} \frac{(1\otimes \tau) \Omega^{(i,j)}_{\gl_m} - m/2}{z_i + z_j} - \frac{\Omega_{\so_m}^{(i)} - m^2/8}{z_i}.
		\end{equation}
		Here, $\Omega_{\gl_n}^{(i,j)}$ is the split Casimir operator as in \cref{sect::gl_duality} and $\Omega^{(i)}_{\so_m}$ is the Casimir element $\Omega_{\so_m} = \frac{1}{2} \sum_{\alpha < \beta} B_{\alpha\beta}^2$. 
		
		\begin{remark}
			The boundary differential $KZ$ operators define a flat connection called the \emph{cyclotomic KZ connection} introduced in \cite{EnriquezEtingof} for an arbitrary finite-order automorphism of $\g$; see also \cite{BrochierCyclotomicKZ}.
		\end{remark}
		
		We can identify $\mathfrak{h}_{\so_{2n}}\rightarrow \C^n$ via $\lambda_i\mapsto z_i$. Under this we have the following theorem.
		
		\begin{thm}
			\label{prop::so_bkz_cas_duality}
			The following relation holds
			\[
			\nabla^{\bKZ,\langle m \rangle}_{z_{i}} (z,\kappa) = \nabla^{\Cas,\langle n \rangle}_{\lambda_i}(-z,-\kappa). 
			\]
			\begin{proof}
				Similar to the discussion in Remark \ref{R:down-to-earth}, it suffices to show the operators in the numerator have the same action on $\fP_{nm}$. Let us compare both sides case-by-case:
				\begin{itemize}
					\item Consider the operator $M_{ij} M_{ji} - M_{ii}$. Its action on $\fP_{nm}$ is equal to
					\begin{align*}
						\sum_{\alpha,\beta} \gamma^{\dagger}_{i\alpha} \gamma_{j\alpha} \gamma^{\dagger}_{j\beta} \gamma_{i\beta}  - \sum_{\alpha} \gamma^{\dagger}_{i\alpha} \gamma_{i\alpha} + \frac{m}{2}= -\sum_{\alpha,\beta} \gamma^{\dagger}_{i\alpha} \gamma_{i\beta} \gamma^{\dagger}_{j\beta} \gamma_{j\alpha}  + \frac{m}{2}.
					\end{align*}
					At the same time, the action of $\Omega^{(i,j)}_{\gl_m}$ is 
					$\sum_{\alpha \neq \beta} \gamma^{\dagger}_{i\alpha} \gamma_{i\beta} \gamma^{\dagger}_{j\beta} \gamma_{j\alpha}$.
					Therefore,
					\[
					\Omega^{(i,j)}_{\gl_m} - \frac{m}{2} = -(M_{ij} M_{ji} - M_{ii}). 
					\]
					\item Consider the operator $M_{i,-j} M_{-j,i} - M_{ii}$. Its action is equal to  
					\[
					\sum_{\alpha,\beta} \gamma^{\dagger}_{i\alpha} \gamma^{\dagger}_{j\alpha} \gamma_{j\beta} \gamma_{i\beta}  - \sum_{\alpha} \gamma^{\dagger}_{i\alpha}\gamma_{i\alpha} + \frac{m}{2}= \sum_{\alpha,\beta} \gamma^{\dagger}_{i\alpha} \gamma_{i\beta} \gamma^{\dagger}_{j\alpha} \gamma_{j\beta}- \sum_{\alpha} \gamma^{\dagger}_{i\alpha}\gamma_{i\alpha} + \frac{m}{2}.
					\]
					At the same time, we have
					\[
					(1 \otimes \tau)\Omega^{(i,j)}_{\gl_m} = -\sum_{\alpha,\beta} \gamma^{\dagger}_{i\alpha} \gamma_{i\beta} \gamma^{\dagger}_{j\alpha} \gamma_{j\beta} + \sum_{\alpha} \gamma^{\dagger}_{i\alpha} \gamma_{i\alpha}.
					\]
					So, 
					\[
					(1\otimes \tau) \Omega^{(i,j)}_{\gl_m} - m/2 = -(M_{i,-j}M_{-j,i} - M_{ii}).
					\]
					
					\item Finally, let us compute the action of the operator $\Omega_{\so_m}^{(i)}$:
					\begin{align*}
						\Omega_{\so_m}^{(i)} = -\frac{1}{4} \sum_{\alpha \neq \beta} (E_{\alpha\beta}^{(i)} - E_{\beta\alpha}^{(i)})^2 = \frac{1}{4}\sum_{\alpha\neq\beta} (\gamma^{\dagger}_{i\alpha} \gamma_{i\beta} \gamma^{\dagger}_{i\beta} \gamma_{i\alpha} + \gamma^{\dagger}_{i\beta} \gamma_{i\alpha} \gamma^{\dagger}_{i\alpha} \gamma_{i\beta})^2 = \\ 
						=-\frac{1}{2} \sum_{\alpha \neq \beta} \gamma^{\dagger}_{i\alpha} \gamma_{i\alpha} \gamma^{\dagger}_{i\beta} \gamma_{i\beta} + \frac{m-1}{2} \sum_{\alpha} \gamma^{\dagger}_{i\alpha} \gamma_{i\alpha}.
					\end{align*}
					At the same time, consider $M_{ii} = \sum_{\alpha} \gamma^{\dagger}_{i\alpha} \gamma_{i\alpha} - m/2$. Then one can check that 
					\[
					\Omega^{(i)}_{\so_m} - m^2/8 = -\frac{1}{2} M_{ii}^2.
					\]
				\end{itemize}
			\end{proof}
		\end{thm}
		
		\subsection{Classical duality}
		
		On the $\so_{2n}$ side we have $\Omega_{\so_{2n}}^{(\alpha,\beta)}$,
		the \emph{split Casimir operator} for $\so_{2n}$, acting on the $(\alpha,\beta)$-components of $\fP_{nm} \cong \fP_n \otimes \ldots \otimes \fP_n$:
		\[
		\Omega_{\so_{2n}}^{(\alpha,\beta)} = \frac{1}{2}\sum_{i,j=1}^n M_{ij}^{(\alpha)} M_{ji}^{(\beta)} + \frac{1}{2} \sum_{i<j} (M_{i,-j}^{(\alpha)} M_{-j,i}^{(\beta)} + M_{-j,i}^{(\alpha)}M_{i,-j}^{(\beta)}).
		\]
		
		Now we consider the \emph{differential KZ operators} 
		\[
		\nabla^{\KZ,\langle n \rangle}_{z_{1}},\ldots,\nabla^{\KZ,\langle n \rangle}_{z_{m}}
		\]
		which are differential operators on $\C^m\setminus \{z_{\alpha}\ne z_{\beta}|\alpha\ne \beta\in \{1, \dots, m\}\}$ defined by:
		\begin{equation}
			\nabla^{\KZ,\langle n \rangle}_{z_{\alpha}} = \kappa\cdot \partial_{z_{\alpha}} - \sum_{\beta \neq \alpha} \frac{\Omega_{\so_{2n}}^{(\alpha,\beta)} - n/8}{z_{\alpha} - z_{\beta}}.
		\end{equation}
		
		On the $\rO_m$-side, consider a slight modification of the \emph{differential boundary dynamical operators} of \cref{D:boundary-Casimir}
		\[
		\nabla^{\bCas,\langle m \rangle}_{\lambda_{1}},\ldots,\nabla^{\bCas,\langle m \rangle}_{\lambda_{m}}
		\]
		which are differential operators on $\mathfrak{h}_{\gl_m}^{\reg}$ defined by:
		\begin{equation}
			\nabla^{\bCas,\langle m \rangle}_{\lambda_{\alpha}}(z,\kappa) = \kappa\cdot \partial_{\lambda_{\alpha}} + \frac{1}{4}\sum_{\beta \neq \alpha} \frac{B_{\alpha\beta}^2}{z_{\alpha} - z_{\beta}}.
		\end{equation}

		Identifying $\C^m\rightarrow \h_{\gl_m}$ by $z_{\alpha}\mapsto h_{\alpha}$, the following proposition can be proved analogously to \cref{prop::so_bkz_cas_duality}.
		\begin{thm}
			\label{thm:orthogonal_duality_bcas_kz}
			The following relation holds
			\[
			\nabla^{\KZ,\langle n \rangle }_{z_{\alpha}}(z,\kappa) = \nabla^{\bCas,\langle m \rangle}_{\lambda_{\alpha}}(-z,-\kappa).
			\]
		\end{thm}
		
		\subsection{Quantum duality}
		\label{subsect::quantum_orthogonal_duality}
		In this subsection, we show the difference analog of \cref{prop::so_bkz_cas_duality} and \cref{thm:orthogonal_duality_bcas_kz}. 
		
		\subsubsection{Yangian}\label{ss:sotn-Yangian}
		
		Unless otherwise stated, we will follow \cite{KarakhanyanKirschner} in conventions, see also \cite{ArnaudonMolevRagoucy}.
		
		Denote by $Q = \sum_{i,j=-n}^{n} E_{ij} \otimes E_{-i,-j} \in \End(\C^{2n} \otimes \C^{2n})$ and define the $R$-matrix
		\[
		R(u) = \id + \frac{P}{u} - \frac{Q}{u+n-1}.
		\]
		The \emph{extended Yangian} $\rX(\so_{2n})$ is generated by $t_{ij}^{(k)}$ for $i,j=-n\ldots n$ and $k \geq 1$ subject to the RTT relation: consider the power series $t_{ij}(u) = \delta_{ij} + \sum_{k\geq 0} t_{ij}^{(k)} u^{-k} \in \rX(\so_{2n})[[u^{-1}]]$ and combine them in the matrix $T(u) = (t_{ij}(u))_{i,j=-n}^n$, then the RTT relation reads
		\[
		R_{12}(u-v) T_1(u) T_2(v) = T_2(v) T_1(u) R_{12} (u-v).
		\]
		As in the case of $\gl_n$, it is a Hopf algebra with the coproduct $\Delta(t_{ij})(u) = \sum_{k=-n}^n t_{ik}(u) \otimes t_{kj}(u)$. Also, for any $z\in \C$ and $f(u) \in 1 + u^{-1} \C[[u^{-1}]]$, the operations
		\begin{align*}
			f(u)&\colon T(u) \mapsto f(u) T(u), \\
			\tau_z&\colon T(u) \mapsto T(u-z)
		\end{align*}
		define automorphisms of $\rX(\so_{2n})$.
		
		There is an embedding $\so_{2n} \hookrightarrow \rX(\so_{2n})$ via $M_{ij} \mapsto \frac{1}{2} (t_{-j,-i}^{(1)} - t_{ij}^{(1)})$, see \cite[Proposition 3.11]{ArnaudonMolevRagoucy}. Unfortunately, unlike in the case of $\gl_n$, there is no evaluation homomorphism $\rX(\so_{2n}) \rightarrow \U(\so_{2n})$. However, it is possible to lift the spinor representation $\fP_n$ via
		\[
		t_{ij}(u) = \delta_{ij} - M_{ij}^{\fP} u^{-1}.
		\]
		where $M_{ij}^{\fP}$ are given by \eqref{eq::so_spinor_representation}. We denote by $\fP_n(z)$ the pullback of this module via the shift automorphism $\tau_z$. 
		
		As in \cref{subsubsect:gl_yangian}, there is an intertwining operator $\check{\cR}_{12}(z_1-z_2)\colon \fP_n(z_1) \otimes \fP_n(z_2) \rightarrow \fP(z_2) \otimes \fP(z_1)$ that we call the \emph{$R$-operator} (as opposed to an $R$-matrix which does not permute factors). It satisfies 
		\begin{align}
			\label{eq::r_op_so2n_comm}
			\check{\cR}_{12}(u)(M_{ij}^{\fP,(1)} + M_{ij}^{\fP,(2)}) &= (M_{ij}^{\fP,(1)} + M_{ij}^{\fP,(2)}) \check{\cR}_{12}(u), \\
			\label{eq::r_op_yang_comm}
			\check{\cR}_{12}(u)(-uM_{ij}^{\fP,(2)} + \sum_{k=-n}^n M_{ik}^{\fP,(1)}M_{kj}^{\fP,(2)}) &= (-uM_{ij}^{\fP,(1)} + \sum_{k=-n}^n M_{ik}^{\fP,(1)}M_{kj}^{\fP,(2)})\check{\cR}_{12}(u).
		\end{align}

		By the $(\so_{2n},\rO_2)$-duality of \cref{thm:o_duality}, it follows from \eqref{eq::r_op_so2n_comm} that $\check{\cR}_{12}(u)$ depends only on 
		\[b^{\langle 2 \rangle}_{12} = -\sum_i (\gamma^{\dagger}_{i1}\gamma_{i2} + \gamma_{i1} \gamma^{\dagger}_{i2}). \]
		We indicate this dependence by $\check{\cR}_{12}(u) = \check{\cR}_{12}(u|b^{\langle 2 \rangle}_{12})$. Observe that by the duality, the operator $b^{\langle 2 \rangle}_{12}$ is diagonalizable with integral eigenvalues. 
		\begin{thm}\cite[Section 3]{KarakhanyanKirschner}
			\label{thm:spinorial_r_operator}
			The $R$-operator has the form
			\begin{equation*}
				\check{\cR}_{12}(u|b^{\langle 2 \rangle}_{12}) = r(u)\cdot \B \left( \frac{ib^{\langle 2 \rangle}_{12} + 1 +u}{2}, -u\right)^{-1} 
			\end{equation*}
			for an arbitrary function $r(u)$.
		\end{thm}
		
		Denote by $\deg$ the degree operator on $\fP_n$ such that
		\[
		\deg(x_{i_1} \ldots x_{i_k}) = k \cdot x_{i_1} \ldots x_{i_k}.
		\]
		Observe that the operator $(-1)^{\deg}$ commutes with the $\so_{2n}$-action as the latter has even degree. In what follows, we will fix the following normalization of the $R$-operator.
		\begin{defn}
			\label{def:spinorial_r_operator}
			The \defterm{spinorial $R$-operator} is 
			\[
			\check{\cR}_{12}(t) = \frac{(-1)^{\deg^{(1)}\deg^{(2)}+1} (-2i)^{-t}}{t}  \cdot \B \left( \frac{ib^{\langle 2 \rangle}_{12} + 1 +t}{2}, -t\right)^{-1}.
			\]
		\end{defn}
		
		\subsubsection{Twisted Yangian}\label{ss:twisted-Yang}
		On the $\rO_m$-side, we will need the so-called \emph{twisted Yangian} $\rY^{\tw}(\so_m)$; we refer the reader to \cite{MolevYangians} for statements and proofs. Note that we use slightly different conventions on the $R$-matrix. 
		
		Denote by $R^t(u) = \id + \frac{\sum E_{ij} \otimes E_{ij}}{u}$ the transposition of the $R$-matrix from \eqref{eq::yang_r_matrix}. The twisted Yangian $\rY^{\tw}(\so_m)$ is generated by $s_{ij}^{(k)}$ for $i,j=1\ldots m$ and $k\geq 1$ subject to the following relation: consider the generating series
		\[
		s_{ij}(u) = \delta_{ij} + \sum_{k\geq 1} s_{ij}^{(k)} u^{-k},
		\]
		then the matrix $S(u) = (s_{ij}(u))_{i,j=1}^m$ satisfies the \emph{reflection equation}:
		\begin{equation}
			R(v-u) S_1(u) R^t(u+v) S_2(v) = S_2(v) R^t(u+v) S_1(u) R(v-u).
		\end{equation}
		
		There is an embedding $\rY^{\tw}(\so_{m}) \hookrightarrow \rY(\gl_m)$ into the Yangian of $\gl_m$ given by $S(u) \mapsto T(u) T^t(-u)$, see \cite[Theorem 2.4.3]{MolevYangians} (in fact, it is a coideal inside $\rY(\gl_m)$). In particular, any $\rY(\gl_m)$-representation is a $\rY^{\tw}$-representation.
		
		We will need the following simple fact.
		\begin{prop}
			\label{prop:irred_rep_twisted_yangian}
			Let $V$ be a $\gl_m$ representation. Assume that it is irreducible as $\so_{m}$-representation, where $\so_m = \gl_m^{\theta}$. Then the pullback of the evaluation representation $V(z)$ to $\rY^{\tw}(\so_{m})$ is irreducible as well.
			\begin{proof}
				Observe that the action of generating series $s_{ij}(u)$ is given by
				\[
				s_{ij}(u) = \delta_{ij} + \frac{E_{ji} - E_{ij}}{u-z} - \frac{\sum_k E_{ki}E_{kj}}{(u-z)^2}.
				\]
				In particular, the first Taylor coefficients $s_{ij}^{(1)}$ generate $\so_m$ inside $\gl_m$. Since $V$ is irreducible under $\so_{m}$-action, it is irreducible under $\rY^{\tw}(\so_{m})$ as well.
			\end{proof}
		\end{prop}
		
		Recall the universal $R$-matrix $\mathfrak{R}$ of the Yangian $\rY(\gl_m)$. Likewise, for the twisted Yangian, there exists a universal $S$-matrix $\fS$, see \cite{MudrovTwistedYangians}. It commutes with the $\rY^{\tw}(\so_m)$-action and satisfies the reflection equation
		\begin{equation}
			\label{eq::uninversal_reflection_equation}
			\fR_{12} \fS_1 \fR^{\tilde{\theta}_1}_{12} \fS_2 = \fS_2 \fR^{\tilde{\theta}_1}_{12} \fS_1 \fR_{12}, 
		\end{equation}
		where $\fR^{\tilde{\theta}_1} = (\tilde{\theta} \otimes 1) \fR$ is the $R$-matrix twisted by the automorphism $\tilde{\theta}$ from \eqref{eq::yangian_automorphisms}. 
		
		We will need a variant of the twisted $R$-matrix. Let $V_1,V_2$ be irreducible $\gl_m$-modules with the highest (resp. the lowest) weight vector $v_1^h \in V_1$ (resp. $v_2^l \in V_2$). Denote by $\cR_{V_1 V_2}^{t_2}(t)$ the unique endomorphism of $V_1 \otimes V_2$ satisfying
		\begin{equation}
			\label{eq::twisted_r_matrix_properties}
			\begin{gathered}
				[\cR_{V_1 V_2}^{t_2}(t), x \otimes 1 + 1 \otimes \theta(x)] = 0, \\
				\cR_{V_1 V_2}^{t_2}(t) \cdot (t E_{ij} \otimes 1 - \sum_{k=1}^n E_{ik} \otimes E_{jk}) = (t E_{ij} \otimes 1 - \sum_{k=1}^n E_{ki} \otimes E_{kj}) \cdot \cR_{V_1 V_2}^{t_2}(t), \\
				\cR_{V_1 V_2}^{t_2}(t) v_1^h \otimes v_2^l = v_1^h \otimes v_2^l. 
			\end{gathered}
		\end{equation}
		In other words, these are the relations \eqref{eq::r_matrix_properties} to which we applied the automorphism $1\otimes \theta$.
		
		We will need the following form of the reflection equation.
		
		\begin{prop}
			\label{prop:reflection_equation}
			Assume that $V_1,V_2$ are $\gl_m$-modules that are irreducible as $\so_m$-modules. Then 
			\[
			\cR_{V_1 V_2} (z_1 - z_2) \cR^{t_2}_{V_1 V_2} (z_1+z_2) = \cR^{t_2}_{V_1 V_2} (z_1+z_2) \cR_{V_1 V_2} (z_1 - z_2).
			\]
			\begin{proof}
				Consider the universal reflection equation \eqref{eq::uninversal_reflection_equation} acting on the tensor product $V_1(z_1) \otimes V_2(z_2)$ of $\rY(\gl_m)$-representations. Since the universal $S$-matrix commutes with $\so_m$, we conclude by \cref{prop:irred_rep_twisted_yangian} and irreducibility of $\so_m$-modules $V_1$ and $V_2$ that the universal $S$-matrices act by scalar functions. Therefore, the reflection equation is reduced to
				\begin{equation}
					\label{prop_eq:reduced_re}
					\fR_{12} \fR^{\tilde{\theta}_1}_{12} = \fR^{\tilde{\theta}_1}_{12} \fR_{12}.
				\end{equation}
				Observe that the action of $(\tilde{\theta} \otimes \tilde{\theta}) \fR$ on $V_1(z_1) \otimes V_2 (z_2)$ satisfies the properties \eqref{eq::r_matrix_properties} for the permuted $R$-matrix $\cR_{V_2 V_1}(t)$ except for, possibly, the normalization condition. In particular, it is proportional to $\cR_{V_2 V_1} (z_2-z_1)$. By unitarity \eqref{eq::gl_rmatrix_unitarity}, the latter is equal to $\cR_{V_1 V_2}(z_1 - z_2)^{-1}$. Likewise, the operator $(1 \otimes \tilde{\theta}) \fR$ acting on $V_1(z_1) \otimes V_2(z_2)$ satisfies the properties \eqref{eq::twisted_r_matrix_properties} up to normalization. Therefore, applying $\tilde{\theta} \otimes \tilde{\theta}$ to both sides of \eqref{prop_eq:reduced_re} and rearranging the terms, we have
				\[
				\cR_{V_1 V_2} (z_1 - z_2) \cR^{t_2}_{V_1 V_2} (z_1+z_2) = \cR^{t_2}_{V_1 V_2} (z_1+z_2) \cR_{V_1 V_2} (z_1 - z_2),
				\]
				as required. 
			\end{proof}
		\end{prop}
		
		\subsubsection{Duality}
		Recall from \cref{subsubsect:so_k} that the positive roots of $\so_{2n}$ are of the form $\epsilon_i \pm \epsilon_j$ for $0<i<j$. Consider the root $\epsilon_i - \epsilon_j$. We have 
		\begin{equation}
			\label{eq::clifford_root_gens}
			e_{\epsilon_i - \epsilon_j} = M_{ij}, \qquad e_{\epsilon_i - \epsilon_j} = M_{ji}, \qquad (\epsilon_i - \epsilon_j)^{\vee} = M_{ii} - M_{jj}.
		\end{equation}
		By construction, they come from the embedding $\sl_n \hookrightarrow \gl_n \hookrightarrow \so_{2n}$. Recall from \cref{subsubsect:o_m_reps} that we have an isomorphism $\fP_{nm} \cong \Lambda^{\bullet}(\C^n \otimes \C^m)$ of $\sl_n$-representations. 
		
		Recall the dynamical $B$-operators \eqref{eq::b_operator}. Because of the isomorphism, we see that the action of the $B$-operators $\cB_{\epsilon_i - \epsilon_j}(t)$ for $\so_{2n}$ coincides with those for $\gl_n$ (since they depend only on the $\sl_n$-structure). On the $\rO_m$-side, consider the $R$-matrix $\cR_{ij}^{\langle m \rangle}(t) $ for $\gl_m$ as in \cref{subsubsect:gl_yangian}. Denote by 
		\begin{equation}
			\label{eq::so_qduality_factor}
			C_{ij}^{\langle n \rangle} (t) = \frac{\Gamma(t + M_{ii} + \frac{m}{2} +1) \Gamma(t - M_{jj} - \frac{m}{2})}{\Gamma(t + M_{ii} - M_{jj} + 1)\Gamma(t)}.
		\end{equation}
		The following proposition is a tautological reformulation of \cref{thm:gln_glm_duality}.
		\begin{prop}
			\label{thm:son_dyn_om_spec_1}
			The following relation holds:
			\[
			\cR^{\langle m \rangle}_{ij} (t) = C^{\langle n \rangle}_{ij}(t) \cB^{\langle n \rangle}_{\epsilon_i - \epsilon_j} (-t).
			\]
		\end{prop}
		
		Consider a positive root $\epsilon_i + \epsilon_j$ for $i<j$. Recall the involution $\tau_j$ of \eqref{eq::fourier_transform_spinor}. As it was discussed in \cref{subsect:so_duality_representation}, it corresponds to (a version of) the Dynkin automorphism of $\so_{2n}$, so,
		\[
		\tau_j( e_{\epsilon_i - \epsilon_j}) = e_{\epsilon_i + \epsilon_j}, \qquad \tau_j( f_{\epsilon_i - \epsilon_j}) = f_{\epsilon_i + \epsilon_j}, \qquad \theta((\epsilon_i - \epsilon_j)^{\vee}) = (\epsilon_i + \epsilon_j)^{\vee},
		\]
		in particular, $\tau_j(\cB_{\epsilon_i - \epsilon_j}^{\langle n \rangle}(t)) = \cB_{\epsilon_i + \epsilon_j}^{\langle n \rangle}(t)$. Also, recall the automorphism $\tau_j^{\fP}$ of $\fP_{nm}$ from \eqref{eq::fourier_transform_spinor}. It follows from \cref{prop:spinor_involution_intertwiner} that $\mathrm{Ad}_{\tau_j^{\fP}}(\Gamma) = \tau_j(\Gamma)$ for any $\Gamma \in \Cl_{nm}$. In particular, we have $\mathrm{Ad}_{\tau_j^{\fP}}(\cB_{\epsilon_i - \epsilon_j}^{\langle n \rangle}(t)) = \cB_{\epsilon_i + \epsilon_j}^{\langle n \rangle}(t)$. 
		
		On the other side, consider the operator $\mathrm{Ad}_{\tau_j^{\fP}}(\cR^{\langle m \rangle}_{ij} (t))$. Also, recall the twisted $R$-matrix $\cR^{t_2,\langle m \rangle}_{ij}(t)$ satisfying the properties \eqref{eq::twisted_r_matrix_properties}.
		
		\begin{prop}
			We have
			\[
			\mathrm{Ad}_{\tau_j^{\fP}}(\cR^{\langle m \rangle}_{ij} (t)) = \cR^{t_2,\langle m \rangle}_{ij}(t-1).
			\]
			\begin{proof}
				Let us apply the operator $\mathrm{Ad}_{\tau^{\fP}_{j}}$ to both sides of the first two defining properties of the $R$-matrix \eqref{eq::r_matrix_properties} (not for the twisted one). They read
				\begin{equation*}
					\begin{gathered}
						[\mathrm{Ad}_{\tau_j^{\fP}}(\cR^{\langle m \rangle}_{ij} (t)), x \otimes 1 + 1 \otimes \theta(x)] = 0, \\
						\mathrm{Ad}_{\tau_j^{\fP}}(\cR^{\langle m \rangle}_{ij} (t)) \cdot ((t-1) E_{ij} \otimes 1 - \sum_k E_{ik} \otimes E_{jk}) = ((t-1) E_{ij} \otimes 1 - \sum_k E_{ki} \otimes E_{kj}) \cdot \mathrm{Ad}_{\tau_j^{\fP}}(\cR^{\langle m \rangle}_{ij} (t)),
					\end{gathered}
				\end{equation*}
				which coincide with the first two properties of \eqref{eq::twisted_r_matrix_properties} after the change of variables $t \mapsto t-1$. Likewise, observe that for any $k$, the vector $x_{j1} \ldots x_{jk}$ is a highest-weight vector of $\Lambda^k \C^m$, and its image under $\tau_j$ is proportional to $x_{j,k+1} \ldots x_{j,m}$, which is a lowest-weight vector in $\Lambda^{m-k} \C^m$. By the third property of \eqref{eq::r_matrix_properties}, we conclude that $\mathrm{Ad}_{\tau_j^{\fP}}(\cR^{\langle m \rangle}_{ij} (t))$ satisfies the normalization condition from \eqref{eq::twisted_r_matrix_properties} as well. Therefore, by uniqueness, it coincides with $\cR^{t_2, \langle m \rangle}_{ij}(t-1)$.
			\end{proof}
		\end{prop}

		Denote by 
		\[
		\hat{C}^{\langle n \rangle}_{ij}(t) = \frac{\Gamma(t + M_{ii} + \frac{m}{2}+1) \Gamma(t+ M_{jj} - \frac{m}{2} )}{\Gamma(t + M_{ii} + M_{jj}+1)\Gamma(t)} 
		\]
		the $\tau_j$-twisted factor \eqref{eq::so_qduality_factor}. The following proposition follows directly from \cref{thm:son_dyn_om_spec_1} by applying $\mathrm{Ad}_{\tau^{\fP}_j}$ on both sides.
		\begin{prop}
			\label{prop:twisted_rmatrix_equal_dyn_operator}
			The following relation holds:
			\[
			\cR^{t_2,\langle m \rangle}_{ij} (t-1) = \hat{C}^{\langle n \rangle}_{ij}(t) \cB^{\langle n \rangle}_{\epsilon_i + \epsilon_j} (-t).
			\]
		\end{prop}
		
		The boundary analog of the dynamical difference operators \eqref{eq::gl_dyn_diff_ops} is given as follows, see \cite[Section 5]{CherednikQKZ}:
		\begin{align}
			\begin{split}
				{}^{b} X^{\langle n \rangle}_{i}(\lambda,\kappa) = &(\cB^{\langle n \rangle}_{\epsilon_{i} - \epsilon_n}(\lambda_{i} - \lambda_n) \ldots \cB^{\langle n \rangle}_{\epsilon_{i} - \epsilon_{i+1}}(\lambda_{i} - \lambda_{i+1}))^{-1} \cdot \\
				\cdot & (\cB^{\langle n \rangle}_{\epsilon_1 + \epsilon_i}(\lambda_1 + \lambda_i) \ldots \cB^{\langle n \rangle}_{\epsilon_n + \epsilon_i} (\lambda_n + \lambda_1))^{-1}\cdot \\
				\cdot &\cB^{\langle n \rangle}_{\epsilon_1 - \epsilon_{i}}(\lambda_1 - \lambda_{i} - \kappa) \ldots \cB^{\langle n \rangle}_{\epsilon_{i-1} - \epsilon_{i}}(\lambda_{i-1} - \lambda_{i} - \kappa) T_{\lambda_{i}}.
			\end{split}
		\end{align}
		
		Here, we put $\cB^{\langle n \rangle}_{\epsilon_j + \epsilon_i}(t) := \cB^{\langle n \rangle}_{ij}(t)$ for $i<j$ and $\cB^{\langle n \rangle}_{\epsilon_i + \epsilon_i}(t) = 1$. By \cite[Theorem 5.1]{CherednikQKZ}, these operators define a consistent system, in other words, they pair-wise commute for all $i,j$. Similarly, one can define the boundary analog of the quantized Knizhnik-Zamolodchikov operators \eqref{eq::gl_qkz_ops}, which quantizes the boundary KZ connection \eqref{eq::bkz_connection}:
		\begin{align}
			\begin{split}
				{}^{b} Z^{\langle m \rangle}_{i}(z,\kappa) = &(\cR^{\langle m \rangle}_{i n}(z_{i} - z_{n})\ldots \cR^{\langle m \rangle}_{i,i+1}(z_{i} - z_{i+1}))^{-1}  \cdot \\
				\cdot & (\cR^{t_2,\langle m \rangle}_{1i}(z_1 + z_i) \ldots  \cR_{i-1,i}^{t_2,\langle m \rangle}(z_{i-1} + z_i) \cR_{i,i+1}^{t_2, \langle m \rangle}(z_{i} + z_{i+1}) \ldots \cR^{t_2,\langle m \rangle}_{in} (z_i + z_n))^{-1}\cdot  \\
				\cdot & \cR^{\langle m \rangle}_{1i}(z_1 - z_{i} - \kappa) \ldots \cR^{\langle m \rangle}_{i-1,i}(z_{i-1}- z_{i} - \kappa) \cdot T_{z_{i}}.
			\end{split}
		\end{align}
		
		\begin{prop}
			The boundary qKZ system is consistent.
			\begin{proof}
				In the notations of \cite[(5.12)]{CherednikQKZ}, we put $R^{--}_{ij}(t):= \cR^{\langle m \rangle}_{ij}(t)$ and 
				\[
				R^{-+}_{ij}(t):= \begin{cases}
					\cR^{t_2,\langle m \rangle}_{ij}(t),\ i<j \\
					\cR^{t_2,\langle m \rangle}_{ji}(t),\ i>j
				\end{cases}.
				\]
				Then the conditions of (5.5) and (5.8) in \emph{loc. cit.} are satisfied. Also, the condition for type $D$ after (5.12)
				\[
				[R^{--}_{ij}(t), R_{ij}^{-+}(s)] = 0\ \mathrm{for\ any}\ s,t
				\]
				is satisfied as well: for $i<j$, observe that any exterior power $\Lambda^k \C^m$ is an irreducible representation of $\so_m$, therefore, the version of the reflection equation of \cref{prop:reflection_equation} gives the commutation relation for a suitable choices of parameters $z_1,z_2$. Likewise, for $i>j$, we have $R_{ij}^{-+}(t) = \cR^{t_2,\langle m \rangle}_{ji}(t)$ and $R^{--}_{ij}(t) = \cR^{\langle m \rangle}_{ij}(t) = \cR^{\langle m \rangle}_{ji}(-t)^{-1}$ by the unitarity relation. So, we need to show that
				\[
				\cR^{\langle m \rangle}_{ji}(-s)^{-1}  \cR^{t_2,\langle m \rangle}_{ji}(t) = \cR^{t_2,\langle m \rangle}_{ji}(t) \cR^{\langle m \rangle}_{ji}(-s)^{-1}
				\]
				for any $s,t$, equivalently,
				\[
				\cR^{t_2,\langle m \rangle}_{ji}(t) \cR^{\langle m \rangle}_{ji}(-s) = \cR^{\langle m \rangle}_{ji}(-s) \cR^{t_2,\langle m \rangle}_{ji}(t).
				\]
				For a suitable choices of parameters $z_1,z_2$, this is equivalent to the reflection equation of \cref{prop:reflection_equation}.
				
				Therefore, according to \cite[Theorem 5.1]{CherednikQKZ}, the system is consistent.
			\end{proof}
		\end{prop}
		
		For a collection of variables $z = (z_1,\ldots,z_n)$ and $\alpha\in \C$, denote by $z+\alpha:=(z_1+\alpha,\ldots,z_n+\alpha)$. Then, \cref{prop:twisted_rmatrix_equal_dyn_operator} implies the following result.
		\begin{thm}
			\label{thm:so_dynamical_bkz_bispectrality}
			The following relation holds:
			\[
			{}^{b} Z_i^{\langle m \rangle}(z,\kappa) = \frac{\prod_{1\leq j < i} C_{ji}^{\langle n \rangle}(z_{j} - z_{i} - \kappa)}{\prod_{m \geq j > i} C_{ij}^{\langle n \rangle}(z_{i} - z_{j})} \cdot \frac{1}{\prod_{j=1}^{i-1} \hat{C}_{ji}(z_j+z_i+1) \prod_{j=i+1}^{n} \hat{C}_{ij}(z_i+z_j+1)} {}^{b} X_i^{\langle n \rangle}(-z-\frac{1}{2},-\kappa).
			\]
		\end{thm}
		
		Finally, consider the spinorial $R$-operator $\check{\cR}^{\langle n \rangle}_{\alpha\beta}(t)$ of \cref{def:spinorial_r_operator} acting on the $(\alpha,\beta)$-components of $\fP_{nm}$, and by $\cR_{\alpha\beta}^{\langle n \rangle}(t) = P_{\alpha\beta} \check{\cR}^{\langle n \rangle}_{\alpha\beta}(t)$ the corresponding $R$-matrix. Recall the boundary dynamical Weyl group operator $\cA^{\langle m \rangle}_{\epsilon_{\alpha} - \epsilon_{\beta}}(t):= A^{\so_m,\epsilon_{\alpha} - \epsilon_{\beta}}_{\fP_{nm}}(t)$ of \cref{D:A-k-in-general} and the corresponding $B$-operator of \cref{def:boundary_b_operator}, which we denote by $\cB^{\langle m \rangle}_{\epsilon_{\alpha} - \epsilon_{\beta}}(t)$. 
		
		Denote by
		\[
		C_{\alpha\beta}^{\langle m \rangle}(t) := e^{\frac{i\pi}{2} (M_{22} - M_{11})} \cos \pi\frac{t - M_{11} - M_{22} - n}{2}. 
		\]

		\begin{prop}
			\label{prop:so_r_operator_equal_bweyl}
			The following relation holds:
			\[
			\cR^{\langle n \rangle}_{\alpha\beta} = C_{\alpha\beta}^{\langle m \rangle}(t)\cB^{\langle m \rangle}_{\epsilon_{\alpha} - \epsilon_{\beta}}(-t). 
			\]
			\begin{proof}
				It is enough to check for $m=2$. Recall that 
				\[
				\cB^{\langle 2 \rangle}_{12}(t) = \bigg(\sin \left(\pi \frac{t -1 - b_{12}^{\langle 2 \rangle}}{2}\right)\bigg)^{-1}\cA^{\langle 2 \rangle}_{\epsilon_1 - \epsilon_2}(t-1). 
				\]
				Observe that 
				\[
				\sin \left(\pi \frac{t -1 - b_{12}^{\langle 2 \rangle}}{2}\right) = -\cos \left(\pi \frac{t-b_{12}^{\langle 2 \rangle}}{2}\right) = -\frac{e^{\frac{i\pi t}{2}} \tilde{s}_{12}^{\so_2} + e^{\frac{-i\pi t}{2}} (\tilde{s}_{12}^{\so_2})^{-1}}{2},
				\]
				where $\tilde{s}_{12}^{\so_2}$ is the boundary Weyl group operator of \cref{def:boundary_b_operator}. Let $\deg_{\alpha}$ be the degree operator acting on the $\alpha$-component; one can check that $\tilde{s}^{\so_2}_{12} = P_{12} (-1)^{\deg_{1} \deg_{2} + \deg_{2}}$. Indeed: consider a monomial $x_{i_1,1} \ldots x_{i_k,1} x_{j_1,2 } \ldots x_{j_l,2}$. By \eqref{eq::weyl_operator_sl2} and tensor property of $\tilde{s}_{12}^{\so_2}$, we have
				\[
				\tilde{s}_{12}^{\so_2}(x_{i_1,1} \ldots x_{i_k,1} x_{j_1,2 } \ldots x_{j_l,2}) = (-1)^l x_{i_1,2} \ldots x_{i_k,2} x_{j_1,1} \ldots x_{j_l,1} = (-1)^{kl + l} x_{j_1,1} \ldots x_{j_l,1} x_{i_1,2} \ldots x_{i_k,2},
				\]
				which coincides with the action of $P_{12} (-1)^{\deg_{1} \deg_{2} + \deg_{2}} $. Observe that $\deg_{\alpha} = \sum_i \gamma^{\dagger}_{i\alpha} \gamma_{i\alpha} = M_{\alpha\alpha} + n/2$. After some algebraic manipulations using $(-1)^k = e^{i \pi k}$, we get 
				\[
				\sin \left(\pi \frac{t -1 - b_{12}^{\langle 2 \rangle}}{2}\right) = (-1)^{\deg_1 \deg_2 + 1} P_{12}  e^{\frac{i\pi}{2} (M_{22} - M_{11})} \cos \left(\pi\frac{t - M_{11} - M_{22} - n}{2}\right).
				\]
				By \cref{def:spinorial_r_operator} and the construction of the $A$-operator, we have
				\[
				\check{\cR}_{12}^{\langle n \rangle}(t) = (-1)^{\deg_1 \deg_2 + 1} \cA_{\epsilon_1 - \epsilon_2}^{\langle 2 \rangle}(-t-1). 
				\]
				Therefore,
				\begin{align*}
					\cB^{\langle 2 \rangle}_{\epsilon_1 - \epsilon_2} (-t) =& \left( \cos \left(\pi\frac{t - M_{11} - M_{22} - n}{2}\right) \right)^{-1} e^{\frac{i\pi}{2} (M_{11} - M_{22})} P_{12} (-1)^{\deg_1 \deg_2 + 1} \cA_{\epsilon_1 - \epsilon_2}^{\langle 2 \rangle}(-t-1) \\
					=& \left( \cos \left(\pi\frac{t - M_{11} - M_{22} - n}{2}\right) \right)^{-1} e^{\frac{i\pi}{2} (M_{11} - M_{22})} \cR^{\langle n \rangle}_{12}(t),
				\end{align*}
				and the proposition follows.
			\end{proof}
		\end{prop}
		
		Analogously to \cref{subsubsect:gln_quantum_duality}, define
		\begin{align*}
			Z^{\langle n \rangle}_{\alpha}(z,\kappa) := &\big(\cR^{\langle n \rangle}_{\alpha m}(z_{\alpha} - z_{m})\ldots \cR^{\langle n \rangle}_{\alpha,\alpha+1}(z_{\alpha} - z_{\alpha+1})\big)^{-1} \cdot \\
			\times &\big( \cR^{\langle n \rangle}_{1\alpha}(z_1 - z_{\alpha} - \kappa) \ldots \cR^{\langle n \rangle}_{\alpha-1,\alpha}(z_{\alpha-1}- z_{\alpha} - \kappa) \big) T_{z_{\alpha}}, \\
			X^{\langle m \rangle}_{\alpha}(\lambda,\kappa) := &\big(\cB^{\langle m \rangle}_{\epsilon_{\alpha} - \epsilon_m}(\lambda_{\alpha} - \lambda_m) \ldots \cB^{\langle m \rangle}_{\epsilon_{\alpha} - \epsilon_{\alpha+1}}(\lambda_{\alpha} - \lambda_{\alpha+1})\big)^{-1} \cdot \\
			\times &\big( \cB^{\langle m \rangle}_{\epsilon_1 - \epsilon_{\alpha}}(\lambda_1 - \lambda_{\alpha} - \kappa) \ldots \cB^{\langle m \rangle}_{\epsilon_{\alpha-1} - \epsilon_{\alpha}}(\lambda_{\alpha-1} - \lambda_{\alpha} - \kappa) \big)T_{\lambda_{\alpha}}.
		\end{align*}
		
		By \cite{CherednikQKZ}, both collections of operators are consistent. The following theorem is a boundary analog of \cref{thm::gl_quantum_bispectrality}.
		\begin{thm}
			\label{thm:orthogonal_duality_qkz_bdyn}
			The following relation holds:
			\[
			Z_{\alpha}^{\langle n \rangle}(z,\kappa) = \frac{\prod_{1\leq \beta < \alpha} C_{\beta\alpha}^{\langle m \rangle}(z_{\beta} - z_{\alpha} - \kappa)}{\prod_{m \geq \beta > \alpha} C_{\alpha\beta}^{\langle m \rangle}(z_{\alpha} - z_{\beta})} \cdot X_{\alpha}^{\langle m \rangle}(-z,-\kappa).
			\]
		\end{thm}

		\appendix
		
		\section{Technical lemmas}
		
		\begin{lm}
			\label{lm:rational_functions_shifts}
			Assume that $\xi_1(\mu),\ldots,\xi_n(\mu)$ are rational functions on $\C$ such that 
			\begin{itemize}
				\item For any $i$ there is at most one $j \neq i$ such that $\xi_i(\mu) = \xi_j(\mu)$;
				\item If $\xi_i(\mu) \neq \xi_j(\mu)$, then the number of zeros or poles of the product
				\[
				\prod_{k=0}^N \frac{\xi_i(\mu+2k)}{\xi_j(\mu+2k)}
				\]
				grows to infinity as $N \rightarrow \infty$. 
			\end{itemize}
			
			Denote by $D_i$ the operator $[D_i f](\mu) := \xi_i(\mu) f_i(\mu + 2)$. For any rational functions $f_1(\mu),\ldots,f_n(\mu)$, consider the matrix
			\begin{equation}
				\label{eq::lm_shift_vandermonde_matrix}
				A = \begin{bmatrix} f_1(\mu) & f_2(\mu) & \ldots & f_n(\mu) \\ [D_1 f_1](\mu) & [D_2 f_2](\mu) & \ldots & [D_n f_n](\mu) \\ \vdots & \vdots & \ddots & \vdots \\ [D^{n-1}_1 f_1](\mu) & [D^{n-1}_2 f_2](\mu) & \ldots & [D^{n-1}_n f_n](\mu)\end{bmatrix}.
			\end{equation}
			Then its determinant is zero if and only if:
			\begin{itemize}
				\item Either there is $i,j$ such that $\xi_i(\mu)= \xi_j(\mu)$ and $f_i(\mu) = c f_j(\mu)$ for some constant $c\in \C$;
				\item Or there is $i$ such that $f_i(\mu) = 0$.
			\end{itemize}
			\begin{proof}
				We prove the statement by induction. The case $n=1$ is obvious. For the induction step, assume that $f_1(\lambda) \neq 0$; if there is $j$ such that $\xi_j(\mu) = \xi_1(\mu)$ then assume $j=2$. Let us divide the $i$-th row of $A$ by $[D_1^{i-1} f_1](\lambda)$ for every $i$. The matrix becomes
				\[
				\begin{bmatrix} 1 & \tilde{f}_2(\lambda) & \ldots & \tilde{f}_n(\lambda) \\ 1 & [\tilde{D}_2 \tilde{f}_2](\lambda) & \ldots & [\tilde{D}_n \tilde{f}_n](\lambda) \\ \vdots & \vdots & \ddots & \vdots \\ 1 & [\tilde{D}^{n-1}_2 \tilde{f}_2](\lambda) & \ldots & [\tilde{D}^{n-1}_n \tilde{f}_n](\lambda)\end{bmatrix},
				\]
				where $[\tilde{D}_i f](\lambda):=\frac{\xi_i(\lambda)}{\xi_1(\lambda)} f(\lambda+2)$ and $\tilde{f}_i(\lambda) := \frac{f_i(\lambda)}{f_1(\lambda)}$, and its determinant is still zero. Additionally, let us subtract $(i-1)$-th row from the $i$-th one for every $i$. Then we get the matrix
				\[
				\begin{bmatrix} 1 & \tilde{f}_2(\lambda) & \ldots & \tilde{f}_n(\lambda) \\ 0 & \hat{f}_2(\lambda) & \ldots & \hat{f}_n(\lambda) \\ \vdots & \vdots & \ddots & \vdots \\ 0 & [\tilde{D}^{n-2}_2 \hat{f}_2](\lambda) & \ldots & [\tilde{D}^{n-2}_n \hat{f}_n](\lambda)\end{bmatrix},
				\]
				where 
				\begin{equation}
					\label{eq::lm_new_func_2}
					\hat{f}_i(\lambda) = \frac{\xi_i(\lambda)}{\xi_1(\lambda)} \tilde{f}_i(\lambda+2) - \tilde{f}_i(\lambda).
				\end{equation}
				Observe that the determinant of this matrix is equal to the determinant of the matrix of the form \eqref{eq::lm_shift_vandermonde_matrix}. By induction assumption, there are two cases:
				\begin{itemize}
					\item There are $i\neq j \geq 2$ such that $\xi_i(\mu) = \xi_j(\mu)$ and $\hat{f}_i(\mu) = c\hat{f}_j(\mu)$. By our assumption, $\xi_i(\mu) \neq \xi_1(\mu)$, since there can be at most two indices with the same factor. We get
					\[
					\frac{\xi_i(\mu)}{\xi_1(\mu)} \frac{f_i(\mu+2)}{f_1(\mu+2)} - \frac{f_i(\mu)}{f_1(\mu)} = \hat{f}_i(\mu) = c\hat{f}_j(\mu) = c\frac{\xi_i(\mu)}{\xi_1(\mu)} \frac{f_j(\mu+2)}{f_1(\mu+2)} - c\frac{f_j(\mu)}{f_1(\mu)},
					\]
					which is equivalent to 
					\[
					\frac{g(\mu+2)}{g(\mu)} = \frac{\xi_1(\mu)}{\xi_i(\mu)}, \qquad g(\mu) := \frac{f_i(\mu) - c f_j(\mu)}{f_1(\mu)}.
					\]
					For every $N$, we get
					\[
					\frac{g(\mu+2N)}{g(\mu)} = \prod_{i=0}^{N-1} \frac{\xi_1(\mu+2i)}{\xi_i(\mu+2i)}.
					\]
					By our assumption, the number of zeros and poles of the right-hand side grows to infinity as $N \rightarrow \infty$, however, the number of zeroes and poles on the left-hand side is bounded. Therefore, $g(\mu) = 0$, and so $f_i(\mu) = cf_j(\mu)$.
					
					\item There is $i$ such that $\hat{f}_i(\mu) = 0$. If $i\neq 2$, then 
					\[
					\frac{g(\mu+2)}{g(\mu)} = \frac{\xi_1(\mu)}{\xi_i(\mu)}, \qquad g(\mu) = f_i(\mu)/f_1(\mu),
					\]
					and we proceed as before. If $i=2$, then $\xi_2(\mu) = \xi_1(\mu)$, and so 
					\[
					\frac{f_2(\mu+2)}{f_1(\mu+2)} = \frac{f_2(\mu)}{f_1(\mu)};
					\]
					since these are rational functions, we conclude that $f_2(\mu) = c f_1(\mu)$ for some constant $c$. 
				\end{itemize}
			\end{proof}
		\end{lm}
		
		Let $A(x,\mu)$ be as in \eqref{E:AxL-notation}:
		\[
		A(x,\mu) = \frac{(-2i)^{\mu+1}}{\mu+1} \B \left( \frac{ix - \mu}{2}, \mu + 1 \right)^{-1}.
		\]
		Recall that it satisfies \eqref{eq::boundary_lambda_shift} and \eqref{eq::boundary_dyn_weyl_minus_weight}:
		\[
		A(x,\mu+2) = \frac{x^2+\mu^2}{\mu(\mu+1)} A(x,\mu), \qquad A(-x,\mu) = -\frac{\sin \left(\pi \frac{ix+\mu}{2} \right)}{\sin \left(\pi \frac{ix-\mu}{2} \right)} A(x,\mu).
		\]
		
		\begin{prop}
			\label{prop:beta_lin_independence_rat_func}
			Let $\{x_1,\ldots,x_n\}$ be a collection of numbers such that $x_i$ are pairwise non-congruent modulo $2\sqrt{-1}$. Assume there are rational functions $f_1(\mu),\ldots,f_n(\mu)$ such that
			\[
			f_1(\mu) A(x_1,\mu) + \ldots + f_n(\mu) A(x_n,\mu) = 0
			\]
			for all integral and sufficiently positive $\mu$. Then $f_i(\mu) = 0$ for all $\mu$ and for every $i$.
			\begin{proof}
				We proceed by induction on $n$. For $n=1$, observe that $A(x,\mu)$ is not zero for infinite number of positive integral values of $\mu$. So, if we have $f(\mu) A(x,\mu) = 0$, then the rational function $f(\mu)$ has infinite number of zeros, therefore, is zero itself.
				
				One can easily see that for every sufficiently positive integral $\mu$, the functions $\{A(x_i,\mu)\}$ are not simultaneously zero. Denote by $\xi_i(\mu):= \frac{x_i^2 + \mu^2}{\mu(\mu+1)}$. Then we have
				\[
				\sum_{i=1}^n \left( \prod_{j=0}^{k-1} \xi(\mu + 2j) \right) f_i(\mu + 2k) A(x_i,\mu) = 0
				\]
				for every $k$ and sufficiently positive $\mu$. In particular, the system given by the matrix \eqref{eq::lm_shift_vandermonde_matrix} with the same meaning of notations has a nonzero solution, hence its determinant zero for infinitely many $\mu$. Since the determinant is a rational function in $\mu$, it is zero for all $\mu$. Therefore, by \cref{lm:rational_functions_shifts}, we get that either one of the functions $f_i(\mu)$ is zero or there are $i,j$ such that $\xi_i(\mu) = \xi_j(\mu)$ (equivalently, $x_i = -x_j$) and $f_i(\mu) = c_if_j(\mu)$. In the former case we can apply the induction assumption. As for the latter, denote by 
				\[
				\tilde{A}_i(x_i,\mu) = c_i A(x_i,\mu) + A(-x_i,\mu).
				\]
				Observe that 
				\[
				\tilde{A}_i(x_i,\mu+2) = \frac{x_i^2+\mu^2}{\mu(\mu+1)}\tilde{A}_i(x_i,\mu).
				\]
				We have a relation
				\begin{equation}
					\label{eq_lm::modified_relation}
					f_i(\mu)\tilde{A}_i(x_i,\mu) + \sum\limits_{k\neq i,j} f_k(\mu) A(x_k,\mu) = 0
				\end{equation}
				for all sufficiently positive integral $\mu$; now there are $n-1$ rational functions. Using property \eqref{eq::boundary_dyn_weyl_minus_weight}, we get for positive integral $\mu$
				\[
				\tilde{A}_i(x_i,\mu) = 0 \Leftrightarrow c_i = \frac{\sin \left(\pi \frac{ix-\mu}{2} \right)}{\sin \left(\pi \frac{ix+\mu}{2} \right)} = (-1)^{\mu}.
				\]
				In particular, we see that it is non-zero for an infinite number of sufficiently positive integral $\mu$. Applying the same arguments as before, we conclude that either one of the functions $f_k(\mu)$ for $k\neq j$ is zero or there exist $k,l$ with $x_k = - x_l$ such that $f_k(\mu) = c_k f_l(\mu)$, moreover, $x_k \neq \pm x_i$. Repeating this process, we can conclude that all functions $f_i(\mu)$ are zero.
			\end{proof}
		\end{prop}
		
        \section{Calculations in rank two}
		\label{appendix:rank_two}
		
		In this appendix, we show the braid relation \cref{thm:boundary_braid_relations}.
		
		Let $\g$ be a simple Lie algebra of rank $2$, i.e., $\mathfrak{sl}_3$, $\mathfrak{sp}_4$, or $\mathfrak{g}_2$. Let $\alpha$ (resp. $\beta$) be the short root (resp. the long root), and $\omega_1$ (resp. $\omega_2$) be the dual fundamental weight. Introduce coordinates on $\h^*$ by $\lambda = \lambda_1 \omega_1 + \lambda_2 \omega_2$.
		
		Denote by $\rho = \omega_1 + \omega_2$ the Weyl vector. For $\gamma\in \{\alpha, \beta\}$, write
		\begin{equation}
			\label{eq::atilde_operator}
			\tilde{A}^{\k,\gamma} (\lambda) := A^{\k, \gamma}(\lambda - \rho),
		\end{equation}
		where $A^{\k, \gamma}(\lambda)$ is as defined in Definition \ref{D:A-k-in-general}. We have the following simple fact.
		\begin{prop}
			\label{prop:unshifted_braid_relation}
			The \emph{unshifted} braid relation
			\[
			\ldots \tilde{A}^{\k}_{\alpha}(s_{\beta}s_{\alpha}\lambda)\tilde{A}^{\k}_{\beta}(s_{\alpha} \lambda) \tilde{A}^{\k}_{\alpha}(\lambda) = \ldots \tilde{A_{\beta}}^{\k}(s_{\alpha}s_{\beta}\lambda)\tilde{A}^{\k}_{\alpha} (s_{\beta} \lambda) \tilde{A}^{\k}_{\beta} (\lambda)
			\]
			is satisfied if and only if the operators $\{ A_{\alpha}^{\k}(\lambda) \}$ satisfy the shifted one of \cref{thm:boundary_braid_relations}.
		\end{prop}
		
		It is evident from \eqref{E:A-in-general-defn} that the boundary dynamical Weyl operators depend only on one variable, so we will use the following notation. 
		
		\begin{notation}
			For $\mu\in \C$, we write
			\begin{equation}
				\label{eq::r_operator}
				R_{\alpha}(\mu) := \tilde{A}^{\k,\alpha}(\mu \omega_1) \qquad \text{and} \qquad  R_{\beta} (\mu) := \tilde{A}^{\k,\beta} (\mu \omega_2).
			\end{equation}
		\end{notation}
		
		To prove \cref{prop:boundary_nondyn_braid_rel} and \cref{thm:boundary_braid_relations}, we perform explicit calculations to verify that the operators $R_{\alpha}$ and $R_{\beta}$ satisfy the braid relations in specific integrable $\k$-modules $\{X_i\}$, such that any other irreducible integrable module is a direct $\k$-summand of $X_i \otimes V$ for suitable $i$ and a finite-dimensional $\g$-representation $V$. These calculations, along with arguments in the proof of \cref{prop:boundary_nondyn_braid_rel} and \cref{thm:boundary_braid_relations}, imply the braid relations in general integrable $\k$-modules.
		
		We will also prove \cref{lm:sine_braid_relation} case-by-case. Denote by 
		\begin{equation}
			\label{eq::appendix_shifted_sine_operators}
			S_{\alpha}(\mu):= \sin \left(\pi \frac{\mu - 1 - ib_{\alpha}}{2}\right), \qquad S_{\beta}(\mu):= \sin \left(\pi \frac{\mu - 1 - ib_{\beta}}{2}\right). 
		\end{equation}
		As before, it is enough to prove the unshifted braid relation for these operators. 
		
		\subsection{Type \texorpdfstring{$A_2$}{A2}} 
		\label{subsect:type_A}
		
		The corresponding Lie algebra is $\sl_3$. We denote by $\{E_{ij}\}$ the basis of matrix units. In particular,
		\begin{align*}
			e_{\alpha} &= E_{12}, & \qquad e_{\beta} &= E_{23}, &\qquad e_{\alpha+\beta} &= E_{13}, \\ 
			f_{\alpha} &= E_{21}, & \qquad f_{\beta} &= E_{32}, &\qquad f_{\alpha+\beta} &= E_{31},
		\end{align*}
		where $\alpha,\beta$ are simple roots for $\sl_3$. 
		
		The Lie subalgebra $\k = \g^{\theta}$ is $\so_3$. Let us choose the basis
		\begin{equation}
			\label{eq::sl3_b_to_j}
			J_0 = -b_{\beta}, \qquad J_1 = -b_{\alpha+\beta}, \qquad J_2 = - b_{\alpha},
		\end{equation}
		where we consider all indices modulo 3. They satisfy the relations
		\begin{equation}
			\label{eq::so3_relations}
			[J_i,J_{i+1}] = J_{i+2},\ i\in \mathbb{Z}/3.
		\end{equation}
		
		There is an isomorphism $\so_3 \cong \sl_2$. Namely, if $\{e,h,f\}$ are the standard $\sl_2$-generators, then
		\begin{equation}
			\label{eq::so3_cong_sl2}
			J_0 = \frac{h}{2i}, \qquad J_1 = \frac{e-f}{2}, \qquad J_2 = \frac{e+f}{2i}.
		\end{equation}
		One can easily see that the restriction of the vector representation $\C^3$ of $\g =\sl_3$ to $\k=\so_3$ is the irreducible representation $V(2)$ of $\sl_2\cong \so_3$. Therefore, the only irreducible integrable representation of $\k$ that we need to check is $V(1)$. 
		
		\subsubsection{Braid relation in special cases}
		Denote by $\{v_+,v_-\}$ the standard weight basis of $V(1)$:
		\[
		h\cdot v_+ = v_+, \qquad h\cdot v_- = -v_-.
		\]
		It follows from \eqref{eq::sl3_b_to_j} and \eqref{eq::so3_cong_sl2} that
		\[
		b_{\beta} = \begin{bmatrix} i/2 & 0 \\ 0 & -i/2 \end{bmatrix}, \qquad b_{\alpha} = \begin{bmatrix}  0 & i/2 \\ i/2 & 0 \end{bmatrix}.
		\]
		
		\begin{proof}[Proof of \cref{prop:boundary_nondyn_braid_rel}]
			The boundary braid group operators are 
			\[
			\tilde{s}_{\beta}^{\k} = \begin{bmatrix}
				e^{\frac{i\pi}{4}} & 0 \\ 0 & e^{-\frac{i\pi}{4}}
			\end{bmatrix}, \qquad \tilde{s}^{\k}_{\alpha} = S\cdot \begin{bmatrix}
				e^{\frac{i\pi}{4}} & 0 \\ 0 & e^{-\frac{i\pi}{4}}
			\end{bmatrix}\cdot S^{-1},
			\]
			where 
			\[
			S = \begin{bmatrix} 1 & -1 \\ 1 & 1 \end{bmatrix}
			\]
			is a change of basis matrix such that $S^{-1}\cdot b_{\alpha}\cdot S$ is equal to the diagonal matrix of $b_{\beta}$. The braid relation reads $\tilde{s}^{\k}_{\alpha} \tilde{s}^{\k}_{\beta} \tilde{s}^{\k}_{\alpha} = \tilde{s}^{\k}_{\beta} \tilde{s}^{\k}_{\alpha} \tilde{s}^{\k}_{\beta}$, which can be verified directly in matrix terms. 
		\end{proof}
		
		\begin{proof}[Proof of \cref{thm:boundary_braid_relations}]
			In the dynamical case, we have
			\begin{align*}
				R_{\beta}(\mu) = &\begin{bmatrix} A(i/2,\mu-1) & 0 \\ 0 & A(-i/2,\mu-1) \end{bmatrix}, \\
				R_{\alpha} (\lambda) = S\cdot &\begin{bmatrix} A(i/2,\mu-1) & 0 \\ 0 & A(-i/2,\mu-1) \end{bmatrix} \cdot S^{-1},
			\end{align*}
			in the notations of \eqref{eq::r_operator}. The braid relation reads
			\begin{equation}
				\label{eq::appendix_unshifted_braid_relation_typeA}
				R_{\alpha}(\lambda_2) R_{\beta}(\lambda_1 + \lambda_2) R_{\alpha}(\lambda_1) = R_{\beta}(\lambda_1) R_{\alpha}(\lambda_1 + \lambda_2) R_{\beta}(\lambda_2), \qquad \text{for $\lambda_1, \lambda_2\in \C$}.
			\end{equation}
			By \eqref{eq::boundary_dyn_weyl_minus_weight}, we have
			\[
			A(-i/2,\mu-1) = \frac{\cos \left(\pi \left( -\frac{1}{4} + \frac{\mu}{2} \right)\right)}{\cos \left(\pi \left( -\frac{1}{4} - \frac{\mu}{2} \right)\right)} A(i/2,\mu-1) = \frac{e^{i\pi \mu} + i}{ie^{i\pi \mu} + 1}A(i/2,\mu-1)
			\]
			If we multiply the braid relation by
			\[
			\frac{(ie^{i\pi \lambda_2} + 1)(i e^{i\pi (\lambda_1 + \lambda_2}) + 1) (ie^{i\pi \lambda_2} + 1)}{A(i/2,\lambda_2-1)A(i/2,\lambda_1+\lambda_2-1)A(i/2,\lambda_1-1)},
			\]
			it becomes
			\[
			S\cdot \begin{bmatrix} iz_2 + 1 & 0 \\ 0 & z_2 + i \end{bmatrix} \cdot S^{-1} \cdot \begin{bmatrix} iz_1 z_2 + 1 & 0 \\ 0 & z_1 z_2 + i \end{bmatrix} \cdot S \cdot \begin{bmatrix} iz_1 + 1 & 0 \\ 0 & z_1 + i \end{bmatrix}
			\]
			\[
			=
			\]
			\[
			\begin{bmatrix} iz_1 + 1 & 0 \\ 0 & z_1 + i \end{bmatrix}\cdot S \cdot \begin{bmatrix} iz_1 z_2 + 1 & 0 \\ 0 & z_1 z_2 + i \end{bmatrix} \cdot S^{-1} \cdot \begin{bmatrix} iz_2 + 1 & 0 \\ 0 & z_2 + i \end{bmatrix}\cdot S,
			\]
			where $z_1 = e^{i\pi \lambda_1}$ and $z_2 = e^{i\pi \lambda_2}$. This renormalized braid relation is a product of matrices over polynomials in $z_1,z_2$ and can be verified directly on a computer.
		\end{proof}
		
		\subsubsection{Proof of \cref{lm:sine_braid_relation}}
		Consider the variables $z_j := e^{\frac{i\pi \lambda_j}{2}}$ for $\lambda = \lambda_1 \omega_1 + \lambda_2 \omega_2$. In the notations of \eqref{eq::appendix_shifted_sine_operators}, we need to prove the relation
		\[
		S_{\alpha}(\lambda_1) S_{\beta}(\lambda_1 + \lambda_2) S_{\alpha} (\lambda_2)  = S_{\beta} (\lambda_2) S_{\alpha}(\lambda_1 + \lambda_2) S_{\beta}(\lambda_1).
		\]
		
		By \eqref{eq::sin_boundary_weyl_group}, we have
		\[
		S_{\alpha}(\mu) = -\frac{e^{\frac{i\pi \mu}{2}} \tilde{s}^{\k}_{\alpha} + e^{-\frac{i\pi\mu}{2}} (\tilde{s}^{\k}_{\alpha})^{-1}}{2}, \qquad S_{\beta}(\mu) = -\frac{e^{\frac{i\pi \mu}{2}} \tilde{s}^{\k}_{\beta} + e^{-\frac{i\pi\mu}{2}} (\tilde{s}^{\k}_{\beta})^{-1}}{2}
		\]
		Substituting it in the unshifted braid relation \eqref{eq::appendix_unshifted_braid_relation_typeA} for $S$-operators and using the braid relation for the non-dynamical braid group operators, the lemma would follow from the relation
		\[
		\tilde{s}_{\alpha}^{\k} (\tilde{s}_{\beta}^{\k})^{-1} \tilde{s}_{\alpha}^{\k} = \tilde{s}_{\beta}^{\k} (\tilde{s}_{\alpha}^{\k})^{-1} \tilde{s}_{\beta}^{\k}. 
		\]
		Recall the tensor property $\Delta \tilde{s}^{\k}_{\alpha} = \tilde{s}^{\k}_{\alpha} \otimes \tilde{s}^{\g}_{\alpha}$. By the same arguments as in the proof of \cref{prop:boundary_nondyn_braid_rel} or \cref{thm:boundary_braid_relations}, it is enough to show this relation for $X = V(1)$ and $V$ an arbitrary $\sl_3$-representation $V$. For the latter, it is enough to show it for the vector representation $\C^3$, as any $\sl_3$-representation is a subrepresentation of a suitable tensor product of $\C^3$. Both cases can be verified directly. 
		
		\subsection{Type \texorpdfstring{$C_2$}{C2}} We work with the choice of simple roots $\alpha = \epsilon_1- \epsilon_2$ and $\beta= 2\epsilon_2$. The corresponding Lie algebra is $\mathfrak{g}=\sp_4$. We will work with an explicit matrix realization of $\mathfrak{sp}_4$ as follows. Let $\C^4 = \mathrm{span}(v_1,v_2,v_{-1},v_{-2})$ with the symplectic form 
		\[
		\Omega = \begin{bmatrix} 0 & 0 & 1 & 0 \\ 0 & 0 & 0 & 1 \\ -1 & 0 & 0 & 0 \\ 0 & -1 & 0 & 0  \end{bmatrix}.
		\]
		We can choose a basis of $\g$ satisfying the assumptions of \cref{sect:dynamical_twist} as follows:
		\begin{align*}
			e_{\alpha} = E_{12} - E_{-2,-1}, \qquad e_{\beta} = E_{2,-2}, \qquad e_{\alpha+\beta} = -E_{1,-2} -E_{2,-1}, \qquad e_{2\alpha + \beta} = E_{1,-1}, \\ 
			f_{\alpha} = E_{21} - E_{-1,-2}, \qquad f_{\beta} = E_{-2,2}, \qquad f_{\alpha + \beta} = -E_{-2,1} -E_{-1,2}, \qquad f_{2\alpha + \beta} = E_{-1,1},
		\end{align*}
		and
		\begin{align*}
			h_{\alpha} = E_{11} - E_{22} - E_{-1,-1} + E_{-2,-2}, \qquad h_{\beta} = E_{22} - E_{-2,-2}.
		\end{align*}
		
		The fixed point subalgebra $\k= \mathfrak{sp}_4^{\theta}$ is isomorphic to $\gl_2$. Indeed, consider the elements
		\begin{equation}
			\label{eq::c2_b_as_j}
			J_0 = -\frac{b_{\alpha}}{2}, \qquad J_1 = \frac{-b_{2\alpha + \beta} + b_{\beta}}{2}, \qquad J_2 = -\frac{b_{\alpha + \beta}}{2}, \qquad Z = i\cdot (b_{2\alpha + \beta} + b_{\beta}).
		\end{equation}
		One can check that $Z$ is central and $\{J_i \}$ satisfy the standard $\so_3$-relations \eqref{eq::so3_relations}. Since $\k$ is the Lie-subalgebra of $\mathfrak{g}$ generated by $b_{\alpha}$ and $b_{\beta}$, and 
		\[
		b_{\alpha} = -2J_0 \qquad \text{and} \qquad b_{\beta} = 2J_1 - iZ,
		\]
		it follows that $\{J_0, J_1, J_2, Z\}$ is a basis of $\k$. 
		
		Consider the following basis of $\C^4$: 
		\begin{align*}
			w_1 &= v_1 - iv_2 + iv_{-1} + v_{-2}, \qquad
			w_2 &= iv_1 - v_2 - v_{-1} - iv_{-2}, \\
			w_1^* &= v_1 - iv_2 -iv_{-1} - v_{-2}, \qquad
			w_2^* &= -iv_1 + v_2 - v_{-1} - iv_{-2}.
		\end{align*}
		One can check that $\C^4 \cong \C^2 \oplus (\C^2)^*$ as $\k$-representations, where $\C^2 = \mathrm{span}(w_1,w_2)$ is the vector representation and $(\C^2)^* = \mathrm{span}(w_1^*,w_2^*)$ is its dual. In particular, the only (integrable) representation of $\k \cong \gl_2$ that we need to consider is the one-dimensional representation $\Det^k$ spanned by a vector $\omega^k$ such that
		\[
		J_i \cdot \omega^k = 0,\ i \in \{0,1,2\},  \qquad \text{and} \qquad  Z\cdot \omega^k = k\omega^k,\ k\in \C.
		\]
		It follows that 
		\[
		b_{\alpha} \cdot \omega_k = 0 \qquad \text{and} \qquad b_{\beta}\cdot \omega_k = -ik.
		\]
		\subsubsection{Braid relation in special cases}
		In basis of fundamental weights $\omega_1:=\epsilon_1$ and $\omega_2:=\epsilon_1 + \epsilon_2$, the reflection operators have the form
		\[
		s_{\alpha} = \begin{bmatrix} -1 & 0 \\ 1 & 1 \end{bmatrix}, \qquad s_{\beta} = \begin{bmatrix} 1 & 2 \\ 0 & -1 \end{bmatrix}.
		\]
		
		\begin{proof}[Proof of \cref{thm:boundary_braid_relations}]
			The braid relation has the form 
			\begin{equation}\label{E:C2-braidreln}
				R_{\beta}(\lambda_2) R_{\alpha}(\lambda_1 + 2\lambda_2) R_{\beta} (\lambda_1 + \lambda_2) R_{\alpha}(\lambda_1)
				= R_{\alpha}(\lambda_1) R_{\beta} (\lambda_1 + \lambda_2) R_{\alpha}(\lambda_1 + 2\lambda_2) R_{\beta}(\lambda_2).
			\end{equation}
			In the one-dimensional representation $\C_k$, the operators $R_{\alpha}(\lambda)$ and $R_{\beta}(\mu)$ commute, so the braid relation in equation \eqref{E:C2-braidreln} is obviously satisfied. 
		\end{proof}
		
		The proof of \cref{prop:boundary_nondyn_braid_rel} is similar to the dynamical case. 
		
		\subsubsection{Proof of \cref{lm:sine_braid_relation}}
		Observe that 
		\[
		(\tilde{s}_{\alpha}^{\k})^2 = e^{-2\pi J_0}, \qquad (\tilde{s}_{\beta}^{\k})^2 = e^{2\pi J_1} \cdot e^{-i \pi Z}. 
		\]
		For every irreducible $\k$-module $V(m) \otimes \det^{n}$, the exponential $e^{-2\pi J_0}$ acts as $(-1)^l$, and $e^{-i\pi Z}$ by a scalar $e^{-i \pi n}$. Therefore,
		\begin{align*}
			(\tilde{s}^{\k}_{\alpha})^{-1} &= (-1)^l \tilde{s}^{\k}_{\alpha}, &\qquad S_{\alpha}(\mu) &= (e^{\frac{i\pi \mu}{2}} + (-1)^l e^{-\frac{i\pi \mu}{2}}) \tilde{s}^{\k}_{\alpha}, \\ 
			(\tilde{s}^{\k}_{\beta})^{-1} &= e^{i\pi n} \tilde{s}^{\k}_{\beta}, &\qquad \text{and} \qquad S_{\beta}(\mu) &= (e^{\frac{i\pi \mu}{2}} + e^{i\pi n} e^{-\frac{i\pi \mu}{2}}) \tilde{s}^{\k}_{\beta}.
		\end{align*}
		Therefore, the dynamical braid relation \eqref{E:C2-braidreln} for $S$-operators follows from the non-dynamical one of \cref{prop:boundary_nondyn_braid_rel}. 
		
		\subsection{Type \texorpdfstring{$G_2$}{G2}} The corresponding Lie algebra is $\g_2$. We will use notations and results of \cite{PollackG2}, see also \cite{SpringerVeldkamp}.
		
		Let $\C^3$ be the vector representation of $\sl_3$. Denote by $\Theta$ the space of (split) octonions. As a vector space, 
		\[
		\Theta = \left\{ \begin{bmatrix} a & v \\ \phi & d \end{bmatrix} \Big| a,d \in \C, v\in \C^3, \phi \in (\C^3)^*  \right\}.
		\]
		Since $\C^3$ is an $\sl_3$-representation, we have canonical isomorphisms $\Lambda^2 \C^3 \cong (\C^3)^*$ and $\Lambda^2 (\C^3)^* \cong \C^3$. Then the multiplication on $\Theta$ is given by
		\[
		\begin{bmatrix} a & v \\ \phi & d \end{bmatrix} \begin{bmatrix} a' & v' \\ \phi' & d' \end{bmatrix} = \begin{bmatrix} a a' + \phi'(v) & av' + d'v - \phi \wedge \phi' \\ a' \phi + d \phi' + v \wedge v' & \phi(v') + dd'\end{bmatrix}.
		\]
		Octonions are equipped with a quadratic norm $N(x)$ defined by 
		\[
		N \begin{bmatrix} a & v \\ \phi & d \end{bmatrix} = ad - \phi(v).
		\]
		In particular, it induces a symmetric bilinear form 
		\[
		(x,x') = ad' + a'd - \phi(v') - \phi'(v), \qquad x = \begin{bmatrix} a & v \\ \phi & d \end{bmatrix}, \qquad x' = \begin{bmatrix} a' & v' \\ \phi' & d' \end{bmatrix}. 
		\]
		Octonions are equipped with conjugation $*$ which is an order-reversing involution on $\Theta$:
		\[
		x = \begin{bmatrix} a & v \\ \phi & d \end{bmatrix}  \mapsto x^* := \begin{bmatrix} d & -v \\ -\phi & a \end{bmatrix} 
		\]
		Define the trace
		\[
		\mathrm{tr} \begin{bmatrix} a & v \\ \phi & d \end{bmatrix}  = a + d,
		\]
		and $\C^7:= \{ x\in \Theta| \mathrm{tr}(x) = 0\}$. It is orthogonal to $1\in \Theta$, in particular, the restriction of the bilinear form to $\C^7$ is non-degenerate. 
		
		Denote by $\Im(x) := (x - x^*)/2$ for $x\in \Theta$. Consider the map
		\[
		\C^7 \otimes \C^7 \rightarrow \C^7, \qquad x \otimes y \mapsto \Im(xy).
		\]
		This map is alternating, thus induces a map $H\colon \so(7) \rightarrow \C^7$, where we identified $\so(7) \cong \Lambda^2 \C^7$ by
		\[
		a \wedge b \mapsto (x \mapsto (b,x) a - (a,x) b).
		\]
		Then $\g_2 := \ker(H) \subset \so(7)$.
		
		\begin{remark}
			We will not use it, but the group $G_2$ inside $\SO(7)$ can be identified as the subgroup of transformations preserving the octonion product.
		\end{remark}
		
		Let $\{w_0,w_1,w_2\} \subset \C^3$ be the standard basis and $\{w_{0}^*,w_{1}^*,w_{2}^*\}\subset (\C^3)^*$ be the dual one. In what follows, we consider all indices modulo 3. We have $(w_i,w_j^*) = -\delta_{ij}$ and all other pairings zero. Also, denote $u_0 = \begin{bmatrix} 1 & 0 \\ 0 & -1 \end{bmatrix} \in \C^7$. Define
		\[
		M_{ij} = w_{j}^* \wedge w_i, \qquad
		v_j = u_0 \wedge w_j + w_{j+1}^* \wedge w_{j+2}^*, \qquad
		\delta_j = u_0 \wedge w_j^* + w_{j+1} \wedge w_{j+2}.
		\]
		One can check that $v_j$ and $\delta_j$ are in $\g_2$ for any $j$. Likewise, if $i \neq j$, then $M_{ij}$ is in $\g_2$ as well. A sum $\alpha_1 M_{11} + \alpha_2 M_{22} + \alpha_3 M_{33}$ is in $\g_2$ if $\alpha_1 + \alpha_2 + \alpha_3 = 0$. 
		
		We will choose a positive root system on $\g_2$ by letting
		\begin{align*}
			e_{\alpha} &= v_1, \qquad &f_{\alpha} &= -\delta_1, \qquad &h_{\alpha^{\vee}} &= 2 M_{11} - M_{00} - M_{22}, \\
			e_{\beta} &= M_{01}, \qquad &f_{\beta} &= M_{10}, \qquad &h_{\beta^{\vee}} &= M_{00} - M_{11},
		\end{align*}
		where $\alpha$ is the short simple root and $\beta$ is the long one. 
		
		Denote by $\iota\colon \Theta \rightarrow \Theta$ the involution 
		\[
		\begin{bmatrix} a & v \\ \phi & d \end{bmatrix} \mapsto \begin{bmatrix} d & -\tilde{\phi} \\ -\tilde{v} & a \end{bmatrix},
		\]
		where $\widetilde{w}_i = w_i^*, \widetilde{w}^*_i = w_i$, and $\tilde{\cdot}$ is extended by linearity. It induced an involution $\theta$ on $\Lambda^2 \C^7 \cong \so(7)$ by $\theta(v \wedge w) = \iota(v) \wedge \iota(w)$. It is clear that $\theta$ preserves $\g_2 \subset \so(7)$. 
		
		\begin{prop}\cite[Claim 2.7]{PollackG2}
			The isomorphism $\theta$ is a Cartan involution on $\g_2$.
		\end{prop}
		
		We have 
		\[
		\theta(M_{ij}) = -M_{ji}, \qquad \theta(v_j) = \delta_j.
		\]
		
		Consider the elements
		\begin{align*}
			J_i = \frac{M_{i+2,i+1} - M_{i+1,i+2} + v_i + \delta_i}{4}, \\
			J'_i = \frac{3M_{i+2,i+1} - 3M_{i+1,i+2} - v_i - \delta_i}{4}.
		\end{align*}
		One can check that $\{J_i\}$ and $\{J'_i\}$ satisfy the $\so_3$-relations \eqref{eq::so3_relations} and $[J_i, J'_k]=0$ for every $i,k$. In particular, we obtain that $\k:=\g_2^{\theta} = \so_3 \oplus \so_3$. In these generators, we have
		\begin{equation}
			\label{eq::g2_b_gens}
			b_{\alpha} = J_0' - 3J_0, \qquad b_{\beta} = J_2' + J_2.
		\end{equation}
		
		The next goal is to describe the vector representation $\C^7$ in terms of $\k$.
		\begin{prop}
			Let
			\begin{gather*}
				Q_{-2,0}  = -i w_1 -w_2 + i w_1^* + w_2^*,\quad
				Q_{0,0} = -w_0 + w_0^*, \quad
				Q_{2,0} = iw_1 - w_2 - i w_1^* + w_2^*, \\
				Q_{-1,-1} = iw_1 + w_2 + i w_1^* + w_2^*, \quad
				Q_{-1,1} = iw_0 + i w_0^* + u_0, \\
				Q_{1,-1} = -iw_0 -i w_0^* + u_0, \quad\text{and} \quad
				Q_{1,1} = -iw_1 + w_2 - iw_1^* + w_2^*.
			\end{gather*}
			Then there are isomorphisms
			\[
			\mathrm{span}(Q_{-2,0},Q_{0,0},Q_{2,0}) \cong V(2) \boxtimes V(0), \qquad \mathrm{span}(Q_{-1,-1},Q_{-1,1},Q_{1,-1},Q_{1,1}) \cong V(1) \boxtimes V(1)
			\]
			of $\so_3 \oplus \so_3$-representations with $Q_{i,j}$ a basis vector of weight $(i,j)$ with respect to $J_0,J'_0$. 
			\begin{proof}
				Direct calculations.
			\end{proof}
		\end{prop}
		
		Therefore, we only need to check the braid relations in the representations $V(1) \boxtimes \C$ and $\C \boxtimes V(1)$. 
		
		\subsubsection{Braid relation in special cases} 
		
		In basis of fundamental weights, the reflection operators have the form
		\[
		s_{\alpha} = \begin{bmatrix} -1 & 0 \\ 1 & 1 \end{bmatrix}, \qquad s_{\beta} = \begin{bmatrix} 1 & 3 \\ 0 & -1  \end{bmatrix},
		\]
		Therefore, the braid relation is 
		\begin{align*}
			R_{\beta}(\lambda_2)R_{\alpha}(\lambda_1 + 3\lambda_2) R_{\beta}(\lambda_1 + 2\lambda_2)R_{\alpha}(2\lambda_1 +3\lambda_2)R_{\beta} (\lambda_1 + \lambda_2)R_{\alpha} (\lambda_1) = \\
			= R_{\alpha}(\lambda_1) R_{\beta}(\lambda_1 + \lambda_2)R_{\alpha} (2\lambda_1 + 3\lambda_2 )R_{\beta} (\lambda_1 + 2\lambda_2) R_{\alpha}(\lambda_1 + 3\lambda_2) R_{\beta} (\lambda_2)
		\end{align*}
		
		We verify it only for $\C \boxtimes V(1)$, the other case is similar. By \eqref{eq::g2_b_gens}, we have $b_{\alpha} = J_0'$ and $b_{\beta} = J_2'$. According to \eqref{eq::so3_cong_sl2}, let us choose a basis in $V(1)$ such that
		\[
		J_0' = \begin{bmatrix}
			-i/2 & 0 \\ 0 & i/2
		\end{bmatrix}, \qquad J_2' = \begin{bmatrix}
			0 & -i/2 \\ -i/2 & 0 
		\end{bmatrix}. 
		\]
		Then we proceed as in \cref{subsect:type_A} for both dynamical braid relation of \cref{thm:boundary_braid_relations} and non-dynamical one \cref{prop:boundary_nondyn_braid_rel}.
		
		\subsubsection{Proof of \cref{lm:sine_braid_relation}}
		Observe that the operator $e^{-\frac{4 \pi J_0}{2}}$ acts by $(-1)^l$ on the irreducible representation $V(k) \boxtimes V(l)$ of $\k = \so_3 \oplus \so_3$. Therefore, in such a representation, we have $e^{\frac{-3\pi J_0}{2}} = (-1)^l e^{\frac{\pi J_0}{2}}$, so that 
		\[
		\exp \left( \frac{\pi b_{\alpha}}{2} \right)= (-1)^l \exp \left( \frac{\pi J_0'}{2} + \frac{\pi J_0}{2} \right), \qquad \exp \left( \frac{\pi b_{\beta}}{2} \right) = \exp \left( \frac{\pi J_2'}{2} + \frac{\pi J_2}{2} \right).
		\]
		Denote by $b_{\alpha}^{\so_3}, b_{\beta}^{\so_3}$ the type $A_2$ generators of \cref{subsect:type_A} and by $\Delta$ the coproduct on $\so_3$. Then we have 
		\[
		S_{\alpha}(\mu) = (-1)^l \sin \left(\pi \frac{\mu - 1 - \Delta b_{\alpha}^{\so_3}}{2}\right), \qquad S_{\beta}(\mu) = \sin \left(\pi \frac{\mu - 1 - \Delta b_{\beta}^{\so_3}}{2}\right).
		\]
		Therefore, it is enough to prove the dynamical braid relation for $\so_3$ acting via the coproduct, and it follows by repeatedly applying the type $A_2$ dynamical braid relation \eqref{eq::appendix_unshifted_braid_relation_typeA}:
		\begin{align*}
			&\underbrace{S_{\beta}(\lambda_2)S_{\alpha}(\lambda_1 + 3\lambda_2) S_{\beta}(\lambda_1 + 2\lambda_2)}S_{\alpha}(2\lambda_1 +3\lambda_2)S_{\beta} (\lambda_1 + \lambda_2)S_{\alpha} (\lambda_1) = \\
			=&S_{\alpha}(\lambda_1 + 2\lambda_2) S_{\beta}(\lambda_1 + 3\lambda_2) S_{\alpha }(\lambda_2) S_{\alpha}(2\lambda_1 +3\lambda_2)S_{\beta} (\lambda_1 + \lambda_2)S_{\alpha} (\lambda_1) = \\
			=&S_{\alpha}(\lambda_1 + 2\lambda_2) S_{\beta}(\lambda_1 + 3\lambda_2) S_{\alpha}(2\lambda_1 +3\lambda_2) \underbrace{S_{\alpha}(\lambda_2)  S_{\beta} (\lambda_1 + \lambda_2)S_{\alpha} (\lambda_1)} = \\
			=& S_{\alpha}(\lambda_1 + 2\lambda_2) \underbrace{S_{\beta}(\lambda_1 + 3\lambda_2) S_{\alpha }(2\lambda_1+3\lambda_2) S_{\beta} (\lambda_1)} S_{\alpha} (\lambda_1 + \lambda_2) S_{\beta}(\lambda_2) = \\
			=& S_{\alpha}(\lambda_1) \underbrace{S_{\alpha}(\lambda_1 + 2\lambda_2) S_{\beta}(2\lambda_1 + 3\lambda_2) S_{\alpha} (\lambda_1 + \lambda_2)} S_{\alpha}(\lambda_1 + 3\lambda_2) S_{\beta}(\lambda_2) = \\
			=& S_{\alpha}(\lambda_1) S_{\beta}(\lambda_1 + \lambda_2)S_{\alpha} (2\lambda_1 + 3\lambda_2 )S_{\beta} (\lambda_1 + 2\lambda_2) S_{\alpha}(\lambda_1 + 3\lambda_2) S_{\beta} (\lambda_2),
		\end{align*}
		where the underbraces indicate to which part we apply the braid relation.

 	\section{Dynamical \texorpdfstring{$K$}{K}-matrix}
	\label{sect::dynamical_k_matrix}
	
	We provide some details concerning how to construct the dynamical $K$-matrix from the boundary fusion operators and the $q=1$ analog of the quantum $K$-matrix. This is a specialization of the approach to quantum dynamical $K$-matrices in \cite[Section 6.4.3]{DeClercq-thesis}. 
 
    \subsection{Motivation from dynamical $R$-matrix}
    
    The usual dynamical fusion operators define a \emph{dynamical $R$-matrix} that solves the \emph{dynamical Yang-Baxter equation} \cite{EtingofVarchenkoExchange}. The corresponding dynamical braiding given by
	\[
	\check{\cR}_{V,W}(\lambda):= J_{WV}(\lambda)^{-1} \circ P \circ J_{VW}(\lambda),
	\]
	where $P:V\otimes W \rightarrow W\otimes V$ is the permutation operator, satisfies the dynamical braid relation
	\begin{gather*}
		\left(\id_W\otimes\check{\cR}_{U,V}(\lambda - h|_W) \right)\circ\left(\check{\cR}_{U,W}(\lambda)\otimes \id_V\right)\circ \left(\id_U\otimes \check{\cR}_{V,W}(\lambda - h|_U)\right)\\
		=\\
		\left(\check{\cR}_{V,W}(\lambda)\otimes \id_U \right)\circ\left(\id_V\otimes\check{\cR}_{U,W}(\lambda - h|_V)\right) \circ \left(\check{\cR}_{U,V}(\lambda)\otimes \id_W\right).
	\end{gather*}
	In this section, we study a boundary analog of this construction giving rise to a \emph{dynamical $K$-matrix}.
	
	
	\subsection{Background on quantum \texorpdfstring{$K$}{K}-matrix}
	In the setting of quantum symmetric pairs, there is a universal $R$-matrix for $U_q(\mathfrak{g})$ and a $1$-tensor universal $K$-matrix for $U_q^{\iota}(\k)\subset U_q(\mathfrak{g})$. These operators \footnote{The operators are modification of the quasi-$R$-matrix\cite{Lusztig} and quasi-$K$-matrix\cite{Bao-Wang-new-approach, BalagovicKolb-qKmatrix}.} are such that\footnote{Ignoring technicalities about completions.}
	\[
	\mathcal{R}_q\in U_q(\g)^{\otimes 2} \quad \text{and} \quad \mathcal{K}_q\in U_q(\mathfrak{g}).
	\]
	They also satisfy the $2$-tensor \emph{reflection equation} \cite[Theorem 3.7]{Kolb-Kmatrixbraided}
	\[
	(\mathcal{K}_q\otimes 1)\circ P\circ \mathcal{R}_q\circ P\circ (1\otimes \mathcal{K}_q)\circ \mathcal{R}_q = P\circ \mathcal{R}_q\circ P \circ (1\otimes \mathcal{K}_q)\circ \mathcal{R}_q \circ (\mathcal{K}_q\otimes 1),
	\]
	where $P$ is the permutation operator. 
	
	\begin{notation}
		If we write $\check{\cR}_q:=P\circ \mathcal{R}_q$, then the $2$-tensor reflection equation becomes
		\[
		(\mathcal{K}_q\otimes 1)\circ \check{\cR}_q\circ (\mathcal{K}_q\otimes 1)\circ \check{\cR}_q = \check{\cR}_q\circ (\mathcal{K}_q\otimes 1)\circ \check{\cR}_q \circ (\mathcal{K}_q\otimes 1).
		\]
	\end{notation}
	
	\begin{remark}
		As explained in \cite[Theorem 3.7]{Kolb-Kmatrixbraided}, the $2$-tensor reflection equation follows from equating two different ways of writing $\Delta(\mathcal{K}_q)$, precisely \cite[equation 3.24]{Kolb-Kmatrixbraided} and \cite[equation 3.25]{Kolb-Kmatrixbraided}.    
	\end{remark}

	There is also a $2$-tensor universal $K$-matrix \cite[equation 1.2]{Kolb-Kmatrixbraided}
	\[
	\kappa_q \in U_q^{\iota}(\k)\otimes U_q(\mathfrak{g}) \quad \text{defined by} \quad \kappa_q = \check{\cR}_q \circ (\mathcal{K}_q\otimes 1) \circ \check{\cR}_q,
	\]
	which satisfies a $3$-tensor reflection equation
	\[
	(\kappa_q\otimes 1)\circ (1\otimes \check{\cR}_q)\circ (\kappa_q\otimes 1)\circ (1\otimes \check{\cR}_q) = (1\otimes \check{\cR}_q) \circ (\kappa_q\otimes 1)\circ (1\otimes \check{\cR}_q)\circ (\kappa_q\otimes 1). 
	\]
	The $1$-tensor universal $K$-matrix is obtained from the two tensor $K$-matrix by applying the counit to the $U_q^{\iota}(\k)$ factor, i.e. $(\epsilon\otimes \id) (\kappa_q) = \mathcal{K}_q\in U_q(\mathfrak{g})$ \cite[Section 2.3]{Kolb-Kmatrixbraided}.
	
	Recall that $-w_0$ permutes the simple roots, inducing an automorphism of the Dynkin diagram. Let $\sigma$ be the order two automorphism\footnote{In general, $\sigma = \theta\theta_0$, where $\theta$ is the diagram automorphism in the data of the Satake diagram and $\theta_0$ is the automorphism of the Dynkin diagram of $\mathfrak{g}$ induced by $-w_0$, in this paper we concentrate on the split case -- so $\theta= \id$ and $\sigma= \theta_0$.} of $U_q(\mathfrak{g})$ such that $\sigma(E_{\alpha}) = E_{-w_0(\alpha)}$ and $\sigma(F_{\alpha}) = F_{-w_0(\alpha)}$ for $\alpha\in R$. The $1$-tensor universal $K$-matrix satisfies the following \cite[equation 1.1]{Kolb-Kmatrixbraided}
	\[
	b\mathcal{K}_q = \mathcal{K}_q\sigma(b), \quad \text{for all} \quad b\in U_q^{\iota}(\k).
	\]
	The $\sigma$-twisted $1$-tensor universal $K$-matrix $\mathcal{K}_q^{\sigma}:=\mathcal{K}_q\sigma$ \cite[equation 3.22]{Kolb-Kmatrixbraided} is such that 
	\[
	b\mathcal{K}_q^{\sigma} = \mathcal{K}_q^{\sigma}b, \quad \text{for all} \quad b\in U_q^{\iota}(\k)
	\]
	\cite[Theorem 3.7]{Kolb-Kmatrixbraided}. Moreover, $\mathcal{K}_q^{\sigma}$ still satisfies the conditions ensuring that representations of $U_q^{\iota}(\k)$ becomes a braided module category over $\sigma$ equivariant $U_q(\mathfrak{g})$ representations \cite[Corollary 3.18]{Kolb-Kmatrixbraided}.
	
	\begin{notation}\label{N:sigma-twist-convention}
		In what follows, we will only consider $\sigma$-twisted $K$-matrices, so will write $\mathcal{K}_q$ in place of $\mathcal{K}_q^{\sigma}$.
	\end{notation}
	
	\subsection{Definition of \texorpdfstring{$q=1$}{q=1} \texorpdfstring{$K$}{K}-matrix}
	
	
	
	
	
	Let $w_0$ denote the longest element in $W$. The operator $-w_0$ permutes simple roots, inducing an automorphism of the Dynkin diagram\footnote{This automorphism may be $\id$, which is the case if $-1\in W\subset GL(\h^*)$.} which we denote by $\sigma$. The Serre-presentation of $\mathfrak{g}$ implies that $\sigma$ determines an associative algebra automorphism of $\U(\mathfrak{g})$, denoted $\mathrm{Ad}_{\sigma}$ such that $\mathrm{Ad}_{\sigma}(x_{\alpha_i}) = x_{\alpha_{\sigma(i)}}$, for $x\in \{f,h,e\}$ and $\alpha_i\in R$. 
	
	\begin{remark}\label{R:w_0-sig=theta}
		As in \cref{remark:simply_connected_group}, let $G$ be the simply connected group with the Lie algebra $\g$. For $g\in G$, we write $\mathrm{Ad}_{g}$ for the automorphism of $\mathfrak{g}$ defined by $x\mapsto gxg^{-1}$. The operators $\widetilde{w}$ can be viewed as elements in $G$ and thus act by conjugation on $\mathfrak{g}$. In particular, one has
		\[
		\widetilde{w_0}e_{\alpha_i}\widetilde{w_0}^{-1} = -f_{\theta(\alpha_i)} \quad \text{and} \quad \widetilde{w_0}f_{\alpha_i}\widetilde{w_0}^{-1} = -e_{\theta(\alpha_i)},
		\]
		c.f. \cite[19.4(a)]{KL-tensor-iii}. 
		It follows that $\mathrm{Ad}_{\widetilde{w_0}}\circ \mathrm{Ad}_{\sigma} = \theta$.
	\end{remark}
	
	\begin{remark}
		Paralleling the relation between $\tilde{s}$ and $\mathrm{Ad}_{\tilde{s}}$, we would like to define an operator $\sigma$, which acts on a finite dimensional $\mathfrak{g}$-module $V$ and has the property: $\mathrm{Ad}_{\sigma}(x)\cdot (\sigma \cdot v) = \sigma\cdot (x\cdot v)$ for all $x\in \mathfrak{g}$ and $v\in V$. However, there is no way to make $\sigma\ne \id$ act in an arbitrary finite dimensional $\mathfrak{g}$-module. In order to fix this, we will work with a slightly different category.
	\end{remark} 
	
	\begin{notation}
		Suppose $\sigma \ne \id$. Consider the algebra $\Usig:=\U(\mathfrak{g})\rtimes \C[\sigma]$,  where $\sigma\cdot u \cdot \sigma^{-1} = \mathrm{Ad}_{\sigma}(u)$, for all $u\in \U(\mathfrak{g})$, and $\sigma^2=\id$. To simplify statements in what follows, we will also write $\Usig = \U(\g)$ when $\sigma = \id$.
	\end{notation}
	
	Since we work over a field where $2\ne 0$, finite dimensional representations of $\Usig$ are completely reducible. If $\sigma \ne \id$, then the irreducibles, up to isomorphism, are: 
	\[
	V(\lambda), \quad V(\lambda)\otimes \mathrm{sgn}, \quad \text{and} \quad V(\mu)\oplus V(\sigma(\mu)),
	\]
	where $\lambda, \mu\in X_+(\mathfrak{g})$ such that $\sigma(\lambda) = \lambda$ and $\sigma(\mu) \ne \mu$ \cite[Lemma 3.4]{Kolb-Kmatrixbraided}. Since $\Usig$ is a Hopf algebra, with $\Delta(\sigma) = \sigma\otimes \sigma$, $S(\sigma) = \sigma^{-1}$, and $\epsilon(\sigma) = 1$, the category of finite dimensional $\Usig$-modules is a symmetric tensor category.
	
	\begin{remark}\label{R:sig-aut-vs-op}
		Now, for a finite dimensional $\Usig$-module $V$, we have an operator $\sigma:V\rightarrow V$ such that $\sigma \circ u \circ \sigma^{-1} = \mathrm{Ad}_{\sigma}(u)$ in $\End(V)$. 
	\end{remark}
	
	\begin{remark}
		The universal $1$-tensor $K$-matrix for a split\footnote{So in the notation of \cite{Kolb-Kmatrixbraided}: $X= \emptyset$ and $T_{w_X}^{-1} = 1$.} quantum symmetric pair is defined in \cite[equation 3.22]{Kolb-Kmatrixbraided} as $\mathcal{K}_q:=\mathfrak{X}\xi T_{w_0}^{-1}\sigma$, see Notation \ref{N:sigma-twist-convention}. The corresponding $2$-tensor $K$-matrix is $\kappa_q=\check{\cR}_q(\mathcal{K}_q\otimes 1)\check{\cR}_q$. It follows from \cite{BalagovicKolb-qKmatrix} that when $q=1$ the operators $\mathfrak{X}$ and $\xi$ act as the identity on finite dimensional representations of $\mathfrak{g}$. 
	\end{remark}
	
	Thus, we arrive at the following definition.
	
	\begin{defn}
		The $q=1$ analog of the $1$-tensor $K$-matrix is defined to be $\mathcal{K}:=\widetilde{w_0}\circ \sigma$, while the $q=1$ analog of the $2$-tensor $K$-matrix is defined as $\kappa:=1\otimes (\widetilde{w_0}\circ \sigma)$.
	\end{defn}
	
	\begin{remark}
		When $q=1$, we have $\mathcal{R}_{q=1}=1\otimes 1$ and $\check{\cR}_{q=1}= P$. Therefore both sides of the $2$-tensor reflection equation become $\mathcal{K}_{q=1}\otimes \mathcal{K}_{q=1}$. Since $\Delta(\widetilde{w_0}) = \widetilde{w_0}\otimes \widetilde{w_0}$ and $\Delta(\sigma) :=\sigma\otimes \sigma$, we have $\Delta(\mathcal{K}) = \mathcal{K}\otimes \mathcal{K}$. This is the $q=1$ analog of \cite[Property (K3')]{BalagovicKolb-qKmatrix}.
	\end{remark}
	
	The following is the $q=1$ analog of \cite[Property (K1')]{BalagovicKolb-qKmatrix}.
	
	\begin{lm}\label{L:Kb=bK}
		If $b\in \U(\k)$, then $b\mathcal{K} = \mathcal{K} b$ in every finite dimensional $\Usig$-module. 
	\end{lm}
	\begin{proof}
		Combining Remark \ref{R:w_0-sig=theta} and Remark \ref{R:sig-aut-vs-op}, we have 
		\[
		\mathcal{K}u\mathcal{K}^{-1} = \mathrm{Ad}_{\widetilde{w_0}}\circ \mathrm{Ad}_{\sigma}(u) = \theta(u),
		\]
		for all $u\in \U(\mathfrak{g})$. The claim follows from noting $\theta(b) = b$, for all $b\in \U(\k)$. 
	\end{proof}
	
	The \emph{3-tensor} reflection equation is satisfied by the $q=1$ analog of the $2$-tensor $K$-matrix.
	
	\begin{prop}\label{P:q=1-reflection}
		Let $X$ be an integrable $\k$-module and let $V$ and $W$ be finite dimensional $\Usig$-modules. Then we have
		\[
		(\kappa\otimes \id_W)(\id_X\otimes P)(\kappa\otimes \id_W)(\id_X\otimes P) = (\id_X\otimes P)(\kappa\otimes \id_V)(\id_X\otimes P)(\kappa\otimes \id_W).
		\]
	\end{prop}
	\begin{proof}
		Since $\kappa = 1\otimes \mathcal{K}$, it is easy to see that both sides of the reflection equation are equal to $\id_X\otimes \mathcal{K}\otimes \mathcal{K}$. 
	\end{proof}
	
	\subsection{Construction of Dynamical \texorpdfstring{$K$}{K}-matrix}

	\begin{defn}
		Define 
		\[
		\check{\cR}_{V,W}(\lambda):=J_{W,V}(\lambda)^{-1}\circ P\circ J_{V,W}(\lambda):V\otimes W\rightarrow W\otimes V
		\]
		and 
		\[
		\kappa_{X,V}(\lambda) := G_{X, V}(\lambda)^{-1}\circ \kappa \circ G_{X, V}(\lambda):X\otimes V\rightarrow X\otimes V.
		\]
	\end{defn}

	\begin{example}
		\[
		\kappa_{\C_x, V}(\lambda) = \begin{bmatrix} 1 & \frac{x}{\lambda+1}\\ 
			0 & 1
		\end{bmatrix}\begin{bmatrix} 0 & -1\\ 
			1 & 0
		\end{bmatrix}\begin{bmatrix} 1 & \frac{-x}{\lambda+1}\\ 
			0 & 1
		\end{bmatrix} = \begin{bmatrix} \frac{x}{\lambda+1} & -1 - \frac{x^2}{(\lambda+1)^2}\\ 
			1 & \frac{-x}{\lambda+1}
		\end{bmatrix}
		\]
	\end{example}
	
	\begin{thm}\label{T:dynamical-reflection-eqn}
		The following \emph{dynamical reflection equation} holds for all integrable $\k$-modules $X$ and finite dimensional $\Usig$-modules $V$ and $W$.
		\[
		(\kappa_{X,V}(\lambda- h|_W)\otimes \id_{W})\circ (\id_X\otimes \check{\cR}_{W,V}(\lambda))\circ (\kappa_{X,W}(\lambda - h|_V)\otimes \id_V)\circ (\id_X\otimes \check{\cR}_{V, W}(\lambda))
		\]
		\[
		=
		\] 
		\[
		(\id_X\otimes \check{\cR}_{W,V}(\lambda))\circ (\kappa_{X,W}(\lambda - h|_V)\otimes \id_V)\circ (\id_X\otimes \check{\cR}_{V,W}(\lambda))\circ (\kappa_{X, V}(\lambda - h|_W)\otimes \id_W)
		\]
	\end{thm}
	\begin{proof}
		To simplify notation, we will leave $\lambda$ dependence implicit, writing $G_{X\otimes V, W}$ instead of $G_{X\otimes V, W}(\lambda)$ etc. By repeatedly using equation \eqref{E:re-write-boundary-dyn-twist} to replace all terms which depend on $\lambda-h|_V$ or $\lambda - h|_W$, one finds that the dynamical reflection equation is equivalent to the equality
		\[
		G_{X,V\otimes W}^{-1}G_{X\otimes V,W}(\kappa\otimes \id_W)G_{X\otimes V,W}^{-1}G_{X,V\otimes W}(\id_X\otimes P)G_{X,W\otimes V}^{-1}G_{X\otimes W, V}(\kappa\otimes \id_W)G_{X\otimes W, V}^{-1}G_{X, W\otimes V}(\id_X\otimes P)
		\]
		\[
		=
		\]
		\[
		(\id_X\otimes P)G_{X, W\otimes V}^{-1}G_{X\otimes W, V}(\kappa\otimes \id_V)G_{X\otimes W, V}^{-1}G_{X, W\otimes V}(\id_X\otimes P)G_{X, V\otimes W}^{-1}G_{X\otimes V, W}(\kappa\otimes \id_W)G_{X\otimes V, W}^{-1}G_{X, V\otimes W}.
		\]
		Using naturality of $G$, one can show the previous equality is equivalent to 
		\[
		G_{X\otimes V,W}(\kappa\otimes \id_W)G_{X\otimes V,W}^{-1}(\id_X\otimes P)G_{X\otimes W, V}(\kappa\otimes \id_W)G_{X\otimes W, V}^{-1}(\id_X\otimes P)
		\]
		\[
		=
		\]
		\[
		(\id_X\otimes P)G_{X\otimes W, V}(\kappa\otimes \id_V)G_{X\otimes W, V}^{-1}(\id_X\otimes P)G_{X\otimes V, W}(\kappa\otimes \id_W)G_{X\otimes V, W}^{-1},
		\]
		which by Lemma \ref{L:Kb=bK} and naturality of $G$ is equivalent to Proposition \ref{P:q=1-reflection}.
	\end{proof}
	
  
		\printbibliography

	\end{document}